\documentclass[10.5pt,reqno]{amsart}
\usepackage{longtable} 
\usepackage{hyperref}
\usepackage[T1]{fontenc}
\usepackage[utf8]{inputenc}
\usepackage[english]{babel} 
\usepackage{textcomp}
\usepackage{dsfont}
\usepackage{latexsym}
\usepackage{amssymb}
\usepackage{amsthm}
\usepackage{amsmath}
\DeclareMathAlphabet{\mathpzc}{OT1}{pzc}{m}{en}
\usepackage{yfonts}
\usepackage{xfrac}
\usepackage{newlfont}
\usepackage{graphicx}
\usepackage{mathtools}
\usepackage{comment}
\usepackage{indentfirst}
\usepackage{braket}
\usepackage{mathrsfs}
\usepackage{xcolor}

\usepackage{etoolbox}

\usepackage{scalerel}[2014/03/10]
\usepackage[usestackEOL]{stackengine}
\newcommand{\dashint}{\,\ThisStyle{\ensurestackMath{%
			\stackinset{c}{.2\LMpt}{c}{.5\LMpt}{\SavedStyle-}{\SavedStyle\phantom{\int}}}%
		\setbox0=\hbox{$\SavedStyle\int\,$}\kern-\wd0}\int}

\newcommand{\Car}{\mathrm{C}}
\newcommand{\RC}{\mathrm{RC}}
\newcommand{\Samp}{\mathrm{S}}

\DeclareMathOperator{\Ad}{Ad}
\DeclareMathOperator{\card}{Card}
\DeclareMathOperator{\sgn}{sgn}

\DeclareMathOperator{\supp}{Supp}

\DeclareMathOperator{\ad}{ad}

\newcommand{\Supp}[1]{\supp\left( #1\right) }

\newcommand{\ee}{\mathrm{e}}

\newcommand{\Ss}{\mathscr{S}}

\newcommand{\loc}{\mathrm{loc}}
\newcommand{\vect}[1]{\mathbf{{#1}}}
\newcommand{\dd}{\mathrm{d}}

\DeclarePairedDelimiter{\abs}{\lvert}{\rvert}

\DeclarePairedDelimiter{\norm}{\lVert}{\rVert}

\let\originalleft\left
\let\originalright\right
\renewcommand{\left}{\mathopen{}\mathclose\bgroup\originalleft}
\renewcommand{\right}{\aftergroup\egroup\originalright}

\newcommand{\grado}{\Df}
\newcommand{\N}{\mathds{N}}
\newcommand{\Z}{\mathds{Z}}
\newcommand{\Q}{\mathds{Q}}

\newcommand{\C}{\mathds{C}}

\newcommand{\R}{\mathds{R}}

\newcommand{\gf}{\mathfrak{g}}

\newcommand{\Df}{\mathfrak{D}}

\newcommand{\Ac}{\mathcal{A}}

\newcommand{\Cc}{\mathcal{C}}
\newcommand{\Dc}{\mathcal{D}}
\newcommand{\Ec}{\mathcal{E}}

\newcommand{\Kc}{\mathcal{K}}
\newcommand{\Lc}{\mathcal{L}}
\renewcommand{\Mc}{\mathcal{M}}
\newcommand{\Nc}{\mathcal{N}}
\newcommand{\Oc}{\mathcal{O}}

\newcommand{\Sc}{\mathcal{S}}

\newcommand{\Sr}{\mathscr{S}}

\newcommand{\meg}{\leqslant}
\newcommand{\Meg}{\geqslant}
\newcommand{\eps}{\varepsilon}
\renewcommand{\phi}{\varphi}
\newcommand{\mi}{\mu}

\thanks{{\em Math Subject Classification 2020}: 46E36, 22E30.}
\thanks{The author is a member of the 	Gruppo Nazionale per l'Analisi
	Matematica, la Probabilit\`a e le	loro Applicazioni (GNAMPA) of
	the Istituto Nazionale di Alta Matematica (INdAM). The author was partially funded by the INdAM-GNAMPA Project CUP\_E5324001950001.
}

\begin{document}
	
	\title[Besov and Triebel--Lizorkin Spaces]{Besov and Triebel--Lizorkin Spaces on Filtered Lie Groups with Polynomial Growth, I: Inclusions, Discretization, Duality}
	
	\author[M.\ Calzi]{Mattia Calzi} 
	\address{Dipartimento di Matematica, Universit\`a degli Studi di
		Milano, Via C. Saldini 50, 20133 Milano, Italy}
	\email{{\tt mattia.calzi@unimi.it}}
	
	\theoremstyle{definition}
	\newtheorem{deff}{Definition}[section]

	\newtheorem{oss}[deff]{Remark}
	
	\newtheorem{ass}[deff]{Assumptions}
	
	\newtheorem{nott}[deff]{Notation}

	\theoremstyle{plain}
	\newtheorem{teo}[deff]{Theorem}
	
	\newtheorem{lem}[deff]{Lemma}
	
	\newtheorem{prop}[deff]{Proposition}
	
	\newtheorem{cor}[deff]{Corollary}
	
	\begin{abstract}
		We continue to develop a theory of  Besov and Triebel--Lizorkin spaces associated with weighted subcoercive operators on a real connected Lie group, specializing to the case of groups with polynomial volume growth. We consider the full scale of spaces and consider equivalent definitions, inclusions, discretization, and duality.
	\end{abstract}
	
	\maketitle
	
	\section{Introduction}

	Besov and Triebel--Lizorkin spaces form a large class of function spaces on the Euclidean spaces which provides a uniform, albeit somewhat technical, way to study simultaneously several classical function spaces, such as Sobolev spaces with integer regularity, fractional Sobolev spaces, both in the version of Sovolev--Slobodeckij--Gagliardo spaces and in the version of Bessel potential spaces, Lipschitz spaces, Hardy and BMO spaces, etc. These spaces have been extensively studied in the classical Euclidean setting (cf., e.g.,~\cite{TriebelFS,TriebelFS2,TriebelFS3}), but have also been extended to more general contexts, such as: open subsets of $\R^n$ (cf., e.g.,~\cite[Chapter 5]{TriebelFS2}); Riemannian manifolds with bounded geometry (cf., e.g.,~\cite[Chapter 7]{TriebelFS2}); Lie groups endowed with a left-invariant Riemannian (cf., e.g.,~\cite[Chapter 7]{TriebelFS2}) or sub-Riemannian metric (cf., e.g.,~\cite{BPV,BPV2,BPV3}); metric spaces endowed with suitable operators resembling a (sub-)Laplacian (cf., e.g.,~\cite{TriebelFS3,Hu}).\footnote{The literature on the subject is quite extensive and the above mentioned reference should only be intended as a very short list of examples, which is by no means complete.} 
	There are nonetheless some contexts where `measuring regularity' in a `Riemannian way,' that is, grouping together all differential operators of the same order, or, more generally, in a `sub-Riemannian way,' appears to be inconvenient since it either clashes with the underlying geometry of the space or with the structure of the differential operator at hand. For instance, let $G$ be a homogeneous group, that is, a simply connected nilpotent Lie group whose Lie algebra $\gf$ has a graduation $(\gf_\lambda)_{\lambda>0}$; in other words, $\gf=\bigoplus_{\lambda>0} \gf_\lambda$ and $[\gf_\lambda,\gf_\nu]\subseteq \gf_{\lambda+\mi}$ for every $\lambda,\mi>0$. Then, $G$ may be identified with $\gf$ by means of the exponential map and $\gf$ may be endowed with a family of automorphic dilations $(\delta_r)_{r>0}$ defined so that $\delta_r(X)=r^\lambda X$ for every $X\in \gf_\lambda$. Operators which are compatible with these dilations (that is, homogeneous operators) are consequently quite natural in this context and one is therefore led to consider (`Goodman type') Sobolev spaces of the form $\Set{f\in L^p(G)\colon\forall \alpha\; (d_\alpha\meg k \implies\vect X^\alpha f\in L^p(G))}$, where $\vect X^\alpha=X_1^{\alpha_1}\cdots X_n^{\alpha_n}$ for some homogeneous basis $(X_1,\dots, X_n)$ of $\gf$, and where $d_\alpha=\sum_j \alpha_j \deg(X_j)$. It turns out, however, that these spaces behave quite weirdly for general $k$ -- for instance, they do not interpolate as one may expect. In fact, a different class of `Bessel potential' Sobolev spaces exhibiting a more natural behviour was introduced in~\cite{FischerRuzhansky}, and was shown to coincide with the previous `Goodman type' Sobolev spaces only for specific values of $k$. It is now worthwhile remarking that, whereas the usual `Bessel potential' Sobolev spaces (and, more generally, several of the Besov and Triebel--Lizorkin spaces briefly mentioned above) are essentially constructed using a (sub-)Laplacian, these `homogeneous' Sobolev spaces were constructed using positive Rockland operators instead, that is, homogeneous and hypoelliptic left-invariant differential operators. A definition using second order differential operator would simply not be possible. 
	Notice, by the way, that replacing a second-order subelliptic differential operator with a higher-order one provides several technical difficulties, since the associated heat kernel cannot be positive, and there is no longer any finite speed property for the corresponding wave propagator -- tools which often lie at the core of several proofs in the literature.
	
	It is therefore natural to wonder whether there is a more general framework which allows to deal at the same time with sub-Laplacians and positive Rockland operators. As a matter of fact, ter Elst and Robinson showed that weighted subcoercive operators provide are quite reasonable and natural choice (cf.~\cite{ElstRobinson}). Indeed, given a group $G$ whose Lie algebra is endowed with a suitable increasing filtration $(\gf_\lambda)_{\lambda>0}$, it is possible to associate a homogeneous group $G_*$ (its `contraction') to $G$ in a natural way, and to associate to every left-invariant differential operator $\Lc$ of degree $d$ some homogeneous left-invariant differential operator $P$ of degree $d$ on $G_*$ (which plays the r\^ole of the `principal part' of $\Lc$). The operator $\Lc$ is then said to be weighted subcoercive if $P+P^*$ is a positive Rockland operator. Notice that this definition mimics closely that of elliptic operators, and that Rockland operators play the r\^ole of homogeneous elliptic operators. Weighted subcoercive operators then enjoy several useful properties. For example, they generate a heat semigroup $(\ee^{-t\Lc})$ whose convolution kernel satisfies suitable Gaussian estimates (even though the exponential decay depends on the degree $d$ and is milder than the classical one). If, in addition, $\Lc$ is formally self-adjoint, then the closure of $\Lc$ on the space of test functions is self-adjoint on $L^2$, hence generates a functional calculus which enjoys particularly interesting properties when $G$ has polynomial growth.
	As shown in~\cite{BCP}, using the heat semigroup associated with a weighted subcoercive operator allows one to define natural Besov and Triebel--Lizorkin  spaces  $B^{p,q}_\alpha$ and $F^{p,q}_\alpha$ on a general connected filtered Lie group. The resulting spaces then do not depend on the chosen operator. In~\cite{BCP,Calzi2,Calzi3} several properties of these spaces were proved, including:
	\begin{itemize}
		\item $B^{p,q}_\alpha$ and $F^{p,q}_\alpha$ are Banach spaces;
		
		\item the space $C^\infty_c(G)$ of test functions is dense in $B^{p,q}_\alpha$ and $F^{p,q}_\alpha$ for $p,q<\infty$;
		
		\item $B^{p',q'}_{-\alpha}$ and $F^{p',q'}_{-\alpha}$ may be canonically identified with the duals of $B^{p,q}_\alpha$ and $F^{p,q}_\alpha$, respectively, when $p,q<\infty$;
		
		\item $B^{p_1,q_1}_{\alpha_1}\subseteq B^{p_2,q_2}_{\alpha_2}$ when $p_1\meg p_2$, $\alpha_2 -Q_*/p_2\meg \alpha_1-Q_*/p_1$, and either $q_1\meg q_2$ or $\alpha_2 -Q_*/p_2< \alpha_1-Q_*/p_1$, where $Q_*$ denotes the homogeneous dimension of $G_*$;\footnote{This and the following fact are actually true when $G$ is endowed with a left Haar measure, and should be slightly modified in the general case.}
		
		\item   $F^{p_1,q_1}_{\alpha_1}\subseteq F^{p_2,q_2}_{\alpha_2}$ when $p_1\meg p_2$, $\alpha_2 -Q_*/p_2\meg \alpha_1-Q_*/p_1$, and either $q_1\meg q_2$ or $p_1<p_2$;
		
		\item $B^{p,q}_\alpha,F^{p,q}_\alpha \subseteq L^p$ when $\alpha>0$, and $F^{p,2}_0=L^p$ for $p\in (1,\infty)$;
		
		\item if $\omega\in \R$ is sufficiently large, then $(\Lc+\omega I)^{-\alpha'}$ induces canonical isomorphisms of $B^{p,q}_\alpha$ and $F^{p,q}_\alpha$ onto $B^{p,q}_{\alpha+\alpha'}$ and $F^{p,q}_{\alpha+\alpha'}$, respectively;
		
		\item if $(X_j)$ is a family of elements of $\gf$ such that the corresponding elements $Y_j$ of $\gf_*$ induce a basis of $\gf_*/[\gf_*,\gf_*]$, and if $\dd$ is the least common multiple of the $d_j=\deg(X_j)$, then $f\in B^{p,q}_\alpha$ (resp.\ $f\in F^{p,q}_\alpha$) if and only if $\ee^{-\Lc}f\in L^p$ and $X_j^{\dd/d_j}\in B^{p,q}_{\alpha-\dd}$ (resp.\ $X_j^{\dd/d_j}\in F^{p,q}_{\alpha-\dd}$) for every $j$. In particular $F^{p,2}_{k\dd}$ coincides with the above-mentioned `Goodman type' Sobolev spaces when $k\in\N$ and $p\in (1,\infty)$;
		
		\item the spaces $B^{p,q}_\alpha$ and the spaces $F^{p,q}_\alpha$ interpolate as the classical ones;
		
		\item $B^{p,q}_\alpha\cap L^\infty$ and $F^{p,q}_\alpha\cap L^\infty$ are algebras under pointwise multiplication for every $\alpha>0$. In particular, $B^{p,q}_\alpha$ and $F^{p,q}_\alpha$ are algebras for every $\alpha>Q_*/p$;
		
		\item the spaces of pointwise multipliers of $B^{p,q}_\alpha$ and $F^{p,q}_\alpha$ may be characterized for $\alpha>Q_*/p$;
		
		\item  the spaces $F^{p,q}_\alpha$  enjoy a localization property as the classical ones;

		\item some Besov and Triebel--Lizorkin spaces may be described in terms of (finite) differences.
	\end{itemize}
	
	The purpose of this paper is to consider the full scale of the Besov and Triebel--Lizorkin spaces (that is, for $p,q\in(0,\infty]$) when $G$ has polynomial volume growth. Indeed, in this case it is possible to effectively use the $L^2$ functional calculus associated to any formally self-adjoint weighted subcoercive operator in order to provide `Littlewood--Paley decompositions' and then obtain descriptions of Besov and Triebel--Lizorkin spaces which, besides the notable absence of the Fourier transform, resemble more closely the `classical' descriptions in the Euclidean setting. We shall, by the way, prove that the more general characterizations with the heat semigroup actually define the same spaces, a fact which was not addressed in~\cite{BCP}, even in the classical Euclidean setting. As a matter of fact, in the classical setting the available proofs (cf., e.g.,~\cite{Triebel2}) only show that the heat semigroup may allow to construct equivalent (quasi-)norms on Besov and Triebel--Lizorkin spaces, but do \emph{not} guarantee that the latter (quasi-)norms, when suitably extended to tempered distributions, are finite exactly on the corresponding Besov and Triebel--Lizorkin spaces. We shall therefore address this little gap in the general theory.
	
	We shall then prove the analogues of several of the above properties of Besov and Triebel--Lizorkin spaces. Notice that the proofs may sometimes be similar in spirit, but require different techniques, since the heat semigroup proves to be quite difficult to handle when $\min(p,q)<1$. In addition to these analogues, we shall also consider the analogue of the discretization procedure devised in~\cite{FrazierJawerth}. Notice that, in this more general context, we cannot hope to get a discretization map which induces isomorphisms of $B^{p,q}_\alpha$ onto the corresponding discretized space $b^{p,q}_\alpha\cong\ell^{p,q}$ and of $F^{p,q}_\alpha$ onto the corresponding discretized space $f^{p,q}_\alpha$. Nonetheless, these discretization techniques still prove quite useful in developing the theory of Besov and Triebel--Lizorkin spaces.
	
	Let us also mention that several additional works (cf.~\cite{Calzi5,CalziRizzo}) on this subject are in preparation. In particular, we plan to compare the spaces $F^{p,2}_0$ with suitably defined local Hardy and bmo spaces, as well as to consider the analogues of other properties of the classical Besov and Triebel--Lizorkin spaces, such as interpolation, algebra properties, and pointwise multipliers. Since the comparison between $F^{p,2}_0$ with local Hardy spaces is most naturally done for distributions with values in a Hilbert space, in this paper we shall  develop a theory of Banach-space valued Besov and Triebel--Lizorkin spaces. Even though we shall state (almost) all technical lemmas in the vector-valued case, we shall actually state the main theory in the   scalar case for simplicity, and only point out the little modifications (mostly in the interpretation of the formulae) that one should implement to the scalar case in order to deal with the vector valued one.
	
	Here is a plan of the paper. In Section~\ref{sec:2}, we collect several preliminary results: after recalling the definition and the main properties of weighted subcoercive operators, we shall add some technical lemmas which allow to control the Schwartz norms of the convolution kernel of suitable functions of $\Lc$; we shall then introduce an appropriate maximal function and discuss some abstract results which will be of great use when developing the aforementioned discretization. In Section~\ref{sec:3}, we shall introduce Besov and Triebel--Lizorkin spaces and discuss some equivalences of the corresponding quasi-norms. We shall then discuss the density of smooth and test functions and the interaction of these spaces with Bessel operators and differential operators. 
	In Section~\ref{sec:4}, we shall provide the sampling and atomic decomposition theorems which are at the core of the discretization of Besov and Triebel--Lizorkin spaces. As a corollary, we shall get the `Sobolev' embeddings between these spaces. The proof of the atomic decomposition theorem, which is partially tangled to the duality results, will only be completed in Section~\ref{sec:5}, which deals with duality.

	\smallskip
	
	We would like to thank professor A.\ Martini for pointing out and discussing his work~\cite{MartiniSfere}, which turned out to be fundamental in the elaboration of this paper and the one to follow.
	
	\section{Preliminaries}\label{sec:2}
	
	\subsection{Relatively Invariant Measures and Convolution}

	Throughout the paper, we shall denote with $G$ a connected (finite-dimensional, real) Lie group  with Lie algebra $\gf$. We shall denote with $\beta$ a left Haar measure and we shall assume that $\beta$ has polynomial volume growth. In other words, we shall assume that there is $Q_G \Meg 0$ such that, for every compact neighbourhood $V$ of the identity $e$ in $G$,  
	\[
	\beta(V^n)\asymp n^{Q_G}
	\]
	for $n\to \infty$. Then, $\beta$ is also a right Haar measure.
		
	We now recall some facts about convolution. Given two convolvable\footnote{We shall not provide a precise definition of this concept, since this would drag us too far away from the main topic. } distributions $f,g$ on $G$, one has
	\[
	\langle f*g, \phi\rangle= \langle f\otimes g, (x,y)\mapsto \phi(xy)\rangle
	\]
	for every $\phi\in C^\infty_c(G)$.
	We shall identify each $f\in L^1_\loc(G)$ with $f\cdot \beta$, that is, the measure with density $f$ with respect to $\beta$. Given two convolvable functions $f,g$ such that $f*g$ is absolutely continuous with respect to $\beta$, we shall generally identify $f*g$ with its density with respect to $\beta$. Thus, under very mild conditions which will be always verified in the applications,
	\begin{equation}\label{eq:1}
	(f* g)(x)=\int_G f(x y^{-1}) g(y)  \,\dd \beta(y)= \int_G f(y) g(y^{-1}x) \,\dd \beta(y).
	\end{equation} 
	
	Similar formulae apply when either $f$ or $g$ is a distribution (and the convolution is still a function).
	Observe that, if $f,g$ are convolvable distributions, $X$ is a left-invariant differential operator, and $Y$ is a right-invariant differential operator, then $Yf $ and $X g$ are convolvable and
	\[
	YX(f*g)=(Yf)*(Xg).
	\]
	
	If $f,g,h$ are distributions then, under some reasonable conditions that will always be satisfied in the applications, 
	\[
	\langle f * g, h \rangle =\langle f, h*  \check g \rangle =\langle g,   \check f * h\rangle,
	\]
	where $\langle \check f,\phi\rangle=\langle f,\check \phi \rangle$ for every $\phi\in C^\infty_c(G)$, and $\check \phi=\phi(\,\cdot\,^{-1})$.  
	
	We now recall Young's inequality in this context. Take $p_1,p_2,p_3\in [1,\infty]$ so that $\frac{1}{p_1'}+\frac{1}{p_2'}=\frac{1}{p_3'}$. Then,
	\[
	\norm{f*g}_{L^{p_3}(G)}\meg \norm{f}_{L^{p_1}(G)}\norm{g}_{L^{p_2}(G)}
	\] 
	for every two positive $\beta$-measurable functions $f,g$ (for positive measurable functions, convolution may be defined by means of~\eqref{eq:1}).  
	
	In order to simplify the notation, given a measure $\mi$ on a measurable space $X$, $p\in (0,\infty]$, and a $\mi$-measurable function $f$ on $X$, we shall also write 
	\[
	\norm{f(x)}_{L^p_x(\mi)} \qquad \text{instead of} \qquad\norm{f}_{L^p(\mi)}.
	\]
	This will be particularly useful when dealing with nested norms.

	\subsection{Differential Operators}
	
	\begin{deff}
		We shall generally identify the (complexification of the) enveloping algebra $U(G)$ of $\gf$ with the algebra of left-invariant differential operators. We shall fix a scalar product on $\gf$, and we shall endow $U(G)$ with the corresponding scalar product, namely the quotient of the natural scalar product on the (complexfication of the) tensor algebra over $\gf$.\footnote{The actual scalar product on $U(G)$ will not matter in the sequel.}
		
		Let $X$ be a left-invariant differential operator. We  denote with $X^R$ the right-invariant differential operator which induces the same point distribution as $X$ at $e$. In other words, $(X f)(e)=(X^R f)(e)$ for every $f\in C^\infty (G)$. We denote with $X^+$ the transpose of $X$ in the enveloping algebra $U(G)$. In other words, the mapping $X\mapsto X^\dag$ is the unique anti-automorphism of $U(G)$ (that is, such that $(XY)^\dag=Y^\dag X^\dag$) which extends the automorphism $X\mapsto -X$ of $\gf$. 
		Then, $X^\dag$ the formal transpose of $X$ (with respect to $\beta$), that is, the unique left-invariant differential operator such that
		\[
		\int_G (X f) g\,\dd \beta=\int_G f X^\dag g\,\dd \beta
		\]
		for every $f,g\in C^\infty_c(G)$. We denote with $X^*$ the formal adjoint of $X$, that is, $\overline X^\dag$. We define the formal transpose and the formal adjoint of right-invariant differential operators in a similar way. If $u$ is a distribution, we then define $X u$ so that
		\[
		\langle X u,\phi\rangle =\langle u, X^\dag \phi\rangle
		\]
		for every $\phi\in C^\infty_c(G)$. In this way, if $f\in C^\infty(G)$, then $(X f)\cdot \beta=X(f\cdot \beta)$.

		In order to simplify the notation, we write $X^{R\dag}$ instead of $(X^R)^\dag$, etc.
	\end{deff}
	 
	Cf.~\cite[Proposition 2.2]{BCP} for a proof of the following result.
	
	\begin{prop}\label{prop:9}
		The following hold:
		\begin{enumerate}
			\item[\textnormal{(1)}]   $X^{R\dag}  =  X^{\dag R} $ for every $X\in U(G)$;
			
			\item[\textnormal{(2)}] $X\delta_e= X^R\delta_e=  (X^\dag \delta_e)\check{\;}$   for every $X\in U(G)$;
			
			\item[\textnormal{(3)}] $X^\dag f=(X^R \check f) \check{\,}$ for every $X\in U(G)$ and for every $f\in C^\infty(G)$.
		\end{enumerate}
	\end{prop}
	
	Notice that, by (2), 
	\[
	\begin{split}
		(Xf)*g=(f*X\delta_e)*g=f*(X\delta_e*g)=f*(X^{R} \delta_e*g)=  f*(X^R g)
	\end{split}
	\]
	under some reasonable conditions on $f$ and $g$ (which are needed to grant associativity of convolution).
	
	In addition, observe that, by our choice of the scalar product on $U(G)$, one has $\abs{X}=\abs{X^+}=\abs{\overline X}$, hence also $\abs{X}=\abs{X^*}$ for every $X\in U(G)$.
	
	\subsection{Filtrations and Weighted Subcoercive Operators}
	
	Throughout the paper, $(\gf_\lambda)_{\lambda\Meg 0}$ will denote an increasing filtration of $\gf$ (that is, $[\gf_\lambda, \gf_\mi]\subseteq \gf_{\lambda+\mi}$ for every $\lambda, \mi\Meg 0$) such that $\gf_\lambda=0$ for every $\lambda<1$, $\bigcup_{\lambda\Meg 0} \gf_\lambda=\gf$, and $\bigcap_{\mi>\lambda} \gf_\mi=\gf_\lambda$ for every $\lambda\Meg 0$.\footnote{We consider the full range $\lambda\Meg 0$  for notational convenience, in analogy with the filtration $(U_\lambda)$, for which $U_\lambda\neq \Set{0}$ for every $\lambda\Meg 0$.}
	
	\emph{In order for a weighted subcoercive operator adapted to the filtration $(\gf_\lambda)$ to exist, we shall assume that the set $\Lambda\coloneqq \Set{\lambda\Meg 1\colon \gf_\lambda \neq \bigcup_{\mi<\lambda} \gf_\mi}$ generate a $\Q$-vector space of dimension $1$.} In other words, setting $\dd_0\coloneqq \min \Lambda$, we require that $\Lambda \subseteq \Q \dd_0$.
	
	For every $X\in \gf$, we define  $\deg X\coloneqq \min\Set{\lambda\Meg0\colon X\in \gf_\lambda}$ and we call $\deg X$ the degree of $X$.
	Define, for every $\lambda>0$, $\gf_{\lambda^-}\coloneqq \bigcup_{\mi<\lambda} \gf_\mi$, $\gf_{*,\lambda}\coloneqq\gf_\lambda/\gf_{\lambda^-} $, and 
	\[
	\gf_*\coloneqq \bigoplus_{\lambda>0} \gf_{*,\lambda}.
	\]
	Define a Lie algebra structure on $\gf_*$ as follows: if $X=\sum_{\lambda>0} (X_\lambda+ \gf_{\lambda^-})$ and $Y= \sum_{\lambda>0} (Y_\lambda+ \gf_{\lambda^-})$ for some $(X_\lambda),(Y_\lambda)\in \prod_{\lambda>0} \gf_\lambda$, then
	\[
	[X,Y]\coloneqq \sum_{\lambda,\mi>0}\left(  [X_{\lambda},Y_{\mi}]+\gf_{(\lambda+\mi)^-}\right) .
	\]
	It is easily seen that $(\gf_{*,\lambda})_{\lambda>0}$ is a graduation of type $((0,+\infty),+)$ of $\gf_*$, that is, $[\gf_{*,\lambda}, \gf_{*,\mi}]\subseteq \gf_{*,\lambda+\mi}$ for every $\lambda,\mi>0$. One may then endow $\gf_*$ with the automorphic dilations $(\delta_r)_{r>0}$ defined so that $\delta_r(X)=r^\lambda X$ for every $X\in \gf_{*,\lambda}$ and for every $\lambda>0$. Thus, $\gf_*$ is the Lie algebra of some homogeneous group $G_*$, with homogeneous dimension $Q_*\coloneqq \sum_{\lambda>0} \dim \gf_{*,\lambda}$.
	
	We observe explicitly that we required $\gf_\lambda=\Set{0}$ for $\lambda<1$ in order for the control modulus $\abs{\,\cdot\,}_*$ (cf.~Definition~\ref{def:3} below) to induce a left-invariant \emph{distance} on $G$ (rather than a quasi-distance). There may be also good reasons to require $\gf_1\neq \Set{0}$. We preferred to avoid imposing this condition in order to keep a natural comparison with graded (or, more generally, homogeneous) groups: if $\gf$ has a graduation $(\tilde \gf_j)_{j\in \Z_+^*}$ (with integer degrees), then it is natural to set $\gf_\lambda=\bigoplus_{j=1}^{[\lambda]} \tilde \gf_j$ for every $\lambda \Meg 0$, but there is no guarantee that $\tilde \gf_1$ should be non-trivial.

	We now extend this filtration to the  enveloping algebra $U(G)$.	
	For every $\lambda\Meg 0$, define $U_\lambda$ as the vector space generated by the products of the form $X_1\cdots X_k$, for $k\Meg 0$, $X_1,\dots, X_k\in \gf$ and $\deg X_1+\cdots+ \deg X_k\meg \lambda$.\footnote{Thus, the identity operator, corresponding to the case $k=0$, belongs to all $U_\lambda$.} 
	Then, $(U_\lambda)$ is an increasing filtration of $U(G)$, that is, $U_\lambda, U_\mi\subseteq U_\lambda U_\mi\subseteq U_{\lambda+\mi}$ for every $\lambda, \mi\Meg 0$.
	For every $X\in U(G)$, we define $\deg X\coloneqq \min\Set{\lambda\Meg 0\colon X\in U_\lambda}$. Notice that $\gf\cap U_\lambda=\gf_\lambda$ for every $\lambda\Meg 0$ as a consequence of Proposition~\ref{prop:8} below (and its proof), so that this definition is consistent with the previous one.
	
	Define $U_{\lambda^-}\coloneqq \bigcup_{\mi<\lambda} U_\mi$ for $\lambda>0$, $U_{0^-}\coloneqq\Set{0}$, $U_{*,\lambda}\coloneqq U_\lambda/U_{\lambda^-}$ for every $\lambda\Meg 0$, and 
	\[
	U_*\coloneqq \bigoplus_{\lambda\Meg 0} U_{*,\lambda} .
	\] 
	We define an algebra structure on $U_*$ as follows:  if $X=\sum_{\lambda\Meg 0} (X_\lambda+ U_{\lambda^-})$ and $Y= \sum_{\lambda\Meg 0} (Y_\lambda+ U_{\lambda^-})$ for some $(X_\lambda),(Y_\lambda)\in \prod_{\lambda\Meg 0} U_\lambda$, then
	\[
	XY\coloneqq \sum_{\lambda,\mi\Meg 0}\left(  X_{\lambda}Y_{\mi}+U_{(\lambda+\mi)^-}\right) .
	\]
	It is easily seen that $U_*$ becomes a graded algebra of type $([0,+\infty),+)$ with this structure.
	
	Cf.~\cite[Proposition 4.2]{BCP} for a proof of the following result.
	
	\begin{prop}\label{prop:8}
		The canonical inclusions $\gf_\lambda \subseteq U_\lambda$, $\lambda> 0$, induce a linear mapping $\pi \colon \gf_*\to U_*$. The canonical extension $U(\pi)\colon U(G_*)\to U_*$ of $\pi$ is an  isomorphism of graded algebras.
	\end{prop}

	From now on, we shall identify $U_*$ and $U(G_*)$ by means of $U(\pi)$.  
	
	\begin{deff}
		We say that a basis $(X_j)_{j\in J}$ of $\gf$ is compatible with the filtration $(\gf_\lambda)$ if $\gf_\lambda$ is generated by the $X_j$ with $\deg(X_j)\meg \lambda$ for every $\lambda \Meg 1$.
	\end{deff}
	
	By~\cite[Proposition 4.3]{BCP}, $(X_j)$ is a basis adapted to the filtration if and only if the $Y_j=X_j+\gf_{\deg(X_j)^-}$, $j\in J$, form a homogeneous basis of $\gf_*$.

	\begin{deff}	
		We say that a family $(X_j)_{j\in J}$ of elements of $\gf$ is a minimal basis if, setting $Y_j\coloneqq X_j+ \gf_{(\deg X_j)^-}$, the family $(Y_j)$ induces a (homogeneous) basis of $\gf_*/[\gf_*,\gf_*]$.
		We denote with $\dd$ the least common multiple of the degrees of the $X_j$, $j\in J$.\footnote{Notice that, since we assumed that $\Lambda \subseteq \Q \dd_0$, the $\deg X_j$ all belong to $\Q \dd_0$, so that their least common multiple is well defined.}
		
		We say that $\Lc \in U(G)$ is weighted subcoercive if $ \Lc+\Lc^*+U_{\grado^-}$ is a positive Rockland operator on $G_*$.
	\end{deff}
	
	Notice that, if $(X_j)$ is a minimal basis of $\gf$, then $(X_j)$ generates $\gf$ as a Lie algebra, and also generates the filtrations $(\gf_\lambda)$ and $(U_\lambda)$ (cf.~\cite[Proposition 4.4]{BCP}). In other words, $\gf_\lambda$  is the vector space generated by the vector fields of the form $\ad(X_{j_1})\cdots \ad(X_{j_{k-1}})X_{j_k}$, where $k\Meg 1$, $j_1,\dots, j_k\in J$, and $\deg(X_{j_1})+\cdots +\deg(X_{j_k})\meg \lambda$. Analogously, $U_\lambda$ is the vector space generated by the differential operators for the form $X_{j_1}\cdots X_{j_k}$, where $k\Meg 0$, $j_1,\dots, j_k\in J$, and $\deg(X_{j_1})+\cdots +\deg(X_{j_k})\meg \lambda$.
	In particular, $(X_j)$ is a `reduced weighted algebraic basis' of $\gf$, in the terminology of~\cite{ElstRobinson}.
	
	Observe that $\dd$ does \emph{not} depend on the choice of $(X_j)$, since it is the least common multiple of the degrees of the non-zero elements of $\gf_*/[\gf_*,\gf_*]$.
	
	Finally, as shown in~\cite[Theorem 4.7]{BCP}, the above   definition of weighted subcoercive operators is consistent with the one given in~\cite{ElstRobinson}. In particular, if $\Lc$ is weighted subcoercive, then $\deg(\Lc)/\dd$ is an even integer. 
	
	We recall the following example from~\cite[Proposition 4.9]{BCP} for the construction of weighted subcoercive operators

	\begin{prop}\label{prop:4}
		Let $(X_j)_{j\in J}$ be a minimal basis of $\gf$. Then, $\Lc\coloneqq \sum_{j\in J} (X_j^{\dd/d_j})^\dag X^{\dd/d_j}_j$, where $d_j=\deg X_j$ for every $j\in J$, is a real,  positive and formally self-adjoint  weighted subcoercive operator.
	\end{prop}
	
	Here, by `positive' we mean that $\int \Lc f \overline f \,\dd \beta\Meg 0$ for every $f\in C^\infty_c(G)$.
	In particular, if $(\gf_\lambda)$ is the filtration generated by $\gf_1$, then $\Lc$ is a sub-Laplacian (a Laplacian, if $\gf_1=\gf$).

	We shall now define a control modulus on $G$. 
	\begin{deff}\label{def:3}
		Given an absolutely continuous curve $\gamma\colon [0,1]\to G$, we  define the content of $\gamma$ as the greatest lower bound of the $\eps>0$ such that 	
		\[
		\abs{P_\lambda \dd L_{\gamma(t)}^{-1}\gamma'(t)} \meg \min(\eps, \eps^{\lambda}) 
		\]
		for almost every $t\in [0,1]$ and for every $\lambda>0$, where $P_\lambda$ is the orthogonal projector of $\gf$ onto $\gf_\lambda \ominus \gf_{\lambda^-}$ (this is non-zero only for finitely many $\lambda>0$) and $L_{\gamma(t)}$ is the left translation by $\gamma(t)$. 
		
		Given $x\in G$, we shall define $\abs{x}_*$  as the greatest lower bound of the contents of the absolutely continuous curves $\gamma\colon [0,1]\to G$ such that $\gamma(0)=e$ and $\gamma(1)=x$.
		
		We  endow $G$ with the left-invariant distance $d(x,y)\coloneqq \abs{y^{-1}x}_*$.
	\end{deff} 
	
	Choosing a different scalar product on $\gf$   gives rise to bi-Lipschitz equivalent control moduli. Equivalence at infinity follows from the fact that all these control moduli are `connected moduli,'\footnote{In fact, every $x\in G$ may be written as $x_1\cdots x_k$, where $\abs{x_j}_*\meg 1$ for $j=1,\dots, k$, and $k\meg \abs{x}_*+1$.}  whereas equivalence near $e$ follows from~\cite[Corollary 6.5]{ElstRobinson}.  
	Notice that here we are not requiring $\gamma$ to be `horizontal.' One may also require $\gamma$ to be horizontal with respect to some fixed weighted algebraic basis compatible with the filtration $(\gf_\lambda)$, in the terminology of~\cite{ElstRobinson}, and still get an equivalent control modulus. This would provide a better mean value theorem for the corresponding `horizontal gradient,' but we shall not need this kind of more precise estimates.

	It is known that 
	\[
	\beta(B(e,r))\asymp r^{Q_*}, \qquad r\to 0^+,
	\]
	while
	\[
	\beta(B(e,r))\asymp r^{Q_G}
	\]
	for $r\to +\infty$.

	\begin{deff}
		From now on, we shall fix a \emph{formally self-adjoint} weighted subcoercive operator $\Lc$ with degree $\grado$. We shall denote with $(h_t)_{t>0}$ the corresponding heat kernel. In other words, $\ee^{-t\Lc}f=f*h_t$ for every $f\in L^2(G)$, where $(\ee^{-t\Lc})_{t>0}$ is the semigroup generated by the closure of $\Lc$, with initial domain $C^\infty_c(G)$, in $L^2(G)$ (cf.~Theorem~\ref{teo:7} below).
		
		For every $\omega\in \R$, we shall set $\Lc_\omega \coloneqq \Lc+ \omega I$.
	\end{deff}
	
	Cf.~\cite[Theorems 4.8 and 5.4]{BCP} for a proof of the following result.
	
	\begin{teo}\label{teo:7}
		There is $\omega\in \R$ such that the following hold:
		\begin{enumerate} 
			\item[\textnormal{(1)}] the closure of $\Lc$, with initial domain $C^\infty_c(G)$,  generates a semigroup of operators of $L^2(G)$; in addition, $\Lc$ is essentially self-adjoint on $C^\infty_c(G)$;
			
			\item[\textnormal{(2)}] for every $\lambda,\lambda'\Meg 0$ there are $b,C>0$ such that  
			\[
			\abs{X Y^R h_t(x)}\meg C \abs{X}\abs{Y} t^{-(Q_* + \deg(X)+\deg(Y))/\grado} \ee^{\omega t} \ee^{- b (\abs{x}_*^{\grado}/t)^{1/(\grado-1)}}
			\]
			for every $X\in U_\lambda$, for every $Y\in U_{\lambda'}$, and for every $x\in G$;
		\end{enumerate}
	\end{teo}

\subsection{`Carleson' and `Reverse Carleson' measures}

\begin{deff}
	We denote with $\Mc_\Car$ (`Carleson measures') the space of positive Radon measures $\mi$ on $(0,+\infty)$ with bounded support such that there is a constant $C>0$ such that
	\[
	\mi((r,2r])\meg C
	\]
	for every $r>0$.
	
	We denote with $\Mc_\RC$  (`reverse Carleson measures') the space of positive Radon measures $\mi$ on $(0,+\infty)$ such that there are constants $\eps,C\in (0,1)$   such that
	\[
	\mi((\eps r, r])\Meg C
	\]
	for every $r\in (0,\eps]$.
	
	We define $\Mc_\Samp\coloneqq \Mc_\Car\cap \Mc_\RC$  (`sampling measures'). 
\end{deff}

\begin{lem}\label{lem:25}
	Take $\eta,\gamma>0$,  $\mi,\nu\in \Mc_\Car$. Then, there is a constant $C>0$ such that
	\[
	\norm*{\int_0^\infty \frac{s^\gamma t^{\eta}}{(s+t)^{\eta+\gamma}} \abs{f(t)}\,\dd \nu(t)  }_{L^q_s(\mi)} \meg C \norm{f}_{L^q(\nu)}
	\]
	for every  $q\in [1,\infty]$ and for every   $\nu$-measurable function $f$.
\end{lem}

\begin{lem}\label{lem:25bis}
	Take $\eta,\gamma>0$ and  let $\mi$ be a Haar measure on $(0,+\infty)$. Take $\nu\in \Mc_\Car$, and assume that either $\eta\Meg 1$ or $\nu\in L^\infty(\mi)\cdot \mi$.	
	Then, there is a constant $C>0$ such that
	\[
	\norm*{\int_0^\infty t^\eta s^\gamma \abs{f(s+t)}\,\dd \mi(t)  }_{L^q_s(\nu)} \meg C \norm{ t^{\gamma+\eta}f(t)}_{L^q_t(\mi)}
	\]
	for every  $q\in [1,\infty]$ and for every  $\mi$-measurable function $f$.
\end{lem}

\subsection{The Schwartz Space}\label{sec:Schwartz}

\begin{deff}\label{def:1bis}
	We define $\Sr(G)$ as the space of $f\in C^\infty(G)$ such that the seminorms $\norm{ (1+ \abs{\,\cdot\,}_*)^k X^R f}_{L^1(G)}$, $k\in \N$, $X\in U(G)$, are finite, endowed with the corresponding topology.
	
	We denote with $\Sr'(G)$ the dual of $\Sr(G)$, endowed with the topology of uniform convergence on the bounded subsets of $\Sr(G)$.
\end{deff}

Notice that we denote the Schwartz space as $\Sr(G)$ instead of $\Sc(G)$ since we wish to avoid any confusion with the \emph{entirely different} `Schwartz space' $\Sc(G)$ which was used in~\cite{BCP, Calzi2, Calzi3}.
 
\begin{oss}\label{oss:2}
	Denoting with $\Ad$ the adjoint representation of $G$ (i.e., $\Ad_x$ is the differential at $e$ of the inner automorphism $y\mapsto x y x^{-1}$), there are $C>0$ and  $k\in\N$ so that $\abs{\Ad_x}\meg C(1+\abs{x}_*)^k$ for every $x\in G$.
\end{oss} 
 
\begin{proof}
	Combining~\cite[Corollary 1.5.12 and the remarks following Definition 1.4.2]{Schweitzer}, we see that there are a Borel-measurable function $\tau\colon G\to [0,+\infty)$, $C'>0$, and $k\in\N$ so that $\abs{\Ad_x}\meg C'(1+\tau(x))^k$ for every $x\in G$ and so that the mapping $(x,y)\mapsto \tau(x^{-1}y)$ is a pseudo-distance on $G$, in the terminology of~\cite{Schweitzer}. By~\cite[Theorem 1.1.21]{Schweitzer}, we may assume that $\tau(x)= \min\Set{n\in \N\colon x\in U^n}$, where $U$ is some compact neighbourhood of $e$ in $G$. Now, take $C''>0$ so that $B(e,C'')\subseteq U$ and take $x\in G$ and a curve  $\gamma\colon [0,1]\to G$   with content $< \abs{x}_*+C''$. Then, setting $k\coloneqq [\abs{x}_*/C'']+2$ and $x_j\coloneqq \gamma(j/k)$ for $j=0,\dots, k$, one has $\abs{x_{j-1}^{-1}x_j }_*< C''$ for every $j=1,\dots, k$, so that $x=\prod_{j=1}^k x_{j-1}^{-1}x_j \in U^k$. Thus, $\tau(x)\meg k \meg \abs{x}_*/C''$, so that the assertion follows by the arbitrariness of $x$.
\end{proof}

\begin{teo}\label{teo:8bis}
	The following hold:
	\begin{enumerate}
		\item[\textnormal{(1)}] $\Sr(G)$ is a nuclear Fréchet space;
		
		\item[\textnormal{(2)}] $\Sr(G)$ is reflexive;
		
		\item[\textnormal{(3)}] the bounded subsets of $\Sr(G)$ are relatively compact;
		
		\item[\textnormal{(4)}] $\Sr(G)$ is a Fréchet $*$-algebra under convolution and under pointwise multiplication;
		
		\item[\textnormal{(5)}] for every $p\in [1,\infty]$, $\Sr(G)$ is the space of $f\in C^\infty(G)$ such that the seminorms  $\norm{(1+ \abs{\,\cdot\,}_*)^k  X Y^R f}_{L^p(G)}$, $k\in \N$, $X,Y\in U(G)$ (resp.\ $\norm{(1+ \abs{\,\cdot\,}_*)^k X   f}_{L^p(G)}$, $k\in \N$, $X\in U(G)$; $\norm{(1+ \abs{\,\cdot\,}_*)^k  X^R f}_{L^p(G)}$, $k\in \N$, $X\in U(G)$), are finite, and has the corresponding topology.
	\end{enumerate} 
\end{teo}

\begin{proof}
	(1) This follows from~\cite[Theorem 6.24]{Schweitzer}.
	
	(2)--(3) These assertions follow from (1) and~\cite[Proposition 50.2, Theorem 36.4, and Corollary 1 to Proposition 33.2]{Treves}.
	
	(4) The first assertion follows from~\cite[Theorem 1.3.13]{Schweitzer} and Remark~\ref{oss:2}. The second assertion follows from the case $p=\infty$ in (5).
	
	(5) By~(4), $\Sr(G)$ is a $*$-algebra under convolution, so that $\Sr(G)$ is the space of $f\in C^\infty(G)$ such that the seminorms  $\norm{(1+ \abs{\,\cdot\,}_*)^k Y f}_{L^1(G)}$, $c>0$, $ Y\in U(G)$, are finite. Since the mapping $f\mapsto X^R f$ is clearly an endomorphism of $\Sr(G)$ for every $X\in U(G)$,   the assertion follows when $p=1$. The assertion for general $p$ follows from~\cite[Theorem 6.24]{Schweitzer} and the above remarks.
\end{proof}

\begin{cor}
	$\Sr'(G)$ is a complete, reflexive, bornological, and nuclear space.  The bounded subsets of $\Sr'(G)$ are relatively compact.
\end{cor}

\begin{proof}
	Combining (1) and (3) of Theorem~\ref{teo:8bis} with~\cite[Corollary to Proposition 39.10 and Proposition 36.10]{Treves}, one sees that $\Sr'(G)$ is reflexive, and that its bounded subsets are relatively compact. Combining (2) of Theorem~\ref{teo:8bis} with~\cite[Proposition 50.6]{Treves} we see that $\Sr'(G)$ is nuclear. Completeness follows from~\cite[Chapter III, Proposition 2 of \S 2 and Corollary 1 to Proposition 12 of \S 3, No.\ 8]{BourbakiTVS}. The fact that $\Sr'(G)$ is bornological follows from~\cite[Corollary to Proposition 4 of Chapter IV, \S 3, No.\ 4]{BourbakiTVS}.
\end{proof}

\begin{deff}
	Define $\Oc'_{C,L}(G)$ and $\Oc'_{C,R}(G)$ as the spaces of $u\in \Dc'(G)$ such that the families $((1+\abs{x}_*)^k L_x u)_{x\in G}$ and $((1+\abs{x}_*)^k R_x u)_{x\in G}$  are bounded in $\Dc'(G)$ (with respect to the topology of uniform convergence on the bounded subset of $\Dc(G)$, or, equivalently, with respect to the weak dual topology), respectively, for every $k\in\N$. Here, $L_x$ and $R_x$ denote the left and right translation operators, so that $\langle L_x u, \phi \rangle= \langle u, \phi(x\,\cdot\,)\rangle$ and $\langle R_x u,\phi\rangle=\langle u, \phi(\,\cdot\,x^{-1})\rangle$ for every $\phi \in C^\infty_c(G)$. We endow $\Oc'_{C,L}(G)$ and $\Oc'_{C,R}(G)$ with the corresponding topology.
\end{deff}

In the following result we collect, without proof, some basic properties of the spaces $\Oc'_{C,L}(G)$ and $\Oc'_{C,R}(G)$. Cf.~\cite{Schwartz,Grothendieck} for more information on these spaces in the case $G=\R^n$.

\begin{prop}
	The following hold:
	\begin{enumerate}
		\item[\textnormal{(1)}] $\Oc'_{C,R}(G)$ and $\Oc'_{C,L}(G)$ are complete nuclear Lusin spaces;
		
		\item[\textnormal{(2)}] every bounded subset of $\Oc'_{C,R}(G)$ and $\Oc'_{C,L}(G)$ is relatively compact,  carries the topology induced by $\Dc'(G)$, and is contained in the closure of some bounded subset of $\Dc(G)$;
		
		\item[\textnormal{(3)}] the mapping $u\mapsto \check u$ induces an isomorphism of $\Oc'_{C,R}(G)$ onto $\Oc'_{C,L}(G)$;
		
		\item[\textnormal{(4)}] convolution induces hypocontinuous  bilinear mappings
		\[
		\begin{aligned}
			&\Sr(G)\times \Oc'_{C,R}(G)\to \Sr(G) & & \Oc'_{C,L}(G)\times \Sr(G)\to \Sr(G)\\
			&\Sr'(G)\times \Oc'_{C,L}(G)\to \Sr'(G) & & \Oc'_{C,R}(G)\times \Sr'(G)\to \Sr'(G)\\
			&\Oc'_{C,R}(G)\times \Oc'_{C,R}(G)\to \Oc'_{C,R}(G) & &\Oc'_{C,L}(G)\times \Oc'_{C,L}(G)\to \Oc'_{C,L}(G)
		\end{aligned}
		\]
		relative to the bounded subsets of each factor;
		
		\item[\textnormal{(5)}] the mappings $\Phi_R\colon \Oc_{C,R}'(G)\to \Lc(\Sr(G))$ and $\Phi_L\colon\Oc_{C,L}'(G)\to \Lc(\Sr(G))$,  defined so that $\Phi_R(u)\phi=\phi*u$ and $\Phi_L(u')\phi=u'*\phi$ for every $\phi\in \Sr(G)$, induce isomorphisms of $\Oc'_{C,R}(G)$ onto the space of left-invariant endomorphisms of $\Sr(G)$ and of  $\Oc'_{C,L}(G)$ onto the space of right-invariant endomorphisms of $\Sr(G)$, respectively;
		
		\item[\textnormal{(6)}] if $(u_1,u_2,u_3)$ belongs to either 
		\[
		\Oc'_{C,R}(G)\times \Oc'_{C,R}(G)\times \Sr'(G)  \qquad \text{or} \qquad \Sr'(G)\times \Oc'_{C,L}(G)\times \Oc'_{C,L}(G),
		\]
		then
		\[
		(u_1*u_2)*u_3=u_1*(u_2*u_3).
		\]
	\end{enumerate}
\end{prop}

\begin{deff}
	Let $Z$ be a Banach space. Then, we set $\Sr'(G;Z)\coloneqq \Lc(\Sr(G);Z)$, endowed with the topology of uniform convergence on the bounded subsets of $\Sr(G)$. In addition, we define, for every $f\in \Sr'(G;Z)$, for every $\phi\in \Oc_{C,L}'(G)$ 
	\[
	f*\phi\colon \Sr(G)\ni \psi\mapsto \langle f, \psi*\check\phi\rangle\in Z,
	\]
	so that $f*\phi\in \Sr'(G;Z)$. We define $\phi'*f$, for $\phi'\in \Oc_{C,R}'(G)$, analogously.
\end{deff}

\begin{prop}
	Let $Z$ be a Banach space, and take $f\in \Sr'(G;Z)$ and $\phi\in \Sr(G)$. Then, $f*\phi, \phi*f\in C^\infty(G;Z)$.
\end{prop}

\begin{proof}
	Define $\psi(x)\coloneqq \langle f, \phi(\,\cdot\,^{-1}x)\rangle$ for every $x\in G$, and observe that, for every $v'\in Z'$, $\langle \psi, v'\rangle=\langle f*\phi,v'\rangle= \langle f,v'\rangle*\phi\in C^\infty(G)$. By the arbitrariness of $v'$, this proves that $\psi\in C^\infty(G)$ (cf.~\cite[Lemma II in the Appendix]{Schwartz2}), and then that $\psi=f*\phi$. The other assertion is proved similarly. Alternatively, one may directly observe that the mapping $G\ni x\mapsto \phi(\,\cdot\,^{-1}x)\in \Sr(G)$ is of class $C^\infty$.
\end{proof}

\subsection{`Goodman' Sobolev Spaces}

\begin{deff}
	Suppose  $\alpha\Meg0$ and $p\in [1,\infty]$. 
	We define the Sobolev space $W^{\alpha,p}(G)$  
	as the space of $f\in L^p(G)$  
	such that $X f\in L^p(G)$  
	for every $X\in U_\alpha$, endowed with the norm 
	\[
	f\mapsto \max_{\substack{X\in U_\alpha, \;  \abs{X}\meg 1}} \norm{X f}_{L^p(G)}.
	\]
	We define $W^{\infty,p}(G)$ as te intersection of all the $W^{\alpha,p}(G)$, endowed with the corresponding topology.
	
	We define $W^{\alpha,p}_\loc(G)$ as the space of $f\in L^p_\loc(G)$ such that $f\phi\in W^{\alpha,p}$ for every $\phi\in C^\infty_c(G)$.
\end{deff}

\subsection{Auxiliary Results}

We shall now consider some results on the functional calculus associated with $\Lc$.
Observe first that, by Theorem~\ref{teo:7}, the operator $\Lc$, with initial domain $C^\infty_c(G)$, is essentially self-adjoint on $L^2(G)$, so that we may consider the associated functional calculus. In particular,  for every Borel function $m$ on $\R$ with polynomial growth, that is, such that $(1+\abs{\,\cdot\,})^{-k} m$ is bounded for some $k\in\N$, there is a unique  distribution $\Kc_\Lc(m)$ on $G$ such that $m(\Lc)\phi=\phi*\Kc_\Lc(m)$ for every $\phi\in C^\infty_c(G)$. In fact, $\Kc_\Lc(m)$ is necessarily a tempered distribution, so that the previous equality may be extended to every $\phi$ in the Schwartz space.

\begin{deff}
	We define $D$ as the maximum of $Q_*$ and $Q_G$.\footnote{Recall that $Q_G$ is defined so that $\beta(B(e,r))\asymp r^{Q_G}$ for $r\to +\infty$, while $Q_*$ is the homogeneous dimension of $G_*$, so that $\beta(B(e,r))\asymp r^{Q_*}$ for $r\to 0^+$}.
\end{deff}

\begin{oss}
	Observe that there is a constant $c>1$ such that
	\[
	\frac 1 c r^{Q_*}(1+r)^{Q_G-Q_*}\meg \beta(B(e,r))\meg c r^{Q_*}(1+r)^{Q_G-Q_*}
	\]
	for every $r>0$.  Consequently,
	\[
	\beta(B(e,t r))\meg c^2 (1+t)^D \beta(B(e,r))  
	\]
	for every $t,r>0$. In particular, $\beta$ is doubling.
\end{oss}

\begin{lem}\label{lem:37b}
	Take $ \rho>0$. Then, there is a constant $C>0$ such that
	\[
	\norm{\chi_{B(e,\rho)\setminus B(e,t)}\abs{\,\cdot\,}_*^{-Q_*/p-\alpha}  }_{L^p(G)}\meg \left( \frac{C}{1-2^{-p\alpha}}\right) ^{1/p} t^{-\alpha}
	\]
	while
	\[
	\norm{\chi_{G\setminus B(e,s)}\abs{\,\cdot\,}_*^{-Q_G/p-\alpha}  }_{L^p(G)}\meg \left( \frac{C}{1-2^{-p\alpha}}\right) ^{1/p} s^{-\alpha}
	\]
	for every $t\in (0,\rho]$, for every $s\in [\rho,+\infty)$, for every $\alpha>0$, and for every $p\in (0,\infty]$.
\end{lem}

\begin{proof}
	The case $p=\infty$ is trivial, so that we may assume that $p\in (0,\infty)$. Take $C_1>0$ such that $   \beta(B(e,r))\meg C_1 r^{Q_*}$ for every $r\in (0,2\rho]$ while $ \beta(B(e,r))\meg C_1 r^{Q_G}$ for every $r\in [2\rho,+\infty)$. Then,  
	\[
	\begin{split}
		\norm{\chi_{B(e,\rho)\setminus B(e,t)}\abs{\,\cdot\,}_*^{-Q_*/p+\alpha} }_{L^p(G)}^p&\meg\sum_{0\meg j<\log_2(\rho/t)} \int_{B(e, t 2^{j+1})\setminus B(e, t 2^{j})}  \abs{x}_*^{-p\alpha-Q_*}\,\dd \beta(x)  \\	
		&\meg C_1 \sum_{0\meg j<\log_2(\rho/t)} (2^{j+1} t)^{Q_*}   ( 2^{j }t)^{-Q_*-p\alpha}\\
		&\meg C_1  2^{Q_* }t^{-p\alpha}  \sum_{j\in\N }  2^{-j p \alpha}\\
		&= C_1  \frac{2^{Q_*}}{1-2^{-p\alpha}}t^{-p\alpha}  
	\end{split}
	\]
	for every $t\in (0,\rho]$. Analogously,
	\[
	\begin{split}
		\norm{\chi_{G\setminus B(e,s)}\abs{\,\cdot\,}_*^{-Q_G/p+\alpha} }_{L^p(G)}^p&\meg\sum_{j\in\N} \int_{B(e, s 2^{j+1})\setminus B(e, s 2^{j})}  \abs{x}_*^{-p\alpha-Q_G}\,\dd \beta(x)  \\	
		&\meg C_1 \sum_{j\in\N} (2^{j+1} s)^{Q_G}   ( 2^{j }s)^{-Q_G-p\alpha}\\
		&= C_1  \frac{2^{Q_G}}{1-2^{-p\alpha}}s^{-p\alpha}  
	\end{split}
	\]
	for every $s\in [\rho,+\infty)$.
	The assertion follows. 
\end{proof}

\begin{lem}\label{lem:37}
	There is a constant $C>0$ such that
	\[
	\norm{(1+\abs{\,\cdot\,}_*/t)^{-\alpha-D/p}  }_{L^p(G)}\meg 3 \left( \frac{C}{1-2^{-p\alpha}}\right) ^{1/p} t^{Q_*/p}\max(1,t)^{(Q_G-Q*)/p}
	\]
	for every $p\in (0,\infty]$ and for every $t,\alpha>0$.
\end{lem}

\begin{proof}
	The case $p=\infty$ is trivial, so that we may assume that $p\in [p_0,\infty)$. By Lemma~\ref{lem:37b}, we see that there is a constant $C_1>0$ such that
	\[
	\norm{\chi_{B(e,1)\setminus B(e,t)}\abs{\,\cdot\,}_*^{-Q_*/p-\alpha}  }_{L^p(G)}\meg \left( \frac{C_1}{1-2^{-p\alpha}}\right) ^{1/p} t^{-\alpha}
	\]
	while
	\[
	\norm{\chi_{G\setminus B(e,s)}\abs{\,\cdot\,}_*^{-Q_G/p-\alpha}  }_{L^p(G)}\meg \left( \frac{C_1}{1-2^{-p\alpha}}\right) ^{1/p} s^{-\alpha}
	\]
	for every $t\in (0,1]$, for every $s\in [1,+\infty)$, for every $\alpha>0$, and for every $p\in (0,\infty]$.
	In addition, we may find a constant $C_2>0$ such that $  \beta(B(e,r))\meg C_2 r^{Q_*}$ for every $r\in (0,1]$ while $  \beta(B(e,r))\meg C_2 r^{Q_G}$ for every $r\in [1,+\infty)$. 
	 If $t\in (0,1]$, then
	\[
	\begin{split}
		\norm{(1+\abs{\,\cdot\,}_*/t)^{-\alpha-D/p}  }_{L^p(G)}&\meg \beta(B(e,t))^{1/p}  +t^{\alpha+D/p} \norm{\chi_{B(e,1)\setminus B(e,t)}\abs{\,\cdot\,}_*^{-D/p-\alpha}  }_{L^p(G)}\\
			&\qquad+t^{\alpha+D/p}\norm{\chi_{G\setminus B(e,1)}\abs{\,\cdot\,}_*^{-D/p-\alpha}  }_{L^p(G)}\\
			&\meg C_2^{1/p}  t^{Q_*/p} + \left( \frac{C_1}{1-2^{-p\alpha}}\right) ^{1/p} t^{Q_*/p}+ \left( \frac{C_1}{1-2^{-p\alpha}}\right) ^{1/p} t^{\alpha+D/p}\\
			&\meg 3 \left( \frac{\max(C_1,C_2)}{1-2^{-p\alpha}}\right) ^{1/p} t^{Q_*/p}.
	\end{split}
	\]
	If, otherwise, $t>1$, then
	\[
	\begin{split}
		\norm{(1+\abs{\,\cdot\,}_*/t)^{-\alpha-D/p}  }_{L^p(G)}&\meg \beta(B(e,t))^{1/p}   +t^{\alpha+D/p}\norm{\chi_{G\setminus B(e,t)}\abs{\,\cdot\,}_*^{-D/p-\alpha}  }_{L^p(G)}\\
			&\meg C_2^{1/p} t^{Q_G/p} +  \left( \frac{C_1}{1-2^{-p\alpha}}\right) ^{1/p}  t^{Q_G/p}\\
			&\meg 2 \left( \frac{\max(C_1,C_2)}{1-2^{-p\alpha}}\right) ^{1/p}  t^{Q_G/p},
	\end{split}
	\]
	whence the conclusion.
\end{proof}

\begin{prop}\label{prop:22}
	There is $\omega_0\in \R$ such that the following hold.
	For every $\omega \Meg \omega_0$, for every $p\in [1,\infty]$, and for every $c,s\Meg 0$ such that $s>c+D(1/p-1/2)_+$ there is a constant $C>0$ such that, for every $m\in B^{\infty,\infty}_s(\R)$ supported in $[-1,1]$ and for every $r>0$,
	\[
	\beta(B(e,r^{1/\grado}))^{1/p'}\norm{\Kc_{r\Lc_\omega} (m ) (1+\abs{\,\cdot\,}_*/r^{1/\grado})^c}_{L^p(G)} \meg C\norm{m}_{B^{\infty,\infty}_s(\R)}.
	\]
	In addition, $\Kc_\Lc$ induces a continuous linear mapping of $\Sr(\R)$ into $\Sr(G)$.
\end{prop}

Here, we write $\Lc_\omega$ instead of $\Lc+ \omega I$, so that $\Kc_{r\Lc_\omega} (m )= \Kc_\Lc(m(r\,\cdot\,+r\omega))$.

\begin{proof}
	Choose $\omega_0>0$ so that $\Lc_{\omega_0}$ is positive and so that there are $b,C_1>0$ such that $\abs{h_t}\meg C_1 \ee^{(\omega_0/2) t} p_{b,t}$ for every $t>0$, where $(h_t)_{t>0}$ denotes the heat kernel associated with $\Lc$ (cf.~Theorem~\ref{teo:7}). 
	We wish to apply~\cite[Theorem 6.1]{MartiniSfere} to the functional calculus associated with $\Lc_\omega$, $\omega \Meg \omega_0$.\footnote{Notice that in the cited reference the fractional Sobolev space $L^\infty_s(\R)$ is used instead of $B^{\infty,\infty}_s(\R)$. Since the condition on $s$ is `open' and since all reasonable fractional Sobolev spaces of type $L^\infty_s$ embed into one another up to (arbitrarily small) losses in regularity, we may use the spaces $B^{\infty,\infty}_s(\R)$.  } To this aim, with the notation of~\cite{MartiniSfere}, we choose $\Phi\colon \R\to \R$, $\Phi(\lambda)=\ee^{-\lambda}$ for every $\lambda\in \R$, $\Omega=(0,+\infty)$ and $\Psi\colon \Omega\to \R$ with $\Psi(\lambda)=-\log(\lambda)$ for every $\lambda>0$, and $\eps_r(\lambda)= r^{\grado} \lambda$ for every $r>0$ and for every $\lambda\in \R$. We therefore need to show that
	\[
	\sup_{r>0} \beta(B(e,r))^{1/2}\ee^{-\omega r^{\grado}}\norm*{ h_{r^{\grado}}(1+\abs{\,\cdot\,}_*/r)^a }_{L^2(G)} 
	\]
	is finite for every $a\Meg 0$.  This follows from Lemma~\ref{lem:37}, the Gaussian estimates (cf.~Theorem~\ref{teo:7}) and our choice of $\omega_0$.
\end{proof}

\begin{cor}\label{cor:12} 
	Take $p_0\in (0,1]$, $\lambda,c\Meg 0$ and $\rho,\eps>0$. Then, there is a constant $C>0$ such that for every $p\in [p_0,\infty]$, for every $s\Meg \eps+ c+D(1/p-1/2)_+$, for every $m\in B^{\infty,\infty}_s(\R)$ supported in $[-\rho,\rho]$, for every $r\in (0,\rho]$, and for every $X,Y\in U_\lambda$,
	\[
	\norm{XY^R\Kc_{r^{\grado}\Lc } (m ) (1+\abs{\,\cdot\,}_*/r)^c}_{L^p(G)} \meg C r^{-\deg(X)-\deg(Y)+(1/p-1)Q_*}\abs{X}\abs{Y}\norm{m}_{B^{\infty,\infty}_s(\R)}.
	\]
\end{cor}

\begin{proof}
	\textsc{Step I}  Take $c'\Meg \eps+c  +D(1/p_0-1/2)$. By Proposition~\ref{prop:22}, there are $\omega\in \R$ and $C_1>0$ such that 
	\[
	\norm{ \Kc_{r^{\grado}\Lc_\omega } (\phi ) (1+\abs{\,\cdot\,}_*/r)^{\bar c}}_{L^p(G)} \meg C_1 \beta(B(e,r))^{-1/p'}\norm{\phi}_{B^{\infty,\infty}_{\eps+\bar c}(\R)}
	\]
	for every $p\in \Set{2,\infty}$, for every $\bar c\in \Set{0,c'}$, for every $r>0$, and for every $\phi \in B^{\infty,\infty}_{\eps+\bar c}(\R)$ supported in $[-1,1]$. By H\"older's inequality, the same holds for every $p\in [2,\infty]$, while by interpolation the same holds for every $\bar c\in [0,c']$ (cf.~\cite[Theorem 6.4.5]{BerghLofstrom}). 
	In addition, observe that there is a constant $C_2\Meg 1$ such that 
	\[
	\frac{1}{C_2} r^{Q_*}(1+r)^{Q_G-Q_*}\meg \beta(B(e,r))\meg  C_2  r^{Q_*}(1+r)^{Q_G-Q_*}
	\]
	for every $r>0$.  For every $r\in (0,\rho]$, define an endomorphism $T_r$ of $L^\infty(\R)$ so that    $(T_r\phi)(\lambda) =\phi((\rho +\rho^{\grado}\abs{\omega}) \lambda- r^{\grado}\omega)$ for every $\lambda\in \R$, and observe that there is a constant $C_3>0$ such that $\norm{T_r \phi}_{B^{\infty,\infty}_s(\R)}\meg C_3 \norm{\phi}_{B^{\infty,\infty}_s(\R)}$ for every $\phi \in B^{\infty,\infty}_s(\R)$, for every $r\in (0,\rho]$, and for every $s\in (0, \eps+c']$. In addition, if $\phi$ is supported in $[-\rho,\rho]$, then $T_r\phi$ is supported in $[-1,1]$, and $\phi(r^{\grado} \lambda)=(T_r\phi)(r^{\grado}(\rho+\rho^{\grado}\abs{\omega})^{-1} (\lambda+ \omega) )$ for every $r>0$ and for every $\lambda\in \R$. Consequently, assuming that $\eta\coloneqq \rho+\rho^{\grado}\abs{\omega}\Meg 1$ (as we may),
	\[
	\begin{split}
		&\norm{ \Kc_{r^{\grado}\Lc } (m ) (1+\abs{\,\cdot\,}_*/r)^{\bar c}}_{L^p(G)} = \norm{ \Kc_{(r\eta^{-1/\grado})^{\grado}\Lc_\omega } ( T_r m ) (1+\abs{\,\cdot\,}_*/r)^{\bar c}}_{L^p(G)}\\
			&\qquad\meg\norm{ \Kc_{(r\eta^{-1/\grado})^{\grado}\Lc_\omega } ( T_r m ) (1+\abs{\,\cdot\,}_*/(r\eta^{-1/\grado}))^{\bar c}}_{L^p(G)}\\
			&\qquad\meg C_1 \beta(B(e,r\eta^{-1/\grado}))^{-1/p'}\norm{T_r m}_{B^{\infty,\infty}_{\eps+\bar c}(\R)}\\
			&\qquad\meg  C_1 C_2^{1/p'} C_3 \eta^{Q_*/(p'\grado)} r^{-Q_*/p'} (1+r\eta^{-1/\grado})^{-(Q_G-Q_*)/p'} \norm{m}_{B^{\infty,\infty}_{\eps+\bar c}(\R)}\\
			&\qquad\meg  C_1  C_2 C_3 (1+\rho)^{(Q_*-Q_G)_+}\eta^{Q_*/\grado} r^{-Q_*/ p'} \norm{m}_{B^{\infty,\infty}_{\eps+\bar c}(\R)}
	\end{split}
	\]
	for every $r\in (0,\rho]$, for every $p\in [2,\infty]$, for every $m\in B^{\infty,\infty}_{\eps+\bar c }(\R)$ supported in $[-\rho,\rho]$, and for every $\bar c\in [0,c']$. 
	
	\textsc{Step II} Observe that there is a constant $C_4>0$ such that $\norm{ \ee^{2\,\cdot\,} m}_{B^{\infty,\infty}_{\eps+\bar c}(\R)} \meg C_4\norm{   m}_{B^{\infty,\infty}_{\eps+\bar c}(\R)}$ for every $\bar c\in [0,c']$ and for every $m\in B^{\infty,\infty}_{\eps+\bar c}(\R)$ supported in $[-\rho,\rho]$. In addition,  $\Kc_{\Lc }(m(r^{\grado}\,\cdot\,))=h_{r^{\grado}}*\Kc_\Lc((\ee^{2\,\cdot\,}m)(r^{\grado}\,\cdot\,))* h_{r^{\grado}}$, where $(h_t)$ denotes the heat kernel associated with $\Lc$.  Consequently, by Young's inequality,
	\[
	\begin{split}
	\norm{XY^R\Kc_{r^{\grado}\Lc } (m ) (1+\abs{\,\cdot\,}_*/r )^{\bar c}}_{L^p(G)}&\meg\norm{ Y^R h_{r^{\grado}} (1+\abs{\,\cdot\,}_*/r )^{\bar c}  }_{L^1(G)}\norm{\Kc_{r^{\grado}\Lc } ( \ee^{2\,\cdot\,} m  ) (1+\abs{\,\cdot\,}_*/r )^{\bar c}}_{L^p(G)}\\
		&\qquad\times\norm{ X h_{r^{\grado}} (1+\abs{\,\cdot\,}_*/r )^{\bar c}  }_{L^1(G)}
	\end{split}
	\]
	for every $r\in (0,\rho]$ and for every $X,Y\in U_\lambda$. 
	Using Lemma~\ref{lem:37} and Theorem~\ref{teo:7}, we  see that there is a constant $C_4>0$ such that
	\[
	\norm{ XY^R h_{r^{\grado}} (1+\abs{\,\cdot\,}_*/r )^{\bar c}  }_{L^1(G)}\meg C_4 \abs{X}\abs{Y} r^{-\deg(X)-\deg(Y)}
	\]
	for every $r\in (0,\rho]$, and for every $X,Y\in U_\lambda$. Therefore, by means of~\textsc{step I}, we   see that there is a constant $C_5>0$ such that
	\[
	\norm{XY^R\Kc_{r^{\grado}\Lc } (m ) (1+\abs{\,\cdot\,}_*/r)^{\bar c}}_{L^p(G)} \meg C_4^2C_5 r^{-\deg(X)-\deg(Y)-Q_*/p'}\abs{X}\abs{Y}\norm{m}_{B^{\infty,\infty}_{\eps+\bar c}(\R)}.
	\]
	for every $r\in (0,\rho]$, for every $p\in [2,\infty]$, and for every $\bar c\in [0,c']$.
	
	\textsc{Step III} Now, take $p\in [p_0,2]$  
	and observe that
	\[
	\begin{split}
		\norm{XY^R\Kc_{r^{\grado}\Lc } (m ) (1+\abs{\,\cdot\,}_*/r)^{  c}}_{L^p(G)} &\meg \norm{XY^R\Kc_{r^{\grado}\Lc } (m ) (1+\abs{\,\cdot\,}_*/r)^{  c_p}}_{L^2(G)} \norm{(1+\abs{\,\cdot\,}_*/r)^{  c-c_p}}_{L^{2p/(2-p)}(\R)}
	\end{split}
	\]
	where $c_p\coloneqq \eps+c+D(1/p-1/2)$. 
	Therefore,  by means of~\textsc{step II} and Lemma~\ref{lem:37}, we   infer that there is a constant  $C_6>0 $ such that
	\[
	\norm{XY^R\Kc_{r^{\grado}\Lc } (m ) (1+\abs{\,\cdot\,}_*/r)^{  c}}_{L^p(G)}\meg C_6 r^{Q_*(1/p-1/2)} r^{-\deg(X)-\deg(Y)-Q_*/2}\abs{X}\abs{Y}\norm{m}_{B^{\infty,\infty}_{\eps+c_p}(\R)},
	\]
	whence the result.
\end{proof}

 \begin{deff}
 	For every $p\in (0,+\infty]$, for every $r>0$, and for every $\beta$-measurable function $f$ from $G$ into a Banach space $Z$, we define 
 	\[
 	(\Nc_{p,a,r} f)(x)\coloneqq r^{-Q_*/p}\norm{ f(y) (1+d(x,y)/r)^{-a}  }_{L^p_y(\beta;Z)}=r^{-Q_*/p}\norm{ \abs{f(y)} (1+d(x,y)/r)^{-a}  }_{L^p_y(G)}
 	\]
 	for every $x\in G$. We define also $(\Nc_{p,a}  f)(x)\coloneqq \sup_{r>0} (1+r)^{(Q_*-D)/p} (\Nc_{p,a,r} f)(x)$.

 	In addition, we define
 	\[
 	(\Mc_p f)(x)\coloneqq \sup_{r>0} \left(\dashint_{B(x,r)}\abs{f}^p\,\dd \beta\right)^{1/p}
 	\]
 	for every $\beta$-measurable function $f\colon G\to Z$ and for every $p\in (0,+\infty)$.
 \end{deff}

\begin{lem}\label{lem:39}
	Take  $a>D $. Then, there is a constant $C>0$ such that
	\[
	\Nc_{p,a/p} f\meg C^{1/p} \Mc_p f
	\]
	for every Banach space $Z$, for every $\beta$-measurable function $f\colon G\to Z$ and for every $p\in (0,+\infty)$.
\end{lem}

\begin{proof}
	Take $C_1>0$ so that $\frac{1}{C_1} r^{Q_*}(1+r)^{Q_G-Q_*}\meg \beta(B(e,r))\meg C_1 r^{Q_*}(1+r)^{Q_G-Q_*}$ for every $r>0$. Then,
	\[
	\begin{split}
	(\Nc_{p,a/p,r} f)(x)^p&= r^{-Q_*} \int_G \frac{\abs{f(y)}^p}{(1+d(x,y)/r)^{a }}\,\dd \beta(y)\\
		&\meg  r^{-Q_*}\sum_{j\in \Z} \int_{B(x, 2^{j+1} r)\setminus B(x, 2^j r)} \frac{\abs{f(y)}^p}{(1+2^j)^{a}}\,\dd \beta(y)\\
		&\meg  r^{-Q_*}\sum_{j\in \Z}  \frac{C_1 (2^{j+1}r)^{Q_*}(1+2^{j+1}r)^{Q_G-Q_*}}{(1+2^j)^{a}} \dashint_{B(x, 2^{j+1}r)} \abs{f}^p\,\dd \beta\\
		&\meg (\Mc_p f)(x)^p (1+r)^{D-Q_*}2^{D}  C_1\sum_{j\in \Z}  2^{j Q_*}(1+2^{j})^{D-Q_*-a}  
	\end{split}
	\]
	for every $x\in G$ and for every $r>0$, so that the assertion follows.
\end{proof}
 
 The next result may seem a weaker version of Lemma~\ref{lem:39} (combined with~\cite{Singular}), but will allow us to prove several results in a more uniform way (since the analogous technical result for the spaces $\Cc^q$ does not seem to hold) and also to deal with duality more effectively.
 
 \begin{cor}\label{cor:18}
 	Take $p_0\in (0,\infty)$, $p\in (p_0,\infty)$, $q\in [p_0,\infty]$, $a>D/p_0$, $\mi$ a positive  Radon measure on $(0,+\infty)$ with bounded support,  and $\kappa\colon (0,+\infty )\to (0,+\infty)$ a $\mi$-measurable function. Then, there is a constant $C>0$ such that
 	\[
 	\norm{ (1+\kappa(t))^{(Q_*-D)/p_0} \Nc_{p_0,a,\kappa(t)}(f(t,\,\cdot\,))(x)  }_{L^{q,p}_{t,x}(\mi,\beta)}\meg C\norm{f}_{L^{q,p}(\mi,\beta)}
 	\]
 	for every Banach space $Z$ and for every $(\mi\otimes \beta)$-measurable function $f\colon (0,+\infty)\times G\to Z$.
 \end{cor}
 
 \begin{proof}
 	Observe first that, replacing $f$ with $\abs{f}$, we may assume that $Z=\C$. Then, arguing as in the proof of~\cite[Lemma 5.11]{BCP}, one may prove that the mapping $(t,x)\mapsto (1+\kappa(t))^{(Q_*-D)/p_0} \Nc_{p_0,a,\kappa(t)}(f(t,\,\cdot\,))(x) $ is $(\mi\otimes \beta)$-measurable. 
 	If $q>p_0$, then the assertion follows from Lemma~\ref{lem:39} and~\cite{Singular}. It remains to deal with the case $q=p_0$. In this case, observe that, by Tonelli's theorem,
 	\[
 	\begin{split}
 		&\norm{(1+\kappa(t))^{(Q_*-D)/p_0} \Nc_{p_0,a,\kappa(t)}(f(t,\,\cdot\,))(x)   }_{L^{p_0,p}_{t,x}(\mi,\beta)}^{p_0} =\norm{(1+\kappa(t))^{Q_*-D}[ \Nc_{1,a p_0,\kappa(t)}\abs{f(t,\,\cdot\,)}^{p_0}](x)   }_{L^{1,p/p_0}_{t,x}(\mi,\beta)}\\
			&\qquad=\sup_{g\in U}\int_G \int_0^{+\infty}\int_G(1+\kappa(t))^{Q_*-D} \kappa(t)^{-Q_*} \abs{f(t,y)}^{p_0} (1+d(x,y)/\kappa(t))^{-a p_0}\dd \beta(y)\,\dd \mi(t) \abs{g(x)}\,\dd \beta(x)\\
			&\qquad = \sup_{g\in U} \int_G \int_0^{+\infty}\abs{f(t,y)}^{p_0}(1+\kappa(t))^{Q_*-D}(\Nc_{1, a p_0,\kappa(t)} g)(y)\,\dd \mi(t)  \,\dd \beta(y)\\
			&\qquad \meg \sup_{g\in U}\int_G \int_0^{+\infty}\abs{f(t,y)}^{p_0} \,\dd \mi(t)  (\Nc_{1, a p_0 } g)(y)\,\dd \beta(y)\\
			&\qquad \meg \norm{ \abs{f}^{p_0}  }_{L^{1,p/p_0}(\mi,\beta)} \sup_{g\in U} \norm{\Nc_{1, a p_0 } g  }_{L^{(p/p_0)'}(G)}\\
			&\qquad = \norm{ f}_{L^{p_0,p}(\mi,\beta)}^{p_0} \sup_{g\in U} \norm{\Nc_{1, a p_0 } g  }_{L^{(p/p_0)'}(G)}
 	\end{split}
 	\]
 	where $U$ denotes the unit ball of $L^{(p/p_0)'}(G)$. The assertion then follows from  Lemma~\ref{lem:39} and~\cite{Singular}, since $(p/p_0)'>1$.
 \end{proof}

\begin{lem}\label{lem:33}
	Take $m\in C^\infty_c(\R)$, $p_0\in (0,1]$, $a_0>0$, and $\rho\Meg 0$. Then, there  is a constant $C>0$ such that
	\[
	\Nc_{\infty,a,r}(X f)\meg C \abs{X} r^{-\deg(X) }  \Nc_{p,a,r} f 
	\]
	for every $p\in [p_0,+\infty)$,   for every $r\in (0,\rho]$, for every $a\in (0,a_0]$, for every   Banach space $Z$,  and for every $f\in \Sr'(G;Z)$ such that $f=f*\Kc_{r^{\grado}\Lc}(m )$. 
\end{lem}

Notice that, if  $f=f*\Kc_{r^{\grado}\Lc}(m )$, then $f$ may be identified with a function of class $C^\infty$.  We still prefer to start with the assumption that $f\in \Sr'(G;F)$ since we wish to avoid the ambiguity of the notation $f\in \Sr'(G;F)\cap C^\infty(G;F)$: namely, a tempered distribution may be representable as a smooth function even though the function itself is \emph{not} `tempered;' in other words, a smooth function may actually represent a tempered distribution of higher order (for example, $\cos(\ee^x)\ee^x$ is a smooth function on $\R$ and equals the derivative of $\sin(\ee^x)$, hence may be identified with a tempered distribution). This would create unfortunate ambiguities in the definition of convolution. Notice, though, that  in our situation $f$ is actually `tempered,' that is, has polynomial growth. Another possible route to avoid this ambiguity  would be to observe that $\Sr'(G)*\Sr(G)\subseteq \Oc_{C,L}(G)$, where $\Oc_{C,L}(G)$ denotes the dual of $\Oc'_{C,L}(G)$, which may be characterized as the space of functions with `very slow growth;' a similar result holds for spaces of vector-valued distributions. We preferred to avoid introducing an additional space only to address this issue.

We proceed as in~\cite[Lemma 3.2]{Besov} (cf.~also~\cite{BesovCorr}).

\begin{proof}
	\textsc{Step I}  By Corollary~\ref{cor:12} we may find a constant $C_1>0$ such that 
	\[
	\norm{X \Kc_{r^{\grado}\Lc }(m ) (1+\abs{\,\cdot\,}_*/r)^{a}}_{L^1(G)}\meg C_1\abs{X} r^{-\deg(X) } 
	\]
	for every $r\in (0,\rho]$, for every $a\in (0,a_0]$, and for every $X\in U_\lambda$.  
	Consequently, if we take $f\in \Sr'(G;Z)$ so that $f=f*\Kc_{r^{\grado}\Lc}(m )$ (which implies that $f\in C^\infty(G;Z)$), then 
	\[
	\begin{split}
		\abs{(X f)(y)}&\meg \int_G \abs{f(z)} \abs{X \Kc_{r^{\grado}\Lc}(m)(z^{-1}y)}\,\dd \beta(z)\\
		&\meg (\Nc_{\infty,a,r} f)(x) \int_G (1+ d(x,z)/r)^{a} \abs{X \Kc_{r^{\grado}\Lc}(m)(z^{-1}y)}\,\dd \beta(z)\\
		&\meg (\Nc_{\infty,a,r} f)(x)     (1+ d(x,y)/r)^{a}  \int_G (1+ \abs{z^{-1}y}/r)^{a} \abs{X\Kc_{r^{\grado}\Lc}(m)(z^{-1}y)}\,\dd \beta(z)\\
		&\meg (\Nc_{\infty,a,r} f)(x)   (1+ d(x,y)/r)^{a} C_1  \abs{X}   r^{-\deg(X) }
	\end{split}
	\]
	for every $x,y\in G$, so that
	\[
	\Nc_{\infty,a,r}(Xf) \meg  C_1 \abs{X} r^{-\deg(X) } \Nc_{\infty,a,r} f
	\]
	for every $r\in (0,\rho]$, for every $p\in [p_0,\infty)$, for every $a\in (0,a_0]$, and for every $X\in U_\lambda$.
	
	\textsc{Step II} Now, take $f\in \Sr'(G;Z)$ so that $f=f*\Kc_{r^{\grado}\Lc}(m)$, and assume that $(\Nc_{\infty,a,r} f)(e)$ is  finite for some $p\in [p_0,\infty)$ and for some $r\in (0,\rho]$. Observe that $(\Nc_{\infty,a,r} f)(x)\meg (1+   \abs{x}_*/r)^{a}(\Nc_{\infty,a,r} f)(e) $ for every $x\in G$, so that $\Nc_{\infty,a,r}f$ is then everywhere finite. Let $(X_j)_{j\in J}$ be an orthonormal basis of $\gf$ compatible with the filtration, and set $d_j\coloneqq \deg(X_j)$ for every $j\in J$.
	Observe that, for every $\eta\in (0,1/\rho]$ and for every $y\in G$,
	\[
	\begin{split}
		\abs{f(y)}&\meg \inf_{B(y,\eta r )} \abs{f} + \sup_{ B(y,\eta r )} \abs{f-f(y)}\\
		&\meg \left(  \dashint_{B(y,\eta r )} \abs{f}^p\,\dd \beta\right) ^{1/p} + \sum_{j\in J} (\eta r )^{d_j} \sup_{ B(y,2\eta r )} \abs{X_j f}
	\end{split}
	\]
	thanks to~\cite[Lemma 7.5]{Calzi2}. Now, observe that there is a constant $C_2>0$ such that $\frac{r'^{Q_*} }{C_2}\meg \beta(B(e,r'))\meg C_2 r'^{Q_*} $ for every $r'\in (0,1]$, so that (assuming that $\rho\Meg 1$, as we may)
	\[
	\begin{split}
		\left(  \dashint_{B(y,\eta r )} \abs{f}^p\,\dd \beta\right) ^{1/p}&\meg C_2^{1/p} (r \eta)^{-Q_*/p} 2^a \left( \int_G \abs{f(z)}^p (1+d(y,z)/r)^{-pa}\,\dd \beta(y)\right) ^{1/p} \\
		&\meg 2^a(C_2/\eta^{Q_*})^{1/p_0} (\Nc_{p,a,r} f)(y)\\
		&\meg 2^a(C_2/\eta^{Q_*})^{1/p_0} (\Nc_{p,a,r} f)(x) (1+d(x,y)/r)^a
	\end{split}
	\]
	for every $x\in G$.  In addition, by~\textsc{step I} (assuming that $\lambda \Meg d_j$ for every $j\in J$, as we may),
	\[
	\Nc_{\infty,a,r} \left( \sup_{ B(\,\cdot\,,2\eta r )} \abs{X_j f}\right) \meg (1+2\eta)^{a}\Nc_{\infty,a,r} (X_j f)\meg C_1(1+2\eta)^{a} r^{-d_j }\Nc_{\infty,a,r} f
	\]
	for every $j\in J$.
	Consequently, assuming that $1+2\eta\meg 2^{p_0/Q_*}$,
	\[
	\Nc_{\infty,a,r} f\meg 2^a (C_2/\eta^{Q_*})^{1/p_0}  \Nc_{p,a,r} f+ C_1\card(J)\eta 2^{a p_0/Q_*}\Nc_{\infty,a,r} f,
	\]
	since $\eta\meg 1$ and $d_j\Meg 1$ for every $j\in J$.
	If we take the largest $\eta\in (0,1/\rho]$ so that $1+2\eta\meg 2^{p_0/Q_*}$ and $C_1\card(J)\eta 2^{a_0p_0/Q_*} \meg 1/2$, we then infer that there is a constant $C_3>0$ such that, for every $p\in [p_0,\infty)$, for every $a\in (0,a_0]$, and for every $r\in (0,\rho]$,
	\[
	\Nc_{\infty,a,r} f\meg   C_3 \Nc_{p,a,r} f
	\]
	since $\Nc_{\infty,a,r} f$ is everywhere finite.  The assertion follows in this case by means of~\textsc{step I}.
	
	\textsc{Step III} Now, take $f\in \Sr'(G)$ so that $f=f*\Kc_{r^{\grado}\Lc}(m)$, and assume that $(\Nc_{\infty,a,r} f)(e)=+\infty$   for some $a\in (0,a_0]$ and for some $r\in (0,\rho]$, so that $(\Nc_{\infty,a,r} f)(x)=+\infty$  for every $x\in G$ by the arguments of~\textsc{step II}. It will suffice   to prove that $(\Nc_{p,a,r} f)(x)=+\infty$ for every $x\in G$ and for every $p\in [p_0,+\infty)$. 
	Let $I$ be the set of $a'\in (a,+\infty)$ such that $(\Nc_{\infty,a',r} f)(e)$ is finite, and observe that $I$ is not empty since $f$ has polynomial growth, and that $I$ is an interval  since the mapping $a'\mapsto (\Nc_{\infty,a',r} f)(e)$ is decreasing. Then, by~\textsc{step II},
	\[
	\Nc_{\infty,a',r} f \meg C_{3,a'} \Nc_{p, a', r} f 
	\]
	for every $a'\in I$, where we highlighted the dependency on $a'$ of the constant $C_3$ found therein since we have no guarantee that $a'\meg a_0$. Using the fact that, for every $x\in G$, the mapping $a'\mapsto (\Nc_{\infty,a',r} f)(x)$ is	 lower semi-continuous (and decreasing), while the mapping $a'\mapsto (\Nc_{p,a',r} f)(x)$ is right continuous (and decreasing) by the monotone convergence theorem, we see that the same inequality holds also for $a'\coloneqq\inf I$ (assuming, as we may, that the function $a\mapsto C_{3,a}$ is right-continuous and increasing). If $a'\not \in I$, then clearly $\Nc_{p, a', r} f  $ is identically equal to $+\infty$, whence the conclusion by the monotonicity of $\Nc_{p,\,\cdot\,,r}f$. We are left with the case in which $a' \in I$, so that $a<a'$. We shall prove that $\Nc_{p,a,r} f$ is identically equal to $+\infty$.
	Observe first that, since $a-2^{-k}\not \in I$ for every $k\in\N$, there is a sequence $(x_k)$ of elements of $G$ such that 
	\[
	\abs{f(x_k)}\Meg 2^k (1+\abs{x_k}_*/r)^{a'-2^{-k}}
	\] 
	for every $k\in\N$. Notice that we may also assume that  $\abs{x_k}_*\Meg k$ for every $k\in \N$. In addition, since $a'\in I$, by~\textsc{step I} there is a constant $C_4\Meg 1$ such that
	\[
	\abs{(X_j f)(x)}\meg C_4 (1+\abs{x}_*/r)^{a'}
	\]
	for every $j\in J$ and for every $x\in G$. Consequently, arguing as in~\textsc{step II} we see that
	\[
	\begin{split}
		\abs{f(x)}&\Meg \abs{f(x_k)}-\abs{f(x)-f(x_k)}\\
		&\Meg  2^k (1+\abs{x_k}_*/r)^{a'-2^{-k}}- C_4 \card(J)\delta_k (1+\abs{x_k}_*/r)^{a'}\\
		&\Meg (2^{k}-1) (1+\abs{x_k}_*/r)^{a'-2^{-k}}
	\end{split}
	\]
	for every $k\in \N$ and for every $x\in B(x_k, \delta_k)$, where $\delta_k=(C_4 \card(J))^{-1} (1+\abs{x_k}_*/r)^{-2^{-k}}<1$. Now, take $x\in G$ and observe that, for every $k\in\N$,
	\[
	\begin{split}
		& r^{-Q_*}\int_{G} \abs{f(y) }^p(1+d(x,y)/r)^{-p a}\,\dd \beta(y) \Meg  r^{-Q_*} (1+\abs{x_k}_*/r+\abs{x}_*/r+\delta_k/r)^{-p a} \int_{B(x_k,\delta_k)} \abs{f}^p\,\dd \beta\\
		&\qquad\Meg C_2   (C_4 \card(J) r)^{-Q_*} (1+\abs{x_k}_*/r)^{p a'-(Q_*+p) 2^{-k}}(1+\abs{x_k}_*/r+\abs{x}_*/r+\delta_k/r)^{-p a} (2^{k}-1)^p  ,
	\end{split}
	\]
	which diverges to $+\infty$ when $k\to \infty$. The assertion follows.
\end{proof}

 We shall now present some technical lemmas which show how one may  use the estimates provided in Corollary~\ref{cor:12} to deal with convolution kernels associated with Schwartz multipliers, and also how one may control the convolution of convolution kernels associated with the functional calculi of two distinct formally self-adjoint weighted sub-coercive operators.

\begin{lem}\label{lem:36}
	Take $\lambda, a,\rho >0$, $p_0\in (0,1]$. Then, there is a continuous norm $\nu$ on $\Sr(\R)$ such that 
	\[
	\norm{X Y^R\Kc_{t^{\grado}\Lc}(\phi)(1+\abs{\,\cdot\,}_*/t)^a}_{L^p(G)} \meg   t^{-\deg(X)-\deg(Y)+(1/p-1)Q_*} \abs{X}\abs{Y} \norm{\phi}_\nu
	\]
	for every $t\in (0,\rho]$, for every $\phi\in \Sr(\R)$, for every $p\in [p_0,\infty]$, and for every $X,Y\in U_\lambda$.
\end{lem}

\begin{proof}
	Take $s\in \N$ with $s> a+D(1/p_0-1/2)_+$.
	By Corollary~\ref{cor:12}, there is $C_1>0$ so that
	\[
	\norm{X Y^R\Kc_{t^{\grado}\Lc}(\phi)(1+\abs{\,\cdot\,}_*/t)^a}_{L^p(G)} \meg C_1 t^{-\deg(X)-\deg(Y)+(1/p-1)Q_*} \abs{X}\abs{Y} \norm{\phi}_{W^{s,\infty} (\R)}
	\]
	for every $\phi\in C^\infty(\R)$ supported in $[-1,1]$, for every $t\in (0,\rho]$,  for every $p\in [p_0,\infty]$, and for every $X,Y\in U_\lambda$.  
	Now, take $\eta_0\in C^\infty_c(\R)$ so that $\chi_{[-1/2,1/2]}\meg \eta_0 \meg \chi_{[-1,1]}$, and define $\eta_j\coloneqq \eta_0(2^{-j}\,\cdot\,)-\eta_0(2^{1-j}\,\cdot\,)$ for every $j\Meg 1$, so that $\sum_{j\in \N} \eta_j=1$ locally uniformly. 
	Next, take $c>(2 \lambda+ Q_*)/\grado$ and define a continuous norm $\nu'$ on $\Sr(\R)$ so that $\norm{\phi}_{\nu'}=\sup_{\lambda\in \R}  (1+\abs{\lambda})^{c+s}\sum_{j=0}^{s}  \abs{\phi^{(j)}(\lambda)} $. Observe that there is a constant $C_2>0$ such that
	\[
	\norm{(\phi \eta_j)(2^j\,\cdot\,)}_{W^{s,\infty}(\R)} \meg C_2 2^{-jc} \norm{\phi}_{\nu'}
	\]
	for every $\phi\in \Sr(\R)$ and for every $j\in\N$. Consequently,
	\[
	\begin{split}
		&\norm{X Y^R\Kc_{t^{\grado}\Lc}(\phi \eta_j )(1+\abs{\,\cdot\,}_*/t)^a}_{L^p(G)}\\
		&\qquad  \meg \norm{X Y^R\Kc_{2^{-j}t^{\grado}\Lc}( (\phi \eta_j)(2^j\,\cdot\,) )(1+\abs{\,\cdot\,}_*/(2^{-j/\grado}t))^a}_{L^p(G)}\\
		&\qquad\meg C_1 (2^{-j/\grado}t)^{-\deg(X)-\deg(Y)+(1/p-1)Q_*} \abs{X}\abs{Y} \norm{(\phi \eta_j)(2^j\,\cdot\,)}_{W^{s,\infty}(\R)}\\
		&\qquad\meg  C_1 C_2 2^{j (-c+ (\deg(X)+\deg(Y)+(1-1/p)Q_*)/\grado )} t^{-\deg(X)-\deg(Y)+(1/p-1)Q_*} \abs{X}\abs{Y} \norm{\phi}_{\nu'}
	\end{split}
	\]
	for every  $\phi\in \Sr(\R)$, for every $t\in (0,\rho]$,  for every $p\in [p_0,\infty]$, and for every $X,Y\in U_\lambda$. Consequently,
	\[
	\begin{split}
		&\norm{X Y^R\Kc_{t^{\grado}\Lc}(\phi)(1+\abs{\,\cdot\,}_*/t)^a}_{L^p(G)}  \meg \norm*{ \sum_{j\in \N} \abs{X Y^R\Kc_{t^{\grado}\Lc}( \phi \eta_j)  (1+\abs{\,\cdot\,}_*/t)^a}^{\min(1,p)}  }_{L^{\max(1,p)}(G)}^{1/\min(1,p)}\\
		&\qquad\meg \left( \sum_{j\in \N} \norm{X Y^R\Kc_{t^{\grado}\Lc}( \phi \eta_j )(1+\abs{\,\cdot\,}_*/t)^a}_{L^p(G)}^{\min(1,p)}   \right)^{1/\min(1,p)}\\
		&\qquad\meg C_1 C_2  t^{-\deg(X)-\deg(Y)+(1/p-1)Q_*} \abs{X}\abs{Y} \norm{\phi}_{\nu'} \left(\sum_{j\in \N} 2^{j(-c+ (2 \lambda+Q_*)/\grado )\min(1,p)} \right)^{1/\min(1,p)} 
	\end{split}
	\]
	for every  $\phi\in \Sr(\R)$, for every $t\in (0,\rho]$,  for every $p\in [p_0,\infty]$, and for every $X,Y\in U_\lambda$, so that the assertion follows.
\end{proof}

\begin{deff}\label{def:20}
	Take $\lambda,\mi, a\Meg0 $, $\rho>0$, and $p\in [1,\infty]$. We say that a family $(f_t)_{t\in (0,\rho]}$ of elements of $W^{\mi,p}_\loc(G)$ is a $(\lambda,\mi,a,\rho,p)$-standard system if the following hold:
	\begin{itemize}
		\item for every $t\in (0,\rho]$ and for every $Y\in U_\mi$,
		\[
		\norm{Y f_{t} (1+\abs{\,\cdot\,}_*/t )^a  }_{L^p(G)}\meg \abs{Y} t^{-\deg(Y)-Q_*/p'};
		\]

		\item there are a basis $(X_j)_{j\in J}$ of    $U_\lambda$ and a family $(\widetilde f_{j,t})_{j\in J, t\in (0,\rho)}$ of elements of $L^{1}_\loc(G)$ such that
		\[
		f_t= \sum_{j\in J} X_j \widetilde f_{j,t}
		\]
		for every $t\in (0,\rho)$, and such that\footnote{Notice that $\deg(X_j)\meg \lambda$, and that equality need not occur. We would have therefore obtained a weaker condition if we had replaced $t^\lambda$ with the apparently more natural $t^{\deg(X_j)}$. The stronger condition is essentially needed because of the lack of homogeneity.}
		\[
		\sum_{j\in J}\norm*{  \widetilde f_{j,t} (1+\abs{\,\cdot\,}_*/t)^a  }_{L^1(G)}\meg \frac{1}{\abs{X_j}} t^{\lambda}
		\]
		for every $t\in (0,\rho)$.
	\end{itemize}
\end{deff}

\begin{lem}\label{lem:70}
	Take $m\in\N$ and  $\mi, a,\rho >0$. Then, there is a continuous norm $\nu$ on $\Sr(\R)$ such that, if $(\phi_t)_{t\in (0,\rho]}\in \Sr(G)$, $\norm{\phi_t}_\nu\meg 1$ for every $t\in (0,\rho]$, and $\phi_t^{(j)}(0)=0$ for every $j=0,\dots, m-1$, then $(\Kc_{t^\grado \Lc}(\phi_t))_{t\in (0,\rho]}$ and $(\Kc_{t^\grado \Lc}(\phi_t)^*)_{t\in (0,\rho]}$ are $(m\grado, \mi,a,\rho,p)$-standard systems for every $p\in [1,\infty]$.
	
	If, in addition, $\Lc=\sum_{j\in J} (-1)^{\dd/d_j} X_j^{2\dd/d_j}$ for some minimal basis $(X_j)_{j\in J}$, with $d_j\coloneqq \deg(X_j)$ (cf.~Proposition~\ref{prop:4}), then one may take $\nu$ so that $(\Kc_{t^\grado \Lc}(\phi_t))_{t\in (0,\rho]}$ and $(\Kc_{t^\grado \Lc}(\phi_t)^*)_{t\in (0,\rho]}$ are also $(m\dd, \mi,a,\rho,p)$-standard systems for every $p\in [1,\infty]$.
\end{lem}

\begin{proof}
	Observe first that $\Kc_{t^\grado \Lc}(\phi_t)^*=\Kc_{t^\grado \Lc}(\overline{\phi_t})$ for every $t\in (0,\rho]$ so that, if we take $\nu$ to be conjugate-invariant, then it will suffice to prove that $(\Kc_{t^\grado \Lc}(\phi_t))_{t\in (0,\rho]}$ is an $(m\grado, \mi,a,\rho,p)$-standard system. In addition, by Lemma~\ref{lem:36}, it is clear how to find $\nu$ in order that $\Kc_{\rho\Lc}(\phi_\rho)$ satisfy the required estimates (that is, the first condition in Definition~\ref{def:20}).
	Next, let $\Sr_m(\R)$ be the space of $\phi\in \Sr(\R)$ such that $\phi^{(j)}(0)=0$ for every $j=0,\dots, m-1$, and observe that the mapping $\Sr(G)\ni \phi\mapsto (\,\cdot\,)^m \phi\in \Sr_m(G)$ is an isomorphism (onto). In fact, by the open mapping theorem, it suffices to observe that it is onto, and this follows from the fact that $x^{-m} \phi(x)= \frac{1}{(m-1)!} \int_0^1 (1-y)^{m-1}\phi^{(m)}(x y)\,\dd y$ for every $x\neq 0$ and for every $\phi\in \Sr_m(\R)$. Consequently, by Lemma~\ref{lem:36} we may find a continuous semi-norm $\nu'$ on $\Sr(\R)$ such that
	\[
	\norm{\Lc^{-m} \Kc_{t^\grado \Lc}(\phi_t)(1+\abs{\,\cdot\,}_*)^a }_{L^p(G)}\meg t^{m\grado-Q_*/p'} \norm{\phi_t}_{\nu'}
	\]
	for every $t\in (0,\rho)$ and for every $p\in [1,\infty]$, where $ \Lc^{-m} \Kc_{t^\grado \Lc}(\phi_t)\coloneqq \Kc_{t^\grado \Lc}((t^{-\grado }\,\cdot\,) ^{-m} \phi_t)$ by  abuse of notation  (no abuse of notation is needed if $\Lc$ is one-to-one on $\Sr'(G)$). The first assertion follows since $\Kc_{t^\grado \Lc}(\phi_t)=\Lc^{m  } [\Lc^{-m} \Kc_{t^\grado \Lc}(\phi_t)]$ for every $t\in (0,\rho)$. The second assertion follows since $ \Kc_{t^\grado \Lc}(\phi_t)=\sum_{j\in J} X_j^{\dd/d_j} [(-X_j)^{\dd/d_j}\Lc^{-m} \Kc_{t^\grado \Lc}(\phi_t)]$ for every $t\in (0,\rho)$.
\end{proof}

\begin{lem}\label{lem:40}
	Take $\lambda ,\mi\Meg 0$, $a,\rho,\rho'>0$, and $p\in [1,\infty]$. Then, there is a constant $C>0$ such that
	\[
	\norm{ (f_t*g^*_s)  (1+\abs{\,\cdot\,}_*/(t+s))^a }_{L^p(G)}\meg C  \min\left((t/s)^\lambda, (s/t)^\mi  \right)    (t+s)^{-Q_*/p'}
	\]
	for every $t\in (0,\rho]$, for every $s\in (0,\rho]$, for every $(\lambda,\mi,a,\rho,p)$-standard system $(f_t)$, and for every $(\mi,\lambda,a,\rho',p)$-standard system $(g_s)$.
\end{lem}
 
\begin{proof}
	If either $t=\rho$ or $s=\rho'$, then the assertion follows easily from Young's inequality. Assume, then, that $t<\rho$ and $s<\rho'$. Take a basis  $(X_j)_{j\in J}$ of   $U_\lambda $ and a family  $(\widetilde f_{j,t})_{j\in J, t\in (0,\rho)}$ of elements of $L^1_\loc(G)$ so that $f_t=\sum_{j\in J} X_j \widetilde f_{j,t}$ and
	\[
	\norm*{  \widetilde  f_{j,t} (1+\abs{\,\cdot\,}_*/t)^{a}}_{L^1(G)}\meg \frac{t^{\lambda}}{\abs{X_j}}
	\]
	for every $j\in J$ and for every $t\in (0,\rho)$. Observe that $\deg(X^*_j)=\deg(X_j)\meg \lambda$ and that $\abs{X^*_j}=\abs{X_j}$ for every $j\in J$.  If $t\meg s$, then
	\[
	\begin{split}
		\norm*{  (f_t*g^*_s   )(1+\abs{\,\cdot\,}_*/(t+s))^a }_{L^p(G)}&\meg \sum_{j\in J} \norm*{   [\widetilde f_{j,t} * X_j^R(g^*_s) ](1+\abs{\,\cdot\,}_*/(t+s))^a}_{L^p(G)}\\
			&\meg \sum_{j\in J} \norm*{   [\widetilde f_{j,t} * (X_j^* g_s)^* ](1+\abs{\,\cdot\,}_*/(t+s))^a}_{L^p(G)}\\
			&\meg \sum_{j\in J} \norm*{   \widetilde f_{j,t} (1+\abs{\,\cdot\,}_*/t)^a}_{L^1(G)} \norm{X_j^* g_s  (1+\abs{\,\cdot\,}_*/s)^a}_{L^p(G)}\\
			&\meg   \sum_{j\in J}t^\lambda/s^{\deg(X_j)} s^{-Q_*/p'}\\
			&\meg 2^{Q_*/p'}\card(J)\max(\rho',1)^\lambda (t/s)^\lambda (t+s)^{-Q_*/p'},
	\end{split}
	\]
	whence the result in this case. The other case follows reversing the roles of $(t,\lambda)$ and $(s,\mi)$, since $(f_t*g^*_s)^*=g_s* f_t^*$.
\end{proof}

\begin{cor}\label{cor:13}
	Take $\lambda, \mi  \Meg 0$, $a,\rho,\rho'>0$, $\phi\in C^\infty_c(\R)$, and $p\in [1,\infty]$. Then, there is a constant $C>0$ such that
	\[
	\norm{ (f_t*g^*_s)  (1+\abs{\,\cdot\,}_*/t)^a }_{L^p(G)}, \norm{ (g_s*f^*_t)  (1+\abs{\,\cdot\,}_*/t)^a }_{L^p(G)}\meg C   \min((t/s)^\lambda,(s/t)^\mi)  (t+s)^{-Q_*/p'}
	\]
	for every $t\in (0,\rho]$, for every $s\in (0,\rho]$, for every $(\lambda+a,\mi,a,\rho,p)$-standard system $(f_t)$, and for every $(\mi,\lambda+a,a,\rho',p)$-standard system $(g_s)$.
\end{cor}

This result improves Lemma~\ref{lem:40} when $t\meg s$.

\begin{proof}
	This follows from Lemma~\ref{lem:40} and the elementary inequality
	\[
	\max(s/t,1) (1+b/(t+s))\Meg \frac 1 2 \max(s/t,1) (1+b/\max(t,s))\Meg \frac 1 2  (1+b \max(s/t,1) /\max(t,s))=\frac 1 2 (1+b/t)
	\]
	for every $b\Meg 0$ and for every $t,s>0$.
\end{proof}

 The next  results extend~\cite[Theorem 3.3 and Corollaries 3.4 and 3.6]{Besov}. We first introduce a suitable definition. 

\begin{deff}
	Take $\delta>0$ and $\rho>1$. We say than a family $(x_{r,j})_{(r,j)\in J}$ is a complete $(\delta,\rho)$-lattice on $G$ if the following hold:
	\begin{enumerate}
		\item[(1)] $J$ is a Borel subset of $(0,+\infty)\times \N$ and the mapping $(r,j)\mapsto x_{r,j}$ is Borel-measurable; 
		
		\item[(2)] for every $r>0$, the balls $B(x_{r,j}, r\delta)$, $(r,j)\in J$, are pairwise disjoint;
		
		\item[(3)] for every $r>0$, the balls $B(x_{r,j}, r \rho \delta)$, $(r,j)\in J$, cover $G$.
	\end{enumerate}
\end{deff}
 
\begin{lem}\label{lem:38}
	Take $\delta>0$ and $\rho>2$. Then, there is a complete $(\delta,\rho)$-lattice $(x_{r,j})_{(r,j)\in J}$ on $G$. If $G$ is not compact, we may take $J=(0,+\infty)\times \N$.
\end{lem}

Obviously, when $G$ is compact we cannot take $J=(0,+\infty)\times \N$.

\begin{proof}
	Set $a\coloneqq \rho/2$. For every $\ell\in \Z$, let $(y_{\ell,j})_{j\in J_\ell}$ be a maximal $(2a^\ell\delta)$-separated family of elements of $G$, so that the $B(y_{\ell,j},a^\ell\delta)$, $j\in\N$ are pairwise disjoint while the $B(y_{\ell,j},2a^\ell\delta)$, $j\in J_\ell$, cover $G$. We may assume that $J_\ell$ is an initial segment of $\N$, so that $J_\ell=\N$ if $G$ is not compact. Set $J\coloneqq \bigcup_{\ell\in \Z} ( (a^{\ell-1}, a^\ell]\times J_\ell)$ and define $x_{r,j}\coloneqq y_{\ell,j}$ for every $r\in (a^{\ell-1}, a^\ell]$, for every $\ell\in\Z$, and for every $j\in J_\ell$. Take $\ell\in\Z$ and $r\in (a^{\ell-1},a^\ell]$. Then, the balls $B(x_{r,j},r\delta)\subseteq B(y_{\ell,j}, a^\ell \delta)$, $j\in J_\ell$, are pairwise disjoint. In addition, the $B(x_{r,j}, r\rho\delta )\supseteq B(y_{\ell,j}, 2 a^\ell \delta)$, $j\in J_\ell$, cover $G$. The assertion follows.
\end{proof}

\begin{deff}
	Take $m\in C^\infty_c(\R)$, a Banach space $Z$, and $q\in (0,\infty]$. We define $\Oc(m,\Lc;\mi,Z)$ as the space of $(f_r)_{r>0}\in \Sr'(G;Z)^{(0,+\infty)}$ such that $f_r=f_r*\Kc_{r^\grado\Lc}(m )$ for every $r>0$ (so that each $f_r$ may be identified with a function of class $C^\infty$) and such that the mapping $(0,+\infty)\ni r\mapsto f_r\in \Sr'(G;Z)$ is $\mi$-measurable.
\end{deff}

Notice that the mapping $(0,+\infty)\ni r \mapsto f_r\in C^\infty(G;Z)$ is then $\mi$-measurable, since it is equal to the mapping $r\mapsto f_r*\Kc_{\rho^\grado\Lc}(\widetilde m )$ on $(0,\rho]$ for every $\rho>0$, where $\widetilde m \in C^\infty_c(\R)$ is equal to $1$ on the convex envelope of $\Set{0}\cup\supp m$.

\begin{prop}\label{prop:10}
	Take $m\in C^\infty_c(\R)$, $p,q\in (0,\infty]$, $\rho>2$, $\delta_+>0$, and $\kappa>0$.  Then, there are $C,\delta_->0$ such that the following hold. For every Banach space $Z$, 
	for every complete $(\delta,\rho)$-lattice  $(x_{r,j})_{(r,j)\in J}$ on $G$, for every positive Radon measure $\mi$ on $(0,\kappa]$, and for every $f=(f_r)\in \Oc(m,\Lc;\mi,Z)$,
	\begin{equation}\label{eq:4}
		\norm{f}_{L^{q,p}(\mi,\beta;Z)}\meg  \norm*{ \chi_J(r,j)\chi_{B(x_{r,j}, \rho r \delta)}(x)\max_{ \overline B(x_{r,j}, \rho r \delta)} \abs{f_r}}_{L^{q,q,p}_{r,j,x}(\mi, \N,\beta)}\meg C \norm{f}_{L^{q,p}(\mi,\beta;Z)},
	\end{equation}
	(with some abuse of notation)
	provided that $\delta\in (0, \delta_+]$, and
	\begin{equation}\label{eq:6}
		\frac{1}{2^{1/p_0}} \norm{f}_{L^{q,p}(\mi,\beta;Z)}\meg   \norm*{ \chi_J(r,j) \chi_{B(x_{r,j}, \rho r \delta)}(x)\min_{ \overline  B(x_{r,j},  \rho r\delta)} \abs{f_r}}_{L^{q,q,p}_{r,j,x}(\mi, \N,\beta)}\meg C \norm{f}_{L^{q,p}(\mi,\beta;Z)},
	\end{equation}
	(with some abuse of notation) provided that $\delta\in (0,\delta_-]$ and $f\in L^{q,p}(\mi,\beta;Z)$.
\end{prop} 
 
Notice that the above maxima and minima exist since $f_r$ is of class $C^\infty$ and since the closed balls $\overline B(x,r)$ are compact.

\begin{proof}
	For simplicity, we shall present the arguments below as if $J=(0,+\infty)\times \N$. The general case only requires some additional restrictions on the range of $(r,j)$ and is therefore essentially identical but substantially more cumbersome. Take $p_0\in (0,\min(1,p,q))$ and $a>D/p_0$.
	 
	\textsc{Step I}	Let us first prove that the mappings $r\mapsto \max_{  \overline B(x_{r,j},\rho r \delta)} \abs{f_r}$ and $r\mapsto \min_{\overline  B(x_{r,j},\rho r \delta)} \abs{f_r}$ are $\mi$-measurable for every $j\in \N$. 
	Observe first that, if $E$ denotes a countable dense subset of $G$, then $\max_{ \overline B(x_{j,\rho r \delta})} \abs{f_r}=\inf_{n\in\N}\sup_{E\cap B(x_{r,j},\rho r \delta+2^{-n})} \abs{f_r}$ and $\min_{\overline  B(x_{j,\rho r \delta})} \abs{f_r}=\sup_{n\in\N}\inf_{E\cap B(x_{r,j},\rho r \delta+2^{-n})} \abs{f_r}$.\footnote{Notice that, since we do not know if $\overline B(x,r)$ is the closure of $B(x,r)$ in full generality, we have to perform a double approximation in order to reduce to the dense set $E$. Actually, the above mentioned fact is known to be true in the sub-Riemannian case, that is, when $\dd=1$ (cf.~\cite[Proposition 2.2]{Nicola}), but the proof of the cited result cannot be extended to the general case because of the lack of geodesics.} It will then suffice to show that the mapping $r\mapsto \chi_{B(x_{r,j},\rho r \delta)}(x) \abs{f_r(x)}$ is $\mi$-measurable for every $x\in E$ and for every $j\in \N$. Then, observe that this mapping is the product of the $\mi$-measurable mapping $r\mapsto \abs{f_r(x)}$ and of the  composite of the Borel-measurable mapping $r\mapsto (x_{r,j},r)$ and the lower semi-continuous (hence Borel-measurable) mapping $(y,r)\mapsto \chi_{B(x,\rho r \delta)}(y) $. The assertion follows.
	
	Now observe that, since the $B(x_{r,j},\rho r \delta)$ cover $G$,
	\[
	\begin{split}
		\norm{f}_{L^{q,p}(\mi,\beta;Z)}&\meg \norm*{ \chi_{B(x_{r,j}, \rho r \delta)}(x) \max_{ \overline B(x_{r,j},\rho r  \delta)} \abs{f_r} }_{L^{q,q,p}_{r,j,x}(\mi,\N,\beta)}
	\end{split}
	\]
	whence the left inequality in~\eqref{eq:4}. In order to prove the right inequality in~\eqref{eq:4}, we may assume that $f\in L^{q,p}(\mi,\beta;Z)$. 
	For every $j\in J$ and for every $r\in (0,\kappa]$, take $x'_{r,j}\in \overline B(x_{r,j},\rho r \delta)$ so that $\abs{f_r(x'_{r,j})}=\max_{\overline B(x_{r,j},\rho r \delta)} \abs{f_r}$, and observe that, by~\cite[Lemma 7.5]{Calzi2} and Lemma~\ref{lem:33}, there is a constant $C_1>0$ such that
	\[
	\begin{split}
		\abs{f_r(x'_{r,j})}&\meg \abs{f_r(x)}+ \sum_{k\in K} \min((2\rho r  \delta)^{d_k}, 2\rho  r\delta) \max_{\overline B(x_{r,j},2\rho r \delta)} \abs{X_k f_r}\\
		&\meg  \abs{f_r(x)}+ \sum_{k\in K} (2\rho r  \delta)^{d_k} (1+2 \rho \delta)^a\Nc_{\infty,a,r} (X_k f_r)(x)\\
		&\meg \abs{f_r(x)}+ C_1 \card(K)  (\Nc_{p_0,a,r} f_r)(x)
	\end{split}
	\]
	for every $x\in B(x_j, \rho r \delta)$, $(X_k)_{k\in K}$ is a basis of $\gf$ compatible with the filtration, and  $d_k=\deg(X_k)$ for every $k\in K$. In addition, notice that  there is a constant $C_2>0$ such that $\sum_{j} \chi_{B(x_{r,j},\rho r \delta)}\meg C_2$ on $G$ for every $r\in (0,\kappa]$ (cf.~\cite[Lemma 4.3]{Calzi2}). Consequently,
	\[
	\begin{split}
		  \norm*{ \chi_{B(x_j,\rho r \delta)}(x)\max_{\overline B(x_{r,j},\rho r \delta)}\abs{f} }_{L^{q,q,p}_{r,j,x}(\mi,\N,\beta)}^{p_0}  &\meg C_2^{p_0/q} \norm{f}_{L^{q,p}(\mi,\beta;Z)}^{p_0}\\
			&\qquad+  (C_2^{1/q} C_1 \card(K)  )^{p_0 } \norm{(\Nc_{p_0,a,r} f_r)(x)}_{L^{ q , p }_{r,x}(\mi,\beta)}^{p_0},
	\end{split}
	\]
	whence the right inequality in~\eqref{eq:4} thanks to Corollary~\ref{cor:18}.
	
	\textsc{Step II} By~\textsc{step I}, we may reduce to proving the left inequality in~\eqref{eq:6}.  Arguing as above and choosing   $x''_{r,j}\in \overline B(x_{r,j}, \rho r \delta)$ so that $\abs{f_r(x''_{r,j})}=\min_{\overline B(x_{r,j},\rho r \delta)} \abs{f_r}$, we see that there is $C_1'>0$ such that
	\[
	\begin{split}
		\abs{f_r(x)}&\meg \abs{f_r(x''_{r,j})}+ \sum_{k\in K} \min((2\rho r  \delta)^{d_k}, 2\rho r \delta) \max_{\overline B(x_{r,j},2\rho r \delta)} \abs{X_k f_r}\\
		&\meg \abs{f_r(x''_{r,j})}+ C_1'  \card(K)  \delta_- (\Nc_{p_0,a,r} f_r)(x)
	\end{split}
	\]
	for every $x\in B(x_{r,j}, \rho r \delta)$, provided that $\delta\meg \delta_-\meg 1$ (so that $\delta^{d_k}\meg \delta_-$ for every $k\in K$). This inequality may in turn be used to deduce that
	\[
	\norm{f}_{L^{q,p}(\nu,\beta;Z)}^{p_0}\meg   \norm*{  \chi_{B(x_{r,j},\rho r \delta)} \min_{\overline B(x_{r,j},\rho r \delta)} \abs{f_r} }_{L^{q,q,p}_{r,j,x}(\mi,\N,\beta)}^{p_0}+(C_1' \card(K) \delta_-)^{p_0} \norm{(\Nc_{p_0,a,r} f_r)(x)}_{L^{q,p}_{r,x}(\mi,\beta)}^{p_0}.
	\]
	By Corollary~\ref{cor:18} again,  there is a constant $C_3>0$ such that $\norm{(\Nc_{p_0,a,r} f_r)(x)}_{L^{q,p}_{r,x}(\mi,\beta)}\meg C_3\norm{f}_{L^{q,p}(\mi,\beta;Z)}$. Since $f\in L^{q,p}(\mi,\beta;Z)$, if we take $\delta_-\meg 1$ so that $(C'_1 C_3 \card(K) \delta_-)^{p_0}\meg 1/2$, we then infer that
	\[
	\norm{f}_{L^{q,p} (\mi,\beta;Z)}\meg   2^{1/p_0}   \norm*{  \chi_{B(x_{r,j},\rho r \delta)} \min_{\overline B(x_{r,j},\rho r \delta)}(x)\abs{f_r} }_{L^{q,q,p}_{r,j,x}(\mi,\N,\beta)},
	\]
	whence the conclusion.
\end{proof}

\begin{cor}\label{cor:17}
	Take $m\in C^\infty_c(\R)$, $p\in (0,\infty]$, $\rho>2$, $\delta_+>0$, and a  positive Radon measure $\mi$ on $(0,+\infty)$ with bounded support.  Then, there are $C,\delta_->0$ such that the following hold. For every Banach space $Z$, for every $q\in (0,\infty]$, for every $r\in (0,1]$, for every complete $(\delta,\rho)$-lattice  $(x_{r,j})_{(r,j)\in J}$ on $G$, and for every $f=(f_r)\in \Oc(m,\Lc;\mi,Z)$,
	\begin{equation}\label{eq:4b}
		\frac{1}{C}\norm{f}_{L^{p,q}(\beta,\mi;Z)}\meg  \delta^{Q_*/p}\norm*{\chi_J(r,j) r^{Q_*/p}  \max_{ \overline B(x_{r,j}, \rho r \delta)} \abs{f_r}}_{L^{p,q}_{j,r}(\N,\mi)}\meg C \norm{f}_{L^{p,q}(\beta,\mi;Z)}
	\end{equation}
	(with some abuse of notation) provided that $\delta\in (0, \delta_+]$, and
	\begin{equation}\label{eq:6b}
		\frac{1}{C} \norm{f}_{L^{p,q}(\beta,\mi;Z)}\meg \delta^{Q_*/p}\norm*{ \chi_J(r,j)r^{Q_*/p}  \min_{ \overline B(x_{r,j}, \rho r \delta)} \abs{f_r}}_{L^{p,q}_{j,r}(\N,\mi)}\meg C \norm{f}_{L^{p,q}(\beta,\mi;Z)}
	\end{equation}
	(with some abuse of notation) provided that $\delta\in (0,\delta_-]$ and $f_r\in L^{p}(\beta;Z)$ for $\mi$-almost every $r>0$.
\end{cor}

\begin{proof}
	Apply Proposition~\ref{prop:10}  with $\mi=\delta_r$, $r>0$, and $q=\infty$, and observe that there is a constant $C'>0$ such that $\frac{1}{C'} (\rho r \delta)^{Q_*}\meg \beta(B(e,\rho r \delta))\meg C' (\rho r \delta)^{Q_*} $ for every $r\in \Supp{\mi}$. 
	This shows that, for every $r\in \supp \mi$,
	\[
	\frac{1}{C}\norm{f_r}_{L^p (\beta;Z)}\meg  (\delta r)^{Q_*/p}\norm*{    \max_{ \overline B(x_{r,j}, \rho r \delta)} \abs{f_r}}_{\ell^{p }_{j}(J_r)}\meg C \norm{f_r}_{L^{p }(\beta;Z )}
	\]
	and, if $f_r\in L^p(\beta;Z)$,
	\[
	\frac{1}{C}\norm{f_r}_{L^p (\beta;Z)}\meg  (\delta r)^{Q_*/p}\norm*{    \min_{ \overline B(x_{r,j}, \rho r \delta)} \abs{f_r}}_{\ell^{p }_{j}(J_r)}\meg C \norm{f_r}_{L^{p }(\beta;Z )},
	\]
	where $J_r\coloneqq \Set{j\in\N\colon (r,j)\in J}$.
	Applying the $L^q(\mi)$ quasi-norm leads to the conclusion.
\end{proof}

\begin{cor}\label{cor:2}
	Take $m\in C^\infty_c(\R)$, $p_1,p_2,q\in (0,\infty]$ with $p_1\meg p_2$, and a positive Radon measure $\mi$ on $(0,+\infty)$ with bounded support. Then there is a constant $C>0$ such that 
	\[
	\norm{f}_{L^{q,p_2}(\mi,\beta;Z)}\meg C  \norm{r^{(1/p_2-1/p_1)Q_*} f_r(x)}_{L^{q,p_1}_{r,x}(\mi,\beta;Z)} 
	\]
	for every Banach space $Z$,  
	 and for every $f=(f_r)\in \Sr'(G)$ with $f\in \Oc(m,\Lc;\mi)$.
\end{cor}

\begin{proof}
	It suffices to apply Proposition~\ref{prop:10}  and to compare the `discretized' norms in~\eqref{eq:4}.  
\end{proof}

\begin{cor}\label{cor:4}
	Take a formally self-adjoint weighted subcoercive operator $\Lc'$ on $G$ with  degree $\grado'$, $m_1,m_2\in C^\infty_c(\R)$, $\rho>0$. Take $\alpha_1,\alpha_2>0$  and $p_1,p_2,p_3\in (0,\infty]$ such that   $p_1,p_2\meg p_3$  and $\frac{1}{p_1'}+\frac{1}{p_2'}\meg \frac{1}{p_3'}$.  Then there is a constant $C>0$ such that 
	\[
	\norm{\abs{f_1}^{\alpha_1}* \abs{f_2^*}^{\alpha_2}}_{L^{p_3}(G)} \meg C 
	 r^{(1+1/p_3-1/p_1-1/p_2)Q_*} \norm{\abs{f_1}^{\alpha_1}}_{L^{p_1}(G)} \norm{\abs{f_2}^{\alpha_2}}_{L^{p_2}(G)} 
	\]
	for every $r\in (0,\rho]$, for every two Banach spaces $Z_1,Z_2$, and for every $f_j\in \Sr'(G;Z_j)$, $j=1,2$, such that $f_1=f_1*\Kc_{\Lc}(m_1(r^{\grado}\,\cdot\,))$ and $f_2=f_2*\Kc_{\Lc'}(m_2(r^{\grado'}\,\cdot\,))$. 
\end{cor}

\begin{proof}
	By means of Corollary~\ref{cor:2} and Young's inequality, we may reduce to the case $p_1=p_2=p_3\in [p_0,1]$. We then simply write $p$ instead of $p_1,p_2,p_3$. Then, take $C_1>0$ as in Corollary~\ref{cor:17} (for $\mi=\delta_r$, $q=\infty$ and  $p=\alpha_1p_1$, and also for $\mi=\delta_r$, $q=\infty$, $p=\alpha_2p_2$, and $\Lc$ replaced by $\Lc'$), and select a maximal $r/2$-separated family $(x_j)_{j\in J}$ of elements of $G$. In addition, take $C_2>0$ so that $\beta(B(e,r'))\meg C_2r'^{Q_*}$ for every $r'\in (0,\rho]$. Then,
	\[
	\begin{split}
		\norm{\abs{f_1}^{\alpha_1}* \abs{f_2^*}^{\alpha_2}}_{L^p(G)}^p
			&= \int_G   \left( \int_G \abs{f_1(y) }^{\alpha_1} \abs{f_2(x^{-1}y)}^{\alpha_2} \,\dd \beta(y) \right)^p\,\dd \beta(x)\\
			&\meg (C_2 r^{Q_*} )^{p} \int_G \left( \sum_{j\in J} \max_{\overline B(x_{j},r)}\abs{f_1 }^{\alpha_1} \max_{\overline B(x^{-1} x_{j}, r) }\abs{f_2}^{\alpha_2}\right)^p\,\dd \beta(x)\\
			&\meg (C_2 r^{Q_*} )^{p}\sum_{j\in J}   \max_{\overline B(x_{j},r)}\abs{f_1 }^{p\alpha_1}  \int_G \max_{\overline B(x^{-1} x_{j}, r ) }\abs{f_2}^{p\alpha_2}\,\dd \beta(x)\\
			&=(C_2 r^{Q_*} )^{p}\sum_{j\in J}   \max_{\overline B(x_{j},r)}\abs{f_1 }^{p\alpha_1}  \int_G \max_{\overline B(x , r) }\abs{f_2}^{p\alpha_2}\,\dd \beta(x)\\
			&\meg (C_2 r^{Q_*} )^{1+p}  r^{-Q_*}C_1^{p \alpha_1} \norm{f_1}^{p\alpha_1}_{L^{p \alpha_1}(\beta;Z_1)} \sum_{j'\in J}\max_{\overline B(x_{j'} , 2 r) }\abs{f_2}^{p\alpha_2}\\
			&\meg (C_2 r^{Q_*} )^{1+p}  r^{-2Q_*}C_1^{p (\alpha_1+\alpha_2)}  \norm{\abs{f_1}^{\alpha_1}}_{L^p(G)}^{p }  \norm{\abs{f_2}^{\alpha_2}}_{L^p(G)}^{p } 
	\end{split}
	\]
	The assertion follows.
\end{proof}

\begin{cor}\label{cor:5}
	Take a formally self-adjoint weighted subcoercive operator $\Lc'$ on $G$ with  degree $\grado'$, and take $m_1\in C^\infty_c(\R)$, $\rho >0$, $p\in (0,\infty]$, and $s> D(1/p-1/2)$.
	Then there is a constant $C>0$ such that
	\[
	\norm*{ \abs{f}*\abs{\Kc_{r^{\grado'}\Lc'}(m_{2} )}}_{L^p(G)}\meg C    \norm{m_{2 }}_{B^{\infty,\infty}_s(\R)}  \norm{f}_{L^{ p}( \beta)}
	\]
	for every  $m_2\in B^{\infty,\infty}_s(\R)$ supported in $[-\rho ,\rho ]$,   for every $r\in (0,\rho ]$, for every Banach space $Z$, and for every $f\in L^{ p}( \beta;Z)$ with  $f=f*\Kc_{r^\grado\Lc}(m_1)$.
\end{cor}

\begin{proof}
	This follows from Corollaries~\ref{cor:4} and~\ref{cor:12}.
\end{proof}

\begin{lem}\label{lem:34}
	Take a formally self-adjoint weighted subcoercive operator $\Lc'$ on $G$ with  degree $\grado'$, and take $m_1\in C^\infty_c(\R)$, $\rho >0$, $p_0\in (0,\infty]$, $s> D(1/p_0+1/2)$, and a positive Radon measure $\mi$ on $(0,+\infty)$ with bounded support.
	Then, there is a constant $C>0$ such that
	\[
	\norm*{ (\abs{f_r}*\abs{\Kc_{r^{\grado'}\Lc'}(m_{2,r} )})(x) }_{L^{q,p}_{r,x}(\mi,\beta)}\meg C   \norm*{\norm{m_{2,r}}_{B^{\infty,\infty}_s(\R)}}_{L^\infty_r(\mi)} \norm{f}_{L^{q,p}(\mi,\beta)}
	\]
	for every $p,q\in [p_0,\infty]$, for every $\mi$-measurable family $(m_{2,r})$ of elements of  $ B^{\infty,\infty}_s(\R)$ supported in $[-\rho ,\rho ]$,    for every Banach space $Z$, and for every $f\in   \Oc(m_1, \Lc; \mi,Z)$.
\end{lem}

\begin{proof}
	Take $p_1\in (0,p_0)$ so that $s> D(1/p_1+1/2)$, and take $a>D/p_1$.
	Set $\psi_r\coloneqq \Kc_{r^{\grado'}\Lc'}(\overline{m_{2,r}} )$ for every $r>0$. Let us prove that the mapping $(r,x)\mapsto (\abs{f_r}*\abs{\psi_r^*})(x)$ is $(\mi\otimes \beta)$-measurable. To this aim, it will suffice to show that the mappings $(r,x)\mapsto ((\chi_{B(0,2^k)}\abs{f_r})*\abs{\psi_r^*})(x)$ are $(\mi\otimes \beta)$-measurable for every $k\in\N$. This in turn follows from Young's inequality, since the mappings $(0,+\infty)\ni r\mapsto \chi_{B(0,2^k)}\abs{f_r} \in L^2(G)$ and $(0,+\infty)\ni r\mapsto \abs{\psi_r^*}\in L^2(G)$ are $\mi$-measurable, thanks to Corollary~\ref{cor:12}.
	
	Then, observe that by Corollary~\ref{cor:12} there is a constant $C_1>0$ such that  
	\[
	\begin{split}
	(\abs{f_r}*\abs{\psi_r^*})(x)&=\int_G \abs{f_r(z) \psi_r(x^{-1}z) }\,\dd \beta(z)\\
		&\meg (\Nc_{\infty,a,r} f_r)(x) \int_G \abs{\psi_r(x^{-1}z)}(1+\abs{x^{-1}z}_*/r)^{a}\,\dd \beta(z)\\
		&\meg C_1 (\Nc_{\infty,a,r} f_r)(x)  \norm{m_{2,r}}_{B^{\infty,\infty}_s(\R)}
	\end{split}
	\]
	for every $r>0$ and for every $x\in G$, so that
	\[
	\norm{(\abs{f_r}*\abs{\psi_r^*})(x)}_{L^q_r(\mi)} \meg C_1  \norm*{\norm{m_{2,r}}_{B^{\infty,\infty}_s(\R)}}_{L^\infty_r(\mi)} \norm{(\Nc_{\infty, D/p_1,r^{\grado}} f_r)(x)  }_{L^q_r(\mi)}
	\]
	for every $x\in G$. Next,  observe that by Lemma~\ref{lem:33} there is a constant $C_2>0$ such that
	\[
	 \Nc_{\infty,a,r^{\grado}} f_r \meg C_2  \Nc_{p_1,a,r^{\grado}} f_r 
	\]
	for every  $r>0$. Consequently,
	\[
	\norm{(\Nc_{\infty, a,r^{\grado}} f_r)(x)  }_{L^{q,p}_{(r,x)}(\mi,\beta)}\meg C_2 \norm{ (\Nc_{p_1,a,r^{\grado}} f_r )(x)}_{L^{q,p}_{(r,x)}(\mi,\beta)} ,
	\]
	so that the assertion follows from Corollary~\ref{cor:18}.
\end{proof}

\begin{lem}\label{lem:25d}
	Take $\eta,\delta>0$, $\eps\in (0,1)$, and $q\in (0,\infty]$. Then, there is a constant $C>0$ such that
	\[
	\norm*{\sum_{j\in \N} \eps^{\delta(j-j')_+ + \eta (j'-j)_+} \abs{a_j}}_{\ell^q_{j'}(\N)}\meg C \norm{a_j }_{\ell^q_j(\N)}
	\]
	for every $(a_j)\in \C^{\N}$.  
\end{lem}

\begin{proof}
	If $q\Meg 1$, the assertion follows from Lemma~\ref{lem:25}, applied to the measure $\sum_{j\in \N} \delta_{\eps^j}$ and to the function $\sum_j \chi_{\eps^j} a_j$, since $\eps^{\delta(j-j')_+ + \eta (j'-j)_+}\meg 2^{\delta+\eta} \frac{\eps^{\delta j}\eps^{j'\eta}}{(\eps^j+\eps^{j'})^{\eta+\delta}}$ for every $j,j'\in\N$. The case $q<1$ follows from the case $q=1$ (with $\eta,\delta$ replaced by $q\eta,q\delta$, respectively), since
	\[
	\begin{split}
		\norm*{\sum_{j\in \N} \eps^{\delta(j-j')_+ + \eta (j'-j)_+} \abs{a_j}}_{\ell^q_{j'}(\N)}&\meg \norm*{\sum_{j\in \N} \eps^{q\delta(j-j')_+ + q\eta (j'-j)_+} \abs{a_j}^q}_{\ell^1_{j'}(\N)}^{1/q}\\
		&\meg C^{1/q} \norm{\abs{a_j}^q }_{\ell^1_j(\N)}^{1/q}\\
		&= C^{1/q} \norm{a_j}_{\ell^q_j(\N)}.
	\end{split}
	\]
	The assertion follows.
\end{proof}

\subsection{The Spaces $\Cc^q(\mi,a)$}

In order to deal with the Triebel--Lizorkin spaces $F^{\infty,q}_\alpha(G)$, we shall need to introduce the spaces $\Cc^q(\mi,a)$ in order to replace the spaces $L^{q,\infty}(\mi,\beta)$. We follow the presentation of~\cite{Calzi3}, which, in turn, follows~\cite{FrazierJawerth,Rychkov}. We shall nonetheless need to modify several results in order to match the current setting.

\begin{deff}
	Take $q\in(0,\infty]$, $a>0$, a Banach space $Z$, and a positive measure $\mi$ on $(0,+\infty)$ with bounded support. Set $c\coloneqq \max \supp \mi$. Then, we denote with $\Cc^q(\mi,a;Z)$ the space of $(\mi\otimes \beta)$-measurable functions $f\colon (0,+\infty)\times G\to Z$ such that 
	\[
	\sup_{(t,x)\in (0,c]\times G}\norm*{t^{-aQ_*/q} \chi_{(0,t]\times B(x,t^a)} f }_{L^q(\mi\otimes \beta;Z)} 
	\]
	is finite, 	endowed with the corresponding topology. We shall simply write $\Cc^q(\mi;Z)$ instead of $\Cc^q(\mi,1;Z)$ and we shall omit $Z$ when it is $\C$.
\end{deff}

Notice that all the statement below hold for functions with values in a general Banach space. Since, however, this is obvious except for Proposition~\ref{prop:10bis} (since it suffices to apply the scalar statement to $\abs{f}$), we shall only consider the general case in Proposition~\ref{prop:10bis}.

\begin{lem}\label{lem:2c} 
	Take $ p_0\in (0,\infty]$, $a >0$, $b>D/p_0$, $c>1$,  a positive Radon measure $\mi$ on $(0,+\infty)$ with bounded support,  and a  function $\kappa\colon (0,+\infty)\to (0,+\infty)$ such that $\kappa(t)\in (0,c t^{a}]\cup [1/c,+\infty)$ for every $t>0$. 
	Then, there is   $C>0$ such that, for every $q\in [p_0,\infty]$ and for every $(\mi \otimes \beta)$-measurable function $f\colon (0,+\infty) \times G\to \C$,
	\[
	\begin{split}
		\norm{ (1+\kappa(t))^{-b+Q_*/p_0}[\Nc_{p_0,b,\kappa(t)}  f(t,\,\cdot\,)](x) }_{\Cc^q_{(t,x)}(\mi,a)} \meg C  \norm{  f }_{\Cc^q(\mi,a)} .
	\end{split}
	\]
\end{lem}
 
\begin{proof} 
	Observe first that we may reduce to the case $q=p_0$, since by H\"older's inequality and Lemma~\ref{lem:37} there is a constant $C_1>0$ such that
	\[
	\begin{split}
		&(1+\kappa(t))^{-b+Q_*/p_0} (\Nc_{p_0,b,\kappa(t)} g)(x)=(1+\kappa(t))^{-b-Q_*/p_0} \kappa(t)^{-Q_*/p_0}\norm{g (1+\abs{x^{-1}\,\cdot\,}_*/\kappa(t))^{-b} }_{L^{p_0}(G)}\\
			&\qquad\meg (1+\kappa(t))^{-b+Q_*/p_0} \kappa(t)^{-Q_*/p_0}\norm{g (1+\abs{x^{-1}\,\cdot\,}_*/\kappa(t))^{-D/q-\eps/2} }_{L^{q}(G)} \norm{(1+\abs{\,\cdot\,}_*/\kappa(t))^{-b+D/q+\eps/2}}_{L^{s}(G)}\\
			&\qquad\meg C_1 (1+\kappa(t))^{-b+Q_*/p_0} \kappa(t)^{Q_*(1/q-1/p_0)} (\Nc_{q,D/q+\eps/2,\kappa(t)} g)(x) \kappa(t)^{Q_*(1/p_0-1/q)}(1+\kappa(t))^{(D-Q_*)(1/p_0-1/q)}\\
			&\qquad= C_1(1+\kappa(t))^{-D/q-\eps/2+Q_*/q} (\Nc_{q,D/q+\eps/2,\kappa(t)}  g)(x)
	\end{split}
	\]
	where $s\in (0,\infty]$ is such that $\frac{1}{s}=\frac{1}{p_0}-\frac{1}{q}$, since $b-D/q-\eps/2>D(1/p_0-1/q)$,
	for every $\beta$-measurable function $g\colon G\to \C$, where $\eps=b-D/p_0$.
	The proof is then essentially identical to the proof of~\cite[Lemma 3.5]{Calzi3}. We only need to replace $\ee^{-\omega \kappa(t)}T_{b,\kappa(t),d}$ with $ (1+\kappa(t))^{-b+Q_*/q}\Nc_{q,b,\kappa(t)}$ and~\cite[Lemma 2.11]{Calzi3} with Lemma~\ref{lem:37}, observe that 
	\[
	[\Nc_{q,b,\kappa(t)} f(t,\,\cdot\,)]^q= \Nc_{1,b q,\kappa(t)}(\abs{f(t,\,\cdot\,)}^q)
	\]
	when $q<\infty$, and perform a `discretization' of $ (1+\kappa(t))^{-b+Q_*/q}\kappa(t)^{-Q_*}(1+d(\,\cdot\,,y)/\kappa(t))^{-bq} $ instead of $\ee^{-\omega \kappa(t)}p_{b,\kappa(t),d}$.
\end{proof}
 
\begin{lem}\label{lem:25c}
	Take $q\in (0,\infty]$, $\eps\in (0,1)$, $ c>1$, $a,  \eta,\delta>0$,  $b>D/q$, and $\nu\in \Mc_\Car$. Set $\mi\coloneqq \sum_{j=1}^\infty \delta_{\eps^j}$. Then, there is a constant $C>0$ such that 
	\[
	\norm*{  \int_0^{+\infty} \frac{s^\delta t^\eta}{(s+t)^{\delta+\eta}} (\Nc_{q,b,c s^a}f_s)(x)\,\dd \mi(s)}_{\Cc^q_{(t,x)}(\nu,a)}\meg  C \norm{   f_t(x)  }_{\Cc^q_{(t,x)}(\mi,a)}
	\]
	for every  $(f_t)\in L^1_\loc(G)^{(0,+\infty)}$.
\end{lem}

\begin{proof}
	Assume first that $q\Meg 1$. Then, the proof similar to that of~\cite[Lemma 4.3]{Calzi3}, with $T_{b,\kappa(s),d}$ replaced by $ \Nc_{q,b,c s^a}$ and $d$ replaced by $1/a$. The only difference arises when trying to estimate $\norm{\chi_{B(x',t'^a)} \Nc_{q,b,c s^a} f_s}_{L^q(G)}$ with $\norm{   f_t(x)  }_{\Cc^q_{(t,x)}(\mi,a)}$ for $x'\in G$, $s\in \supp \mi\cap [t'+\infty)$, and $t'\in (0,c']$, where $c'=\max \supp(\mi+\nu)$. In fact, in this case one should observe that $(\Nc_{q,b,c s^a} f_s)(x)\meg (1+d(x,y)/(cs^a))^b(\Nc_{q,b,c s^a} f_s)(y)\meg 3^b(\Nc_{q,b,c s^a} f_s)(y)$ for every $x,y\in B(x',t'^a)$, so that there is a constant $C_1>0$ such that
	\[
	\norm{\chi_{B(x',t'^a)} \Nc_{q,b,c s^a} f_s}_{L^q(G)}\meg C_1 \norm{\chi_{B(x',s^a)} \Nc_{q,b,c s^a} f_s}_{L^q(G)}\meg C_1 \norm{  \Nc_{q,b,c t^a} f_t(x)  }_{\Cc^q_{(t,x)}(\mi,a)},
	\]
	so that the desired estimate follows by means of Lemma~\ref{lem:2c}.
	The case $q<1$ follows from the case $q=1$, with $\delta,\eta$ replaced by $q\delta,q\eta$, respectively, since 
	\[
	\Big(\int_0^{+\infty} \frac{s^\delta t^\eta}{(s+t)^{\delta+\eta}} (\Nc_{q,b,c s^a}f_s)(x)\,\dd \mi(s)\Big)^q\meg \int_0^{+\infty} \frac{s^{q\delta} t^{q\eta}}{(s+t)^{q\delta+q\eta}} (\Nc_{1,b q,c s^a}\abs{f_s}^q)(x)\,\dd \mi(s)
	\]
	thanks to the our choice of $\mi$.
\end{proof}

\begin{lem}\label{lem:34b}
	Take a formally self-adjoint weighted subcoercive operator $\Lc'$ on $G$ with  degree $\grado'$, and take $m_1\in C^\infty_c(\R)$, $\rho >0$, $p_0\in (0,\infty)$, $s> D(1/p_0+1/2)$, and a positive Radon measure $\mi$ on $(0,+\infty)$ with bounded support.
	Then there is a constant $C>0$ such that
	\[
	 \norm{(\abs{f_r}*\abs{\Kc_{r^{\grado'}\Lc'}(m_{2,r} )})(x)}_{\Cc^q_{(r,x)}(\mi,1/\grado')} \meg C   \norm*{\norm{m_{2,r}}_{B^{\infty,\infty}_s(\R)}}_{L^\infty_r(\mi)} \norm{f_r(x)}_{\Cc^{q}_{(r,x)}(\mi,1/\grado)}
	\]
	for every $\mi$-measurable family $(m_{2,r})$ of elements of  $ B^{\infty,\infty}_s(\R)$ supported in $[-\rho ,\rho ]$,  for every $ q\in [p_0,\infty]$, and for every $f\in   \Oc(m_1, \Lc; \mi)$.
\end{lem}

\begin{proof}
	It suffices to repeat the proof of Lemma~\ref{lem:34}, replacing Corollary~\ref{cor:18} with Lemma~\ref{lem:2c}.
\end{proof}

\begin{prop}\label{prop:10bis}
	Take $m\in C^\infty_c(\R)$, $q \in (0,\infty]$, $\rho>2$, $\delta_+>0$, and $\kappa>0$.  Then, there are $C,\delta_->0$ such that the following hold. For every   complete $(\delta,\rho)$-lattice  $(x_{r,j})_{(r,j)\in J}$ on $G$, for every positive Radon measure $\mi$ on $(0,\kappa]$, for every Banach space $Z$, and for every $f=(f_r)\in \Oc(m,\Lc;\mi,Z)$, setting $J_r\coloneqq \Set{j\in \N\colon (r,j)\in J}$ for every $r>0$,
	\begin{equation*} 
		\norm{f}_{\Cc^{q}(\mi;Z)}\meg  \norm*{\norm*{\chi_{B(x_{r,j}, \rho r \delta)}(x)\max_{ \overline B(x_{r,j}, \rho r \delta)} \abs{f_r}}_{\ell^q_j(J_r)}}_{\Cc^{q}_{r,x}(\mi)}\meg C \norm{f}_{\Cc^{q}(\mi;Z)},
	\end{equation*}
	provided that $\delta\in (0, \delta_+]$, and
	\begin{equation*}
		\frac{1}{2} \norm{f}_{\Cc^{q}(\mi;Z)}\meg  \norm*{\norm*{ \chi_{B(x_{r,j}, \rho r \delta)}(x)\min_{ \overline B(x_{r,j}, \rho r \delta)} \abs{f_r}}_{\ell^q_j(J_r)}}_{\Cc^{q}_{r,x}(\mi)}\meg C \norm{f}_{\Cc^{q}(\mi;Z)},
	\end{equation*}
	provided that $\delta\in (0,\delta_-]$ and $f\in \Cc^{q}(\mi;Z)$.
\end{prop} 

\begin{proof}
	It suffices to repeat the proof of Proposition~\ref{prop:10}, replacing  Corollary~\ref{cor:18} with Lemma~\ref{lem:2c}.
\end{proof}

\begin{deff}\label{def:1}
	Take $q\in (0,\infty]$, $a>0$, $\eta\in (0,1)$,  and a positive measure $\mi$ on $(0,+\infty)$. Define 
	\[
	(m^q_{\eta,a,t} f)(x)\coloneqq \sup\Set{\eps>0\colon  \beta \big(\Set{y\in B(x,t^a) \colon  \norm{  \chi_{(0,t]} f(\,\cdot\,,x) }_{L^q(  \mi)}>\eps  }\big)> \eta\beta(B(e,t^a))}
	\]
	and
	\[
	(m^q_{\eta,a} f)(x)\coloneqq \sup_{t\in (0,c]} (m^q_{\eta,a,t} f )(x)
	\]
	for every $(\mi\otimes\beta)$-measurable function $f\colon (0,+\infty)\times G\to \C$ and for every $x\in G$, where $c\coloneqq \sup \supp \mi$.
\end{deff}

Cf.~\cite[Remark 3.7]{Calzi3} for a discussion of the measurability of $m^q_{\eta,a,t}$ and $m^q_{\eta,a}$.
We now repeat the statement of~\cite[Lemma 3.8]{Calzi3}. We shall  provide later a suitable `converse' for this setting, thus replacing~\cite[Lemma 3.9]{Calzi3}.

\begin{lem}\label{lem:59b}
	Take $p\in (0,\infty)$, $q\in (0,\infty]$, $a>0$, $\eta\in (0,1)$,    and a positive measure $\mi$  on $(0,+\infty)$ with bounded support. Then there is a constant $C>0$ such that for every $(  \mi \otimes \beta)$-measurable function $f\colon (0,+\infty) \times G\to \C$ there is a $(\mi\otimes \beta)$-measurable subset $E$ of $(0,+\infty)\times G$ such that 
	\[
	\beta(\Set{y\in B(x,t^a)\colon (t,y)\in E})\Meg (1-\eta) \beta(B(e,t^a))
	\]
	for every $x\in G$ and for  every $t\in (0,\max \supp \mi]$, and such that
	\[
	\norm{\chi_E f}_{L^{q,p}(\mi,\beta)}\meg \norm{m^q_{\eta,a} f}_{L^p(G)}\meg C\norm{f}_{L^{q,p}(\mi,\beta)}
	\]
	and
	\[
	\frac{1}{C}\norm{\chi_E f}_{L^{q,\infty}(\mi,\beta)}\meg \norm{m^q_{\eta,a} f}_{L^\infty(G)}\meg C\norm{f}_{\Cc^{q}(\mi,a)}
	\]
\end{lem}

\section{Besov and Triebel-Lizorkin Spaces}\label{sec:3}

Before proving the equivalence theorem between the formulations with the heat semigroups and with Littlewood--Paley decompositions, it will be convenient to (re-)define Besov and Triebel--Lizorkin spaces using Littlewood--Paley decompositions and establish their first properties. Since \emph{a posteriori} these spaces will coincide with the ones previously defined for general Lie groups, we shall \emph{not} introduce different notation (except for the fact that we shall no longer put emphasis on the underlying measure, since in this case there is no ambiguity). We shall nonetheless never use the results previously proved with the heat semigroups until we establish the relevant equivalence theorem.

\emph{From now on,  functions and distributions denoted by $f$ and $g$ will be assumed to take values in a fixed Banach space $Z$, unless otherwise specified. We shall generally \emph{not} highlight $Z$ in the name of the spaces we shall consider. } Obviously, the previous convention does not apply to functions of the form $\Kc_\Lc(\phi)$ or, more generally, to auxiliary functions $\psi$ or $H$. In this way, we shall deal with vector-valued Besov and Triebel--Lizorkin spaces without making the notation heavier; on the contrary, this should highlight that (most of the time) no extra effort is  needed to deal with the general case, at least for our current purposes. Considering vector-valued Triebel--Lizorkin spaces will make the comparison with local Hardy and bmo spaces more natural in~\cite{CalziRizzo} (when $Z$ is a Hilbert space). Similar extensions also apply to~\cite{BCP,Calzi2,Calzi3}, with minor exceptions.

In the following lemmas, we shall (partially) extend~\cite[Theorem 1]{Triebel2} to our setting. Notice, though, that since we lack a workable analogue of the Fourier transform, our results will be necessarily weaker.
We shall first state a definition to simplify the notation.

\begin{deff}
	Take $\eps\in (0,1)$ and $q\in (0,\infty]$. Then, we define $\Cc^q(\eps)$ as the space of $(f_j)\in L^1_\loc(G)^{\N}$ such that  
	\[
	\sup_{N\in \N, x\in G} \eps^{-N Q_*/q}\norm{  \chi_{[N,+\infty)\times B(x,\eps^{N})}(j,x') f_j(x') }_{L^q_{(j,x')}(\N\times G)}
	\]
	is finite, endowed with the corresponding quasi-norm. 
\end{deff}

\begin{oss}
	Take $\eps\in (0,1)$,  $a>0$, and $q\in (0,\infty]$. Take $(f_j)\in L^1_\loc(G)^\N$ and set $\mi\coloneqq \sum_{j\in\N} \delta_{\eps^{aj}}$. Define $g\colon (0,+\infty)\times G\to \C$ so that $g(t,x)=f_j(x)$ if $t= \eps^{a j}$, while $g(t,x)=0$ otherwise. Then,
	\[
	\eps^{a Q_*/q}\norm{g}_{\Cc^q(\mi,a)}\meg \norm{f_j(x)}_{\Cc^q_{(j,x)}(\eps^a)}\meg \norm{g}_{\Cc^q(\mi,a)}.
	\]
\end{oss}

Consequently, we may apply the results proved for the spaces $\Cc^q(\mi,a)$ also to the spaces $\Cc^q(\eps^a)$. Notice that this correspondence allows to get rid of the parameter $a$.
 
\begin{deff}\label{def:12}
	Take $\eps\in (0,1)$ and a formally self-adjoint weighted subcoercive operator $\Lc'$ on $G$. We define $\Psi_\eps(\Lc')$ the set of $(\psi_j)\in \Sr(G)^\N$ such that there is some bounded sequence $(\phi_j)$ of elements of $ C^\infty_c(\R)$ such that 
	\[
	0\not \in \overline{\bigcup_{j\Meg 1}\supp \phi_j }, \qquad 
	\inf_{\lambda\in \R} \sum_{j\in\N} \abs{\phi_j(\eps^{j }\lambda)}>0,
	\]
	and
	\[
	\psi_j= \Kc_{\eps^{j }\Lc'}(\phi_j )
	\]
	for every $j\in\N$.
	We define $\widetilde \Psi_\eps(\Lc')$ as the set of $(\psi_j)\in \Psi_\eps(\Lc')$ such that $\sum_{j\in\N} \psi_j=\delta_e$ in $\Sr'(G)$; equivalently, with the above notation, $\sum_{j\in\N} \phi_j(\eps^j\,\cdot\,)=1$ on $\sigma(\Lc')$.
	
	We shall simply write $\Psi_\eps$ and $\widetilde \Psi_\eps$ instead of $\Psi_\eps(\Lc)$ and $\widetilde \Psi_\eps(\Lc)$, respectively.
\end{deff}
 
\begin{oss}
	Take $\eps\in (0,1)$, a formally self-adjoint weighted subcoercive operator $\Lc'$ on $G$, $(\psi_j)\in \Psi_\eps(\Lc')$, and $\alpha\in \Z$. With some abuse of notation, we shall simply write $\Lc^\alpha \psi_j$ instead of $\Kc_{\Lc'}((\,\cdot\,)^\alpha \phi_j(\eps^j\,\cdot\,))$ for every $j\Meg 1$, where the $\phi_j$ are chosen as in Definition~\ref{def:12}. When $\alpha\in\N$, this is consistent with the usual notation, whereas when $\alpha\in -\N$, this leads to the useful equality $\psi_j= \Lc^{-\alpha} (\Lc^\alpha \psi_j)$ (for $j\Meg 1$).
\end{oss}

\begin{lem}\label{lem:35}
	Take $p,q\in (0,\infty]$, with $p<\infty$, $\alpha\in \R$, $\eps\in (0,1)$ $a>D/\min(p,q)$, $m_0,m_1\in \R$ so that $m_0<\alpha/\grado<m_1$.  
	Take $\mi\in \Mc_\Car$, a $\mi$-measurable 
	family $(\psi_{1,t})_{t\Meg 0}$ of elements of $\Sr(G)$ and $h_0 , h_1\in C^\infty_c(\R)$ such that  $\chi_{[-1,1]}\meg h_0\meg \chi_{[-1/\eps,1/\eps]}$, and $\chi_{[-1,1]\setminus [-\eps,\eps]}\meg h_1\meg \chi_{[-1/\eps,1/\eps]\setminus [-\eps^2,\eps^2]}$, and define $H_{k,j}\coloneqq\Kc_{\eps^j\Lc}(h_k )$ for $k=0,1$ and $j\in\Z$.  Assume that there is a constant $C>0$ such that 
	\begin{equation}\label{eq:10}
		\int_G \abs{(H_{0,0}*\psi_{1,t})(x)} (1+\abs{x}_*)^a\,\dd \beta(x)\meg C  t^{m_1}
	\end{equation}
	for every $t> 0$,
	\begin{equation}\label{eq:11}
		\int_G \abs{(H_{1, j}*\psi_{1,t})(x)} (1+\abs{x}_*/\eps^{j/\grado})^a\,\dd \beta(x)\meg C(t/\eps^j)^{m_1}
	\end{equation}
	for every $0<t\meg \eps^j$, $j\Meg 1$, 
	\begin{equation}\label{eq:12}
		\int_G \abs{(H_{1,j}*\psi_{1,t})(x)} (1+\abs{x}_*/\eps^{j/\grado})^a\,\dd \beta(x)\meg C (t/\eps^j)^{m_0}
	\end{equation}
	for every  $t\Meg \eps^j$, $j\Meg 1$, and
	\begin{equation}\label{eq:13}
		\int_G \abs{(H_{1,j}*\psi_{1,0})(x)} (1+\abs{x}_*/\eps^{j/\grado})^a\,\dd \beta(x)\meg C \eps^{-jm_0}
	\end{equation}
	for every $j\Meg 1$.
	In addition, take $(\psi_{2,j})\in \Psi_\eps$. 
	
	Then, there is a constant $C'>0$ such that
	\[
	\norm{f*\psi_{1,0}}_{L^p(G)}+\norm*{\norm{t^{-\alpha/\grado}   (f*\psi_{1,t})(x)}_{L^q_t(\mi )}}_{L^p_x(G)}\meg C' \norm*{\norm{\eps^{-j\alpha/\grado}  (f*\psi_{2,j})(x)}_{\ell^q_j(\N)}}_{L^p_x(G)}
	\]
	and
	\[
	\norm{f*\psi_{1,0}}_{L^\infty(G)}+ \norm{t^{-\alpha/\grado}   (f*\psi_{1,t})(x)}_{\Cc^q_{(t,x)}(\mi,1/\grado )} \meg C' \norm{\eps^{-j\alpha/\grado}  (f*\psi_{2,j})(x)}_{\Cc^q_{(j,x)}( \eps^{1/\grado}) }
	\]
	for  every $f\in \Ss'(G)$.
\end{lem}

This result (partially) extends one half of~\cite[Theorem 1]{Triebel2}. In the cited result, the $\psi_{1,t}$ are defined by means of their Fourier transform, which is required to be of class $C^\infty$. In addition, the Fourier transforms of the $\psi_{1,t}$, $t>0$, are required to be dilates of one another, so that condition~\eqref{eq:10} may be simply dropped, whereas conditions~\eqref{eq:11},~\eqref{eq:12}, and~\eqref{eq:13} may be considerably simplified, and only required to hold for $t=1$ (say). Another difference occurs when comparing condition~\eqref{eq:11} with~\cite[(16)]{Triebel2}, since in the cited result one actually requires something which could be rephrased, in our notation,  as
\begin{equation}\label{eq:11b}
	\int_G \abs{(\Lc^{-m_1} \psi_{1,1})(x)}(1+\abs{x}_*)^a\,\dd \beta(x) <\infty
\end{equation}
(where $\Lc$ must be interpreted as the usual Laplacian on $G=\R^n$). On the one hand, it is shown in the proof of~\cite[Theorem 1]{Triebel2} that this condition implies~\eqref{eq:11}. On the other hand, in our context condition~\eqref{eq:11b} need not even make sense, so that it seems more natural to replace it with~\eqref{eq:11}. Notice that, in doing so, we also allow for more general families $(\psi_{1,t})$, so that when $G=\R^n$ our result is neither more nor less general than~\cite[Theorem 1]{Triebel2}.

Next, we would like to stress the fact that we impose that the $\psi_{1,t}$ should be elements of $\Oc'_{C,L}(G)\cap \Oc'_{C,R}(G)$, whereas~\cite[Theorem 1]{Triebel2} allows them to be much less regular.  One the one hand, contrary to what is done in~\cite[Theorem 1]{Triebel2}, we need a result that applies to \emph{every} $f\in \Sr'(G)$, in order to establish the equality of the Triebel--Lizorkin spaces defined by different quasi-norms, and not only the equivalence of different quasi-norms on (already defined) Triebel--Lizorkin spaces. On the other hand, without the Fourier transform, we are not able to perform the extremely delicate verifications that would allow to check that the convolutions $f*\psi_{1,t}$ make sense when $f\in F^{p,q}_\alpha(\R^n)$ and $\psi_{1,t}$ is only required to satisfy the assumptions of~\cite[Theorem 1]{Triebel2}.

Finally, let us comment briefly on the fact that the second assertion seems asymmetrical, since the term $\norm{f*\psi_{1,0}}_{L^\infty(G)}$  in the left hand side  seems to have no analogue in the right hand side. This is only apparent, as the following remark shows. We kept the term $\norm{f*\psi_{1,0}}_{L^\infty(G)}$  on the left hand side since it may provide a stronger estimate in full generality.

\begin{oss} \label{oss:8}
	Take $q_0\in (0,1]$, $\eps\in (0,1)$, $t_0\in (0,+\infty)$, and $\phi\in C^\infty_c(\R)$. Set $\psi\coloneqq \Kc_\Lc(\phi)$. Then, there is a constant $C>1$ such that
	\[
	\frac{1}{C} \norm{f*\psi}_{L^\infty(G)}\meg t^{-Q_*/q}\sup_{x\in G}\norm{\chi_{B(x,t)} f*\psi}_{L^q(G)}\meg C\norm{f*\psi}_{L^\infty(G)}
	\]
	for every $q\in [q_0,\infty]$ and for every $t\in (0,t_0]$.
\end{oss}

\begin{proof}
	The second inequality is clear. Concerning the first one, take $a>D/p_0$ and observe that, by Lemma~\ref{lem:33}, there is a constant $C_1>0$ such that
	\[
	\abs{(f*\psi)(x)}\meg \Nc_{\infty,a,t_0}(f*\psi)(x)\meg C_1 \Nc_{q_0,a,t_0}(f*\psi)(x)\meg C_1 2^a \Nc_{q_0,a,t_0}(f*\psi)(y)
	\]
	for every $x\in G$ and for every $y\in B(e,t)$, so that there is a constant $C_2>0$ such that
	\[
	\begin{split}
		\norm{f*\psi}_{L^\infty(G)} &=\sup_{x\in G}\abs{(f*\psi)(x)}  \meg C_2 t^{-Q_*/q}\sup_{x\in G}\norm{\chi_{B(x,t)}  \Nc_{q_0,a,t_0}(f*\psi)}_{L^q(G)}.
	\end{split}
	\]
	The assertion then follows by means of Lemma~\ref{lem:2c} (applied with $\mi=\delta_t$).
\end{proof}

\begin{lem}\label{lem:41}
	Take $\eps\in (0,1)$, a formally self-adjoint weighted subcoercive operator $\Lc'$ on $G$, and  $(\psi_{j})\in \Psi_\eps(\Lc')$ (resp.\ $(\psi_j)\in \widetilde \Psi_\eps(\Lc')$). Then, there is a bounded family $(\eta_j)_{j\in \N}$ in $C^\infty_c(\R)$ such that
	\[
	\delta_e=  \sum_{j=0}^\infty \psi_{j}*\Kc_{\eps^j\Lc'}(\eta_j ) \qquad \text{(resp.\ $\delta_e=  \sum_{j=0}^\infty \psi_{j}$)}
	\]
	in $\Oc'_{C,L}(G)$  and in $\Oc'_{C,R}(G)$.  
\end{lem}

\begin{proof} 
	Take $(\phi_j)$ as in Definition~\ref{def:12}.
	Set $\phi\coloneqq \sum_{j\in \N} \abs{\phi_j(\eps^j\,\cdot\,)}^2$, and observe that $\inf_\R \phi>0$ because of the support restrictions on the $\phi_j$.
	 Define  $\eta_j\coloneqq \frac{  \overline{\phi_j} }{\phi(\eps^{-j}\,\cdot\,)}$ for every $j\in\N$. Since $\inf_\R \phi>0$ and since the $\phi(\eps^{-j}\,\cdot\,)$, $j\Meg 1$, are uniformly bounded in $C^\infty(\R\setminus \Set{0})$, it is readily seen that the $\eta_j$, $j\in \N$, are bounded in $C^\infty_c(\R)$. In addition, $\sum_{j\in\N} (\eta_j \phi_j)(\eps^j\,\cdot\,)= 1$ on $\R$. It is then easily seen that $(\Kc_{\eps^j \Lc'}(\phi_j\eta_j))\in \widetilde \Psi_\eps(\Lc')$, so that we may reduce to the case in which $(\psi_j)\in \widetilde\Psi_\eps(\Lc')$.

	Observe that $(\psi_j^*)\in \widetilde \Psi_\eps(\Lc')$ since $ \psi_{j}^*= \Kc_{\eps^j\Lc'}( \overline{\phi_j }) $. Since $u\mapsto u^*$ induces an isomorphism of $\Oc'_{C,R}(G)$ onto $\Oc_{C,L}'(G)$, it will therefore suffice to show that $\delta_e=  \sum_{j=0}^\infty \psi_{j} $ in $\Oc'_{C,R}(G)$. Next, notice that, if $h\in L^2(G)$, then $\sum_{j\in \N} h*\psi_{j} =\sum_{j\in \N} \phi_j(\eps^j {\Lc'}) h=h$ in the weak topology $\sigma(L^2(G),L^2(G))$, hence in $\Sr'(G)$. In order to conclude, it will therefore suffice to show that, if $B$ is a bounded subset of $\Sr(G)$, then the $\sum_{j=0}^k h*\psi_{j} $, $h\in B$, $k\in\N$, are uniformly bounded in $\Sr(G)$. This will follow  from Corollary~\ref{cor:12} if we show that the
	\[
	h_k\coloneqq\sum_{j=0}^k  \phi_j  (\eps^{j-k}\,\cdot\,),
	\] 
	$k\in\N$, are uniformly bounded in $C^\infty_c(\R)$.  
	Then, take $N\in\N$ so that $\supp \phi_0 \subseteq [-\eps^{-N},\eps^{-N}]$ and $\supp \phi_j \subseteq [-\eps^{-N},\eps^{-N}]\setminus (-\eps^{N},\eps^{N})$ for every $j\Meg 1$, and observe that $h_k=1$ on $ [-\eps^{N},\eps^{ N}]$, and that $h_k$ is supported in $ [-\eps^{-N},\eps^{-N}]$ for every $k\in\N$, so that $h^{(\ell)}_k=\sum_{j=k-2N+1}^k \eps^{\ell(j-k)} \phi_j  ^{(\ell)}(\eps^{j-k}\,\cdot\,)$ for every $\ell\Meg 1$. The assertion follows.
\end{proof}

\begin{proof}[Proof of Lemma~\ref{lem:35}.]
	Observe first that the mapping $[0,+\infty)\ni t\mapsto f*\psi_{1,t}\in C^\infty(G)$ is $\mi$-measurable, so that the mapping $(t,x)\mapsto (f*\psi_{1,t})(x)$ is $(\mi\otimes \beta)$-measurable for every $f\in \Sr'(G)$.
	
	\textsc{Step I} Take $(\eta_j)$ as in Lemma~\ref{lem:41} and set $\psi_{3,j}\coloneqq \Kc_{\eps^j\Lc}(\eta_j)$ for every $j\in\N$, so that
	\[
	f*\psi_{1,t} = \sum_{j=0}^\infty f*\psi_{2,j}*\psi_{3,j}*\psi_{1,t}
	\]
	(with convergence in $\Sr'(G)$ and in $C^\infty(G)$) for every $t\Meg 0$. Take $N\in\N$ and a bounded sequence $(\phi_j)$ of elements of $C^\infty_c(\R)$ such that $\psi_{2,j}=\Kc_{\eps^j \Lc}(\phi_j)$ and  $\Supp {\phi_j}\subseteq (-\eps^{-N},\eps^{-N})$ for every $j\in\N$, and such that $\Supp{\phi_j}\subseteq (-\eps^{-N},\eps^{-N})\setminus [-\eps^N, \eps^N]$ for every $j\Meg 1$. We may assume that the same restrictions on the support hold also for the $\eta_j$. In addition, by means of Corollary~\ref{cor:12} and Young's inequality, we may assume that conditions~\eqref{eq:10} to~\eqref{eq:13} hold with $H_{0,j}$ and $H_{1,j}$ replaced by $\psi_{3,j}$. We may also assume that $\int_G \abs{(\psi_{3,0}*\psi_{1,0})(x)} (1+\abs{x}_*)^a\,\dd \beta(x)\meg C $, since $\psi_{3,0}*\psi_{1,0}\in \Sr(G)$.
	Then, take $j\in\N$, $t\in (0,\eps^j]$,  
	and take $p_0\in (0,\min(p,q)]$ so that $a>D/p_0$. Observe that, by Lemma~\ref{lem:33}, there is a constant $C_1>0$ such that
	\[
	\Nc_{\infty, a,\eps^{j/\grado}} g\meg C_1 \Nc_{p_0, a,\eps^{j/\grado}} g
	\]
	for every $j\in\N$ and for every $g\in \Sr'(G)$  such that $g=g*H_{0,N+j}$. Observe that, by~\eqref{eq:10} and~\eqref{eq:11} (and the above remarks),  
	\[
	\begin{split}
		\abs{(f*\psi_{2,j}*\psi_{3,j}*\psi_{1,t})(x)}&= \abs{((f* \psi_{2,j})*(\psi_{3,j}*\psi_{1,t}))(x)}\\
		&\meg  \int_{G}  \abs{ (f* \psi_{2,j})(y) } \abs{(\psi_{3,j}*\psi_{1,t})(y^{-1}x)}\,\dd \beta(y)\\
		&\meg C_1  \Nc_{p_0,a,\eps^{j/\grado}}(f* \psi_{2,j})(x)\int_G (1+\abs{y^{-1}x}_*/\eps^{j/\grado})^a\abs{(\psi_{3,j}*\psi_{1,t})(y^{-1}x)}\,\dd \beta(y)\\
		&\meg C C_1(t/\eps^j)^{m_1}\Nc_{p_0,a,\eps^{j/\grado}}(f* \psi_{2,j})(x)
	\end{split}
	\]
	for every $x\in G$.  Consequently, 
	\[
	\begin{split}
		&\norm*{ t^{-\alpha/\grado} \sum_{\eps^j \Meg t}  \abs{(f*\psi_{2,j}*\psi_{3,j}*\psi_{1,t})(x)} }_{L^q_t(\mi)}  \meg C C_1\norm*{ t^{-\alpha/\grado} \sum_{\eps^j \Meg t} (t/\eps^j)^{ m_1} \Nc_{p_0,a,\eps^{j/\grado}}(f* \psi_{2,j})(x) }_{L^{q}_t(\mi)}\\
		&\qquad\meg C C_1\norm*{ \sum_{\eps^j \Meg t} (t/\eps^j)^{m_1-\alpha/\grado} \eps^{-j  \alpha /\grado} \Nc_{p_0,a,\eps^{j/\grado}}(f* \psi_{2,j})(x) }_{L^{q}_t(\mi)}.
	\end{split}
	\]
	Then, by means of Lemma~\ref{lem:25d} we see that there is a constant $C_2>0$ such that
	\[
	\norm*{ t^{-\alpha/\grado} \sum_{\eps^j \Meg t}  \abs{(f*\psi_{2,j}*\psi_{3,j}*\psi_{1,t})(x)} }_{L^q_t(\mi)}\meg C_2 \norm*{   \eps^{-j \alpha /\grado}\Nc_{p_0,a,\eps^{j/\grado}}(f* \psi_{2,j})(x) }_{\ell^{q}_j(\N)}.
	\]
	Consequently, by means of Corollary~\ref{cor:18}, we see that there is a constant $C_3>0$ such that
	\[
	\begin{split}
		\norm*{\norm*{ t^{-\alpha/\grado} \sum_{\eps^j \Meg t}  \abs{(f*\psi_{2,j}*\psi_{3,j}*\psi_{1,t})(x)} }_{L^q_t(\mi)}}_{L^p_x(G)} 
		&\meg C_3\norm*{\norm*{   \eps^{-j\alpha /\grado}  (f* \psi_{2,j})(x)  }_{\ell^{q }_j(\N)}}_{L^{p }_x(G)}.
	\end{split}
	\] 
	
	Now,  take   $t > \eps^j$ with $t\meg \rho\coloneqq \sup \supp \mi$. If $j\Meg 1$, then~\eqref{eq:12} (and the above remarks) shows that 
	\[
	\begin{split}
		\abs{(f*\psi_{2,j}*\psi_{3,j}*\psi_{1,t})(x)} 
		&\meg  \int_{G}  \abs{ (f* \psi_{2,j})(y) } \abs{(\psi_{3,j}*\psi_{1,t})(y^{-1}x)}\,\dd \beta(y)\\
		&\meg C_1 \Nc_{p_0,a,\eps^{j/\grado}}(f* \psi_{2,j})(x)\int_G (1+\abs{y^{-1}x}_*/\eps^{j/\grado})^a \abs{(\psi_{3,j}*\psi_{1,t})(y^{-1}x)}\,\dd \beta(y)\\
		&\meg (t/\eps^j)^{m_0} C C_1\Nc_{p_0,a,\eps^{j/\grado}}(f* \psi_{2,j})(x)
	\end{split}
	\]
	for every $x\in G$. Analogously, if $j=0$, then~\eqref{eq:10} (and the above remarks) shows that,   
	\[
	\begin{split}
		\abs{(f*\psi_{2,0}*\psi_{3,0}*\psi_{1,t})(x)} 
		&\meg C_1\Nc_{p_0,a,1}(f* \psi_{2,0})(x)\int_G (1+\abs{y}_*)^a \abs{(\psi_{3,0}*\psi_{1,t})(y)}\,\dd \beta(y)\\
		&\meg   C t^{m_1} C_1 \Nc_{p_0,a,1}(f* \psi_{2,0})(x)\\
		&\meg C C_1 \rho^{(m_1)_+ +(m_0)_-} t^{m_0} \Nc_{p_0,a,1}(f* \psi_{2,0})(x)
	\end{split}
	\]
	for every $x\in G$, since $1\meg t \meg \rho$.  
	Then, by means of Lemma~\ref{lem:25d} we see that there is a constant $C_4>0$ such that
	\[
	\begin{split}
		&\norm*{ t^{-\alpha/\grado} \sum_{ \eps^j < t}   \abs{(f*\psi_{2,j}*\psi_{3,j}*\psi_{1,t})(x)} }_{L^q_t(\mi)} \\
		&\qquad  \meg C C_1 \rho^{(m_1)_++(m_0)_-}\norm*{ t^{-\alpha/\grado} \sum_{ \eps^j < t}  (t/\eps^j)^{ m_0} \Nc_{p_0,a,\eps^{j/\grado}}(f* \psi_{2,j})(x) }_{L^{q}_t(\mi)}\\
		&\qquad\meg C C_1  \rho^{(m_1)_++(m_0)_-}\norm*{  \sum_{ \eps^j < t}  (\eps^j/t)^{\alpha/\grado-m_0} \eps^{-j\alpha /\grado}  \Nc_{p_0,a,\eps^{j/\grado}}(f* \psi_{2,j})(x) }_{L^{q}_t(\mi)}\\
		&\qquad\meg C_4\norm*{  \eps^{-j\alpha/\grado}  \Nc_{p_0,a,\eps^{j/\grado}}(f* \psi_{2,j})(x)  }_{\ell^{q}_j(\N )}.
	\end{split}
	\] 
	One may then conclude as before that there is a constant $C_5>0$ such that
	\[
	\norm*{\norm*{ t^{-\alpha/\grado} \sum_{ \eps^j < t }  \abs{(f*\psi_{2,j}*\psi_{3,j}*\psi_{1,t})(x)} }_{L^q_t(\mi)}}_{L^p_x(G)}\meg C_5\norm*{\norm*{  \eps^{-j\alpha/\grado}    (f* \psi_{2,j})(x) }_{\ell^{q }_j(\N )}}_{L^p_x(G)} .
	\]
	It remains to deal with the term $f*\psi_{1,0}$. However, the sum $\sum_{j\in \N} f*\psi_{2,j}*\psi_{3,j}*\psi_{1,0}$ may be dealt with as in the above computations, using~\eqref{eq:10} and~\eqref{eq:13} (instead of~\eqref{eq:12}), so that the first assertion follows.
	
	\textsc{Step II} In order to prove the second assertion, it essentially suffices to replace  Lemma~\ref{lem:25d} and Corollary~\ref{cor:18} with Lemmas~\ref{lem:25c} and~\ref{lem:2c}, respectively: the arguments of~\textsc{step I} proceed essentially unchanged for $j\Meg 1$, whereas the term corresponding to $j=0$ is even easier to deal with (using Remark~\ref{oss:8}). Nonetheless, the estimation of $\norm{f*\psi_{1,0}}_{L^\infty(G)}$ has to be modified as follows.  
	Observe first that, using~\eqref{eq:10} (and the fact that $\psi_{3,0}*\psi_{1,0}\in \Ss(G)$), we see that there is a constant $C_6>0$ such that
	\[
	\abs{f*\psi_{2,j}*\psi_{3,j}*\psi_{1,0}}\meg   C_6 \eps^{-j m_0}\Nc_{q,a,\eps^{j/\grado}}(f*\psi_{2,j})
	\]
	for every $j\in\N$. Then, observe that
	\[
	\Nc_{q,a,\eps^{j/\grado}}(f*\psi_{2,j})(x)\meg 2^a \Nc_{q,a,\eps^{j/\grado}}(f*\psi_{2,j})(y)
	\]
	for every $x,y $ with $d(x,y)<\eps^{j/\grado}$   and for every $j\in \N$. Consequently, there is a constant $C_7>0$ such that, for every $x\in G$,
	\[
	\begin{split}
	\Nc_{q,a,\eps^{j/\grado}}(f*\psi_{2,j})(x)&\meg C_7\eps^{-j Q_*/(q\grado)}\norm{\chi_{B(x,\eps^{j/\grado})}\Nc_{q,a,\eps^{j/\grado}}(f*\psi_{2,j})}_{L^q(G)}\\
		&\meg C_7 \eps^{j\alpha/\grado}\norm{\eps^{-j'\alpha/\grado} (f*\psi_{2,j'}) (x')}_{\Cc^q_{(j',x')}(\eps^{1/\grado})}
	\end{split}
	\]
	for every $j\Meg 1$, while $\Nc_{q,a,1}(f*\psi_{2,0})(x)\meg C_7 \norm{f*\psi_{2,0}}_{L^\infty(G)}$. This leads to the conclusion since $\alpha/\grado>m_0$, by assumption.
\end{proof}

The following result is an analogue of Lemma~\ref{lem:35} for Besov quasi-norms. As in the classical case (cf.~\cite[Theorem 3]{Triebel2}), one may impose weaker conditions on the $\psi_{1,t}$. 
The proof of   is similar to that of Lemma~\ref{lem:35} (in fact, slightly simpler since one may replace the use of maximal functions with Corollary~\ref{cor:4}), and is therefore omitted.

\begin{lem}\label{lem:35bis}
	Take $p,q\in (0,\infty]$,   $\alpha\in \R$, $\eps\in (0,1)$, $m_0,m_1\in \R$ so that $m_0<\alpha/\grado<m_1$.  
	Take $\mi\in \Mc_\Car$, a $\mi$-measurable family $(\psi_{1,t})_{t\Meg 0}$ of elements of $\Sr(G)$, and define $H_{k,j}$ and $\psi_{2,j}$ as in Lemma~\ref{lem:35}.  
	Assume that there is a constant $C>0$ such that  
	\begin{equation*}
		\norm{H_{0,0}*\psi_{1,t}}_{L^{\min(1,p)}(G)} \meg C  t^{m_1}
	\end{equation*}
	for every $t> 0$,
	\begin{equation*}
		\norm{H_{1, j}*\psi_{1,t} }_{L^{\min(1,p)}(G)}\meg C(t/\eps^j)^{m_1 } \eps^{j(1/p-1)_+Q_*/\grado}
	\end{equation*}
	for every $0<t\meg \eps^j$, $j\Meg 1$, 
	\begin{equation*}
		\norm{H_{1,j}*\psi_{1,t} }_{L^{\min(1,p)}(G)}\meg C (t/\eps^j)^{m_0 } \eps^{j(1/p-1)_+Q_*/\grado}
	\end{equation*}
	for every  $t\Meg \eps^j$, $j\Meg 1$, and
	\begin{equation*}
		\norm{H_{1,j}*\psi_{1,0}}_{L^{\min(1,p)}(G)}\meg C \eps^{-jm_0+j(1/p-1)_+ Q_*/\grado}
	\end{equation*}
	for every $j\Meg 1$.

	Then, there is a constant $C'>0$ such that
	\[
	\norm{f*\psi_{1,0}}_{L^p(G)}+\norm*{t^{-\alpha/\grado}\norm{   f*\psi_{1,t}}_{L^p (G)}}_{L^q_t(\mi )}\meg C' \norm*{\eps^{-j\alpha/\grado} \norm{ f*\psi_{2,j}}_{L^p(G)}}_{\ell^q_j(\N)}
	\]
	for  every $f\in \Sr'(G)$.
\end{lem}

In the following result we show how one may construct the $(\psi_{1,t})$ of Lemma~\ref{lem:35} using the functional calculus associated with some formally self-adjoint weighted subcoercive operator. Since we wish to prove the independence of Besov and Triebel--Lizorkin spaces from the chosen operator $\Lc$, we shall actually consider the functional calculus associated with a different operator $\Lc'$.

\begin{cor}\label{cor:11}
	Keep the notation of Lemma~\ref{lem:35}, and assume that $m_1\in \N$. Let $\Lc'$ be a formally self-adjoint weighted subcoercive operator of degree $\grado'$, and let $(\eta_t)_{t\Meg 0}$ be a bounded Borel measurable family of elements of $\Sr(\R)$ with $\eta_t^{(j)}(0)=0$ for $j=0,\dots, m_1-1$ and for every $t>0$. 
	Then, the conclusion of Lemmas~\ref{lem:35} and~\ref{lem:35bis} hold  choosing $\psi_{1,0}=\Kc_{\Lc'}(\eta_0)$, $\psi_{1,t}=\Kc_{t^{\grado'/\grado}\Lc'}(\eta_t )$ for $t\in \Supp \mi$, and $\psi_{1,t}=0$ otherwise.
\end{cor}

\begin{proof}
	It will suffice to show that~\eqref{eq:10} to~\eqref{eq:13} are satisfied with this choice of $(\psi_{1,t})$. This follows from Lemmas~\ref{lem:70} and~\ref{lem:40},  and Corollary~\ref{cor:13}.
\end{proof}

\begin{cor}\label{cor:16}
	Take $p,q\in (0,\infty]$, with $p<\infty$, $\alpha\in \R$, $\eps,\eps'\in (0,1)$, a formally self-adjoint weighted subcoercive operator $\Lc'$ with degree $\grado'$ on $G$, $(\psi_j)\in \Psi_\eps(\Lc)$, and $(\psi'_j)\in \Psi_{\eps'}(\Lc')$.  Then, there is a constant $C>0$ such that
	\[
	\norm*{\eps'^{-j\alpha/\grado'} \norm{ f*\psi'_j}_{L^p(G)}}_{\ell^q_j(\N)}\meg C \norm*{\eps^{-j\alpha/\grado} \norm{ f*\psi_j}_{L^p(G)}}_{\ell^q_j(\N)},
	\]
	\[
	 \norm*{\norm{\eps'^{-j\alpha/\grado'}  (f*\psi'_j)(x)}_{\ell^q_j(\N)}}_{L^p_x(G)}\meg C \norm*{\norm{\eps^{-j\alpha/\grado}  (f*\psi_j)(x)}_{\ell^q_j(\N)}}_{L^p_x(G)},
	\]
	and
	\[
	\begin{split}
	& \norm{\eps'^{-j\alpha/\grado'}  (f*\psi'_j)(x)}_{\Cc^q_{(j,x)}(\eps'^{1/\grado'})} \meg  C \norm{\eps^{-j\alpha/\grado}  (f*\psi_j)(x)}_{\Cc^q_{(j,x)}(\eps^{1/\grado})}
	\end{split}
	\]
	for  every $f\in \Sr'(G)$.
\end{cor}

\begin{proof}
	The assertion follows by means of  Corollary~\ref{cor:11} (with $\mi=\sum_{j=1}^\infty \delta_{\eps'^{j\grado/\grado'}}$).
\end{proof}

\begin{deff}\label{def:10}
	Take $p,q\in (0,\infty]$ and $\alpha\in \R$. Take $\eps\in (0,1)$ and $(\psi_{j})\in \Psi_\eps$. Then, we define $B^{p,q}_\alpha(G)$ as the space of $f\in \Sr'(G)$ such that
	\[
	\norm*{\eps^{-j\alpha/\grado} \norm{ f*\psi_{j}}_{L^p(G)}}_{\ell^q_j(\N)}<\infty,
	\]
	endowed with the corresponding quasi-norm. We define $\mathring B^{p,q}_\alpha(G)$ as the closure of $\Sr(G)$ in $B^{p,q}_\alpha(G)$.
	
	If  $p<\infty$, then we define $F^{p,q}_\alpha(G)$ as the space of $f\in \Sr'(G)$ such that
	\[
	\norm*{ \norm{\eps^{-j\alpha/\grado} (f*\psi_{j})(x)}_{\ell^q_j(\N)}}_{L^p_x(G)}<\infty,
	\]
	endowed with the corresponding quasi-norm. We define $F^{\infty,q}_\alpha(G)$ as the space of $f\in \Sr'(G)$ such that
	\[
	  \norm{\eps^{-j\alpha/\grado} (f*\psi_{j})(x)}_{\Cc^q_{(j,x)}(\eps^{1/\grado})}<\infty,
	\]
	endowed with the corresponding quasi-norm.
	We define $\mathring F^{p,q}_\alpha(G)$ as the closure of $\Sr(G)$ in $F^{p,q}_\alpha(G)$.
\end{deff}

Notice that Corollary~\ref{cor:16} shows that the spaces $B^{p,q}_\alpha(G)$ and $F^{p,q}_\alpha(G)$ are well defined and do not depend on the choice of the formally self-adjoint weighted subcoercive operator $\Lc$.
 
We now pass to the generalization of the second half of~\cite[Theorem 1]{Triebel2}. We shall only consider the case $q\Meg 1$ and $p>1$, since it is simpler to formulate and prove, and meets our immediate needs. Besides that,  for more general $p$ and $q$, we would not be able to prove the analogous result for every $f\in \Sr'(G)$, but only for $f\in F^{p,q}_\alpha(G)$.

\begin{lem}\label{lem:42}
	Take $p,q\in [1,\infty]$, $p\neq 1,\infty$, $\alpha\in \R$, and $\eps\in (0,1)$. Take a positive Radon measure $\nu$ on $(0,+\infty)$ such that there are $N\in\Z$ and $c>0$ such that $\nu((\eps^{j+1},\eps^j])\Meg c$ for every $j\Meg N+1$. Define $H_{0,j}$ and $H_{1,j}$ as in Lemma~\ref{lem:35}. 
	In addition, take a $\nu$-measurable family $(\psi_{1,t})_{t\Meg 0}$ of elements of $\Sr(G)$, and assume that there is a family $(\psi_{2,t})_{t\Meg 0}$ of elements of $L^\infty(G)$ such that $H_{1,N+j}=\psi_{1,t}*\psi_{2,t}$ for every $j\Meg 1$ and for every $t\in (\eps^{j+N+1},\eps^{j+N}]$, and such that $H_{0,N}=\psi_{1,0}*\psi_{2,0}$.
	Assume that there are $a>D$ and $C>0$ such that 
	\[
	\norm{\psi_{2,t}(1+\abs{\,\cdot\,}_*/t^{1/\grado})^a}_{L^\infty(G)}\meg C t^{- Q_*/\grado}
	\]
	for every $t>0$.
	
	Then, there is a constant $C'>0$ such that
	\[ 
	\norm{f}_{F^{p,q}_\alpha(G)}\meg C'\norm*{\norm{t^{-\alpha/\grado}( f*\psi_{1,t})(x) }_{L^q_t(\nu)}}_{L^p_x(G)} 
	\]
	and
	\[
	\norm{f}_{F^{\infty,q}_\alpha(G)}\meg C'\norm{f*\psi_{1,t}}_{L^\infty(G)}+C'\norm{t^{-\alpha/\grado}( f*\psi_{1,t})(x) }_{\Cc^q_{(t,x)}(\nu,1/\grado)} 
	\]
	for every $f\in \Sr'(G)$.
\end{lem}

Notice that the rather awkward requirement on the existence of the $(\psi_{2,j})$ is the only replacement we found for the Tauberian conditions which appear in~\cite[Theorem 1]{Triebel2}. This requirement can be easily met for our immediate needs; nonetheless, it substantially hinders any further applications of results of this kind (which, in the classical case, lead for instance to an efficient proof of the characterization by differences).

\begin{proof} 
	\textsc{Step I} Set $\psi_{3,0}\coloneqq H_{0,N}$ and $\psi_{3,j}\coloneqq H_{1,N+j}$ for every $j\Meg 1$, and observe that  $(\psi_{3,j})\in \Psi_\eps$.
	Then, observe  that
	\[
	f*\psi_{3,0}=  f*\psi_{1,0}*\psi_{2,0} ,
	\]
	so that
	\[
	\norm{f*\psi_{3,0}}_{L^p(G)}\meg \norm{f*\psi_{1,0}}_{L^p(G)} \norm{\psi_{2,0} }_{L^1(G)}
	\]
	by Young's inequality.  
	Next, for $j\Meg 1$ and $t\in (\eps^{j+N+1},\eps^{j+N}]$,
	\[
	f*\psi_{3,j}= f*\psi_{1,t}*\psi_{2,t} ,
	\]
	so that 
	\[
	\begin{split}
		\abs{f*\psi_{3,j}}&\meg C t^{-Q_* /\grado}\abs{f*\psi_{1,t}}*(1+\abs{\,\cdot\,}_*/t^{1/\grado})^{-a} \\ 
			&=C  \Nc_{1,a,t^{1/\grado}}(f*\psi_{1,t}).
	\end{split}
	\]
	Consequently,
	\[
	\begin{split}
	\norm*{\eps^{-j \alpha/\grado}(f*\psi_{3,j})(x)}_{\ell^q_j(\N^*)}&\meg C C_1 c^{-1/q} \norm*{\eps^{-j \alpha/\grado}\norm{ \Nc_{1,a,t^{1/\grado}}(f*\psi_{1,t})(x) }_{L^q_t( (\eps^{j+1+N},\eps^{j+N}], \nu)}  }_{\ell^q_j(\N^*)}\\
		&\meg  C C_1 c^{-1/q}\eps^{ N\alpha/\grado-\alpha_-/\grado} \norm{ t^{-\alpha/\grado} \Nc_{1,a,t^{1/\grado}}(f*\psi_{1,t})(x) }_{L^q_t(\nu)} .
	\end{split}
	\]
	Then, Corollary~\ref{cor:18} shows that there is a constant $C_2>0$ such that
	\[
	\norm*{\norm*{\eps^{-j \alpha/\grado}(f*\psi_{3,j})(x)}_{\ell^q_j(\N^*)}}_{L^p_x(G)}\meg C_2 \norm*{\norm{t^{-\alpha/\grado}( f*\psi_{1,t})(x) }_{L^q_t(\nu)}}_{L^p_x(G)} ,
	\] 
	whence the conclusion. 
	
	\textsc{Step II} The proof of the second assertion proceeds in a similar way, replacing Corollary~\ref{cor:18} with Lemma~\ref{lem:2c}.
\end{proof}

\begin{lem}\label{lem:42bis}
	Take $p,q\in [1,\infty]$,   $\alpha\in \R$, and $\eps\in (0,1)$. Take a positive Radon measure $\nu$ on $(0,+\infty)$ such that there are $N\in\Z$ and $c>0$ such that $\nu((\eps^{j+1},\eps^j])\Meg c$ for every $j\Meg N+1$. Define $H_{0,j}$ and $H_{1,j}$ as in Lemma~\ref{lem:35}. 
	In addition, take a Borel measurable family $(\psi_{1,t})_{t\Meg 0}$ of elements of $\Sr(G)$, and assume that there is a family $(\psi_{2,t})_{t\Meg 0}$ of elements of $\Sr(G)$ such that $H_{1,j+N}=\psi_{1,t}*\psi_{2,t}$ for every $j\Meg 1$ and for every $t\in (\eps^{j+N+1},\eps^{j+N}]$, and such that $H_{0,N}=\psi_{1,0}*\psi_{2,0}$.
	Assume that there is $C>0$ such that $\norm{\psi_{2,t} }_{L^1(G)}\meg C$ for every $t>0$.
	
	Then, there is a constant $C'>0$ such that
	\[ 
	\norm{f}_{B^{p,q}_\alpha(G)}\meg C\norm*{t^{-\alpha/\grado}\norm{ f*\psi_{1,t} }_{L^p (G)}  }_{L^q_t(\nu)}
	\]
	for every $f\in \Sr'(G)$.
\end{lem}

The proof is similar to that of Lemma~\ref{lem:42} (in fact, slightly simpler since one may replace the use of maximal functions with an application of Young's inequality) and is therefore omitted.

\begin{cor}\label{cor:14}
	Keep the notation of Lemmas~\ref{lem:42} and~\ref{lem:42bis}.
	Take a bounded family $(\phi_t)_{t\Meg 0}$ of elements of $\Sr(\R)$ such that, for some $c>0$,
	\[
	\abs{\phi_0 (x)}\Meg c
	\]
	for every $x\in \R$ with $\abs{x}\meg \eps^{-N-1}$, and
	\[
	\abs{\phi_t(x)}\Meg c
	\]
	for every $x\in \R$ with $  \eps^{3}\meg \abs{x}\meg   \eps^{-1 }$ and for every $j\Meg 1$ and for every $t\in (\eps^{j+N+1},\eps^{j+N}]$.
	Then one may take $\psi_{1,0}\coloneqq \Kc_\Lc(\phi_0)$, $\psi_{1,t}\coloneqq \Kc_{t\Lc}(\phi_t)$ for $t\in (0,\eps^{N+1}]$, and $\psi_{1,t }=0$ for $t>\eps^{N+1}$ in Lemmas~\ref{lem:42} and~\ref{lem:42bis}.
\end{cor}

\begin{proof}
	It suffices to set $\psi_{2,t}\coloneqq \Kc_{t\Lc}( h_1(\eps^{j+N}/t\,\cdot\,) /\phi_t   )$ for $j\Meg 1$ and $t\in (\eps^{j+N+1},\eps^{j+N}]$, and $\psi_{2,0}\coloneqq \Kc_\Lc(h_0(\eps^{N+1}\,\cdot\,)/\phi_0 )$,   and to apply Corollary~\ref{cor:12}.
\end{proof}

\begin{cor}\label{cor:15}
	Take $p,q\in [1,\infty]$, $\alpha\in \R$, $m\in \N$ with $m>\alpha/\grado$, and $\mi\in \Mc_\Samp$.  Then, there is a constant $C>0$ such that
	\[
	\frac{1}{C}\norm{f}_{B^{p,q}_\alpha(G)}\meg\norm{\ee^{-\Lc}f}_{L^p(G)}+ \norm*{t^{-\alpha/\grado}\norm{  (t\Lc)^m\ee^{-t \Lc}f  }_{L^p (G)}}_{L^q_t(\mi)}\meg C \norm{f}_{B^{p,q}_\alpha(G)},
	\]
	\[
	\frac{1}{C}\norm{f}_{F^{\infty,q}_\alpha(G)}\meg\norm{\ee^{-\Lc}f}_{L^\infty(G)}+ \norm{ t^{-\alpha/\grado} ((t\Lc)^m\ee^{-t \Lc})(x) }_{\Cc^q_{(t,x)}(\mi,1/\grado)}\meg C \norm{f}_{F^{\infty,q}_\alpha(G)}
	\]
	and, if $p\neq 1,\infty$,
	\[
	\frac{1}{C}\norm{f}_{F^{p,q}_\alpha(G)}\meg \norm{\ee^{-\Lc}f}_{L^p(G)}+\norm*{\norm{t^{-\alpha/\grado} ((t\Lc)^m\ee^{-t \Lc})(x) }_{L^q_t(\mi)}}_{L^p_x(G)}\meg C \norm{f}_{F^{p,q}_\alpha(G)}
	\]
	for every $f\in \Sr'(G)$.
\end{cor}

Notice that this result does \emph{not} prove that the spaces $B^{p,q}_\alpha(G)$ and $F^{p,q}_\alpha(G)$ defined in Definition~\ref{def:10} are the same as the spaces $B^{p,q}_\alpha(\beta)$ and $F^{p,q}_\alpha(\beta)$ defined in~\cite{BCP}, since in this result we only consider $f$ in the space $\Sr'(G)$ instead of the strictly larger space $\Sc'(G)$ defined in~\cite{BCP}.  Notice, in addition, that we could \emph{not} allow $f\in \Sc'(G)$ in the above statement, since $f*\psi_{j}$ would not be well defined, in general, for $(\psi_{j})\in \Psi_\eps$, $\eps\in (0,1)$. Actual equality follows from the fact that   the spaces defined in~\cite{BCP} are contained in $\Lc_\omega^{-\alpha-1}L^p(G)\subseteq \Sr'(G)$ by~\cite[Theorem 7.9]{BCP} for a sufficiently large $\omega\in\R$.

\begin{proof}
	The assertion follows from Corollaries~\ref{cor:11} and~\ref{cor:14}. 
\end{proof}

\begin{prop}\label{prop:20}
	Take $p,q\in (0,\infty)$ and $\alpha\in \R$. Then, $ B^{p,q}_\alpha(G)=\mathring B^{p,q}_\alpha(G)$ and $F^{p,q}_\alpha(G)=\mathring F^{p,q}_\alpha(G)$.
\end{prop}

\begin{proof}
	Take $f\in F^{p,q}_\alpha(G)$, $\eps\in (0,1)$, and $\phi\in C^\infty_c(\R)$ so that $\chi_{[-1,1]}\meg \phi\meg \chi_{[-1/\eps,1/\eps]}$. Define $\psi'_j\coloneqq \Kc_{\eps^j\Lc}(\phi)$ for every $j\in\N$, and define $\psi_0\coloneqq \psi'_0$ and $\psi_j\coloneqq \psi'_j-\psi'_{j-1}$ for every $j\Meg 1$, so that $(\psi_j)\in \widetilde \Psi_\eps$.
	Then, Lemma~\ref{lem:41} shows that $f=\lim\limits_{j\to \infty} f*\psi'_j$ in $\Sr'(G)$ and in $C^\infty(G)$. Let us prove that one has convergence in $F^{p,q}_\alpha(G)$. To this aim,  take $p_0<\min(p,q)$ and observe that $f*\psi'_j*\psi_{j'}=0$ if $j'\Meg j+2 $, whereas 
	\[
	f*\psi'_j*\psi_{j'}=\sum_{k=(j'-1)_+}^{\min(j,j'+1)} f*\psi_k*\psi_{j'} 
	\]
	if $j'\meg j+1$. Consequently, by Corollary~\ref{cor:12} we may find a constant  $C>0$ such that 
	\[
	\abs{f*\psi'_j*\psi_{j'}}\meg C \sum_{k=(j'-1)_+}^{j'+1} \Nc_{\infty, 2 D/p_0,\eps^{j'}}(f*\psi_k)
	\]
	for every $j,j'\in \N$. Since $\norm*{ \sum_{k=(j'-1)_+}^{j'+1} \Nc_{\infty, 2 D/p_0,\eps^{j'}}(f*\psi_k)(x) }_{L^{q,p}_{j',x}(\N,G)}$ is finite by Corollary~\ref{cor:18}, 
	the dominated convergence theorem ensures  that $f=\lim\limits_{j\to \infty} f*\psi'_j$ in $F^{p,q}_\alpha(G)$, whence our claim. Now, observe that there is $\eta\in C_c^\infty(\R)$ so that $\psi'_j=\psi'_j*\Kc_{\eps^j \Lc}(\eta)$ for every $j\in\N$, and that $F^{p,q}_\alpha(G)$ and $L^p(G)$ induce the same topology on the space $V_j\coloneqq \Sr'(G)*\Kc_{\eps^j \Lc}(\eta)$. Then, take $\tau_k\in C^\infty_c(G)$ so that $\chi_{B(e,k)}\meg \tau_k \meg \chi_{B(e,k+1)}$ for every $k\in\N$, and observe that we may reduce to proving that $[(f*\psi'_j)\tau_k]*\Kc_{\eps^j \Lc}(\eta)$ converges to $f*\psi'_j$ in the topology of $L^p(G)$. Observe first that it clearly converges in $\Sr'(G)$ and in $C^\infty(G)$. In addition,
	\[
	\abs{[(f*\psi'_j) \tau_k]*\Kc_{\eps^j\Lc}(\phi)}\meg \abs{f*\psi'_j}*\abs{\Kc_{\eps^j \Lc}(\phi)}\in L^p(G)
	\]
	for every $k\in\N$, thanks to Corollary~\ref{cor:4}. We may therefore conclude by dominated convergence. The assertion concerning Besov spaces is proved similarly.
\end{proof}

\begin{prop}\label{prop:20b}
	Take $p\in (0,\infty]$, $q\in (0,\infty)$, and $\alpha\in \R$. Then, $B^{p,q}_\alpha(G)\cap W^{\infty,p}(G)$ is dense in $B^{p,q}_\alpha(G)$.
	
	In particular, if $G$ is compact, then $ B^{p,q}_\alpha(G)=\mathring B^{p,q}_\alpha(G)$.
\end{prop}

\begin{proof}
	Take $f\in B^{p,q}_\alpha(G)$, and take $\eps,(\psi'_j)$, and $(\psi_j)$ as in the proof of Proposition~\ref{prop:20}. It will suffice to show that $f=\lim\limits_{j\to \infty} f*\psi'_j$ in $B^{p,q}_\alpha(G)$. The case $p<\infty$ has already been established in the proof of Proposition~\ref{prop:20}. Then, assume that $p=\infty$, and observe that $f*\psi'_j*\psi_{j'}=f*\psi_{j'}$ whenever $j\Meg j'+1$. It then follows that $f*\psi'_j*\psi_{j'}$ converges to $f*\psi_{j'}$ in $L^\infty(G)$ for every $j'\in\N$. Since, in addition, $\norm{f*\psi'_j*\psi_{j'}-f*\psi_{j'}}_{L^\infty(G)}\meg C \norm{f*\psi_{j'}}_{L^\infty(G)}$ for every $j,j'\in\N$, where $C=1+\sup_{j\in\N} \norm{\psi'_j}_{L^1(G)}$ is finite by Corollary~\ref{cor:12}, the conclusion follows from the dominated convergence theorem.
\end{proof}
  
\begin{teo} \label{teo:17}
	Take $p,q\in (0,+\infty]$, $\alpha,\alpha'\in \R$, and $s>D(1/\min(p,q)+1/2)$. Take $\omega\in \R$ so that $\sigma(\Lc_\omega)\subseteq (0,+\infty)$. Then,  the $(1+\abs{\gamma})^{-s}\Lc_\omega^{(\alpha-\alpha')/\grado+i\gamma }$, $\gamma\in \R$, induce equicontinuous isomorphisms of $B^{p,q}_\alpha(G)$ onto $B^{p,q}_{\alpha'}(G)$ and  of $F^{p,q}_\alpha(G)$ onto $F^{p,q}_{\alpha'}(G)$.
\end{teo}

Notice that, unlike in the analogous~\cite[Theorem 7.9]{BCP}, this time the operator $\Lc_\omega^{(\alpha-\alpha')/\grado}$  is well defined on the whole of $\Sr'(G)$, since its convolution kernel belongs to $\Ec'(G)+\Sr(G)\subseteq \Oc_{C,R}'(G)\cap \Oc_{C,L}'(G)$. 
In addition, in the case of Besov spaces, one may  take $s>D(1/p-1/2)_+$ as well. Nonetheless, the control on $s$ is unlikely to be optimal, in general.

\begin{proof}
	Observe first that $\Lc_\omega^{(\alpha-\alpha')/\grado} \Lc_\omega^{(\alpha'-\alpha)/\grado}$ is the identity operator, so that it will suffice to prove that there are continuous inclusions.  This, in turn, follows from Corollary~\ref{cor:5} and Lemmas~\ref{lem:34} and~\ref{lem:34b}.  
\end{proof}

\begin{prop}\label{prop:11}
	Take $p,q\in (0,+\infty]$, $\alpha\in \R$, and $X\in \gf$. Then, the mapping $f\mapsto X f$ induces a continuous linear mapping of $B^{p,q}_\alpha(G)$ into $B^{p,q}_{\alpha-\deg(X)}(G)$ and of $F^{p,q}_\alpha(G)$ into $F^{p,q}_{\alpha-\deg(X)}(G)$.
\end{prop}

\begin{proof}
	This follows from Lemmas~\ref{lem:35bis} and~\ref{lem:35}, choosing  $\mi =\sum_{j\in\N} \delta_{\eps^j}$, $\psi_{1,\eps^j}= \eps^{-j\deg(X)/\grado}X H_{1,j}$ for $j\Meg 1$ and $\psi_{1,0}=X H_{0,0}$.
\end{proof}

\begin{teo}\label{teo:16}
	Take $p,q\in (0,+\infty]$, $\alpha\in \R$, and $h\in \N^*$. Let $(X_j)_{j\in J}$ be a minimal basis of $\gf$, and set $d_j\coloneqq \deg(X_j)$ for every $j\in J$. Take a $\phi\in C^\infty_c(\R)$ which vanishes nowhere in a neighbourhood of $0$, and take $f\in \Sr'(G)$ with $f*\Kc_\Lc(\phi)\in L^p(G)$. Then, $f\in B^{p,q}_{\alpha+h\dd}(G)$ if and only if $X_j^{h\dd/d_j} f\in B^{p,q}_{\alpha}(G)$ for every $j\in J$. Analogously, $f\in F^{p,q}_{\alpha+h\dd}(G)$ if and only if $X_j^{h\dd/d_j}f\in F^{p,q}_{\alpha}(G)$ for every $j\in J$.
\end{teo}

\begin{proof}
	One implication follows from Proposition~\ref{prop:11}.
	For the other implication, observe first that we may assume that $\Lc=\sum_{j\in J} (X_j^\dag)^{h\dd/d_j} X_j^{h\dd/d_j}=\sum_{j\in J} (-1)^{h\dd/d_j} X_j^{2 h\dd/d_j}$, thanks to Proposition~\ref{prop:4}, so that $\grado=2h\dd$. Take $(\psi_j)\in \Psi_\eps$ so that $\psi_0=\Kc_\Lc(\phi)$, and observe that $(\psi'_j)\in \Psi_\eps$, where $\psi'_0=\psi_0$ and $\psi'_j= \eps^j\Lc \psi_j$ for every $j\Meg 1$.  Then, Corollary~\ref{cor:16} and Proposition~\ref{prop:11} show that there is a constant $C_1>0$ such that, setting $\ell=\min(1,p,q)$,
	\[
	\begin{split}
		\norm*{ \norm{ \eps^{-j(\alpha+h\dd)/\grado} (f*\psi_{j})(x)}_{\ell^q_j(\N^*)}  }_{L^p_x(G)}^\ell&\meg C_1\norm*{ \norm{ \eps^{-j(\alpha-h\dd)/\grado} (f*\Lc \psi_j)(x)}_{\ell^q_j(\N^*)}  }_{L^p_x(G)}^\ell\\
			&=C_1\norm*{ \norm{ \eps^{-j(\alpha-h\dd)/\grado } ((\Lc f)* \psi_j)(x)}_{\ell^q_j(\N^*)}  }_{L^p_x(G)}^\ell\\
			&\meg C_1^2\sum_{j\in J} \norm*{ X_j^{2 h \dd/d_j} f}_{F^{p,q}_{\alpha-h\dd}(G)}^\ell\\
			&\meg  C_1^3\sum_{j\in J}\norm*{ X_j^{ h \dd/d_j} f}_{F^{p,q}_{\alpha }(G)}^\ell,
	\end{split}
	\]
	so that the (second) assertion follows since $f*\psi_0\in L^p(G)$ by assumption, when $p<\infty$. The case $p=\infty$ and the first assertion are proved similarly.
\end{proof}

	\section{Discretization}\label{sec:4}
	
	In this section we deal with an analogue of the discretization discussed in~\cite{FrazierJawerth}. We emphasize the fact that our techniques do \emph{not} allow to recover the full strength of the results of~\cite{FrazierJawerth}, since in this more general context we are not able to show that Besov and Triebel--Lizorkin spaces are \emph{retracts} of suitable spaces of sequences. In fact we shall only show that they embed  naturally as closed vector subspaces of the analogous spaces of sequences (which may be considered as some sort of sampling theorems), and that there are `atomic decomposition' mappings from these spaces of sequences into Besov and Triebel--Lizorkin spaces. These weaker results are nonetheless sufficient to recover some of the basic properties of Besov and Triebel--Lizorkin spaces.
	
	We begin with introducing a `discretized' notion of lattices, and then the relevant spaces of sequences.
	
	\begin{deff}
		Take $\eps\in (0,1)$, $\delta>0$, and $R\Meg 2$. Given $K\subseteq \N\times \N$, we say that a family $(x_{j,k})\in G^{K}$ is a reduced $(\eps,\delta,R)$-lattice on $G$ if the $B(x_{j,k},\delta\eps^j)$, $j\in K_k$, are pairwise disjoint and the $B(x_{j,k}, \delta R\eps^j)$, $j\in K_k$, cover $G$ for every $k\in\N$, where $K_k\coloneqq \Set{j\in\N\colon (j,k)\in K}$.
	\end{deff}
	
	Existence of reduced $(\eps,\delta,2)$-lattices is trivial, since it suffices to take $(x_{j,k})$ so that $(x_{j,k})_k$ is a maximally $2\delta\eps^j$-separated for every $j\in \N$. If $G$ is not compact, we may therefore assume that $K=\N\times \N$.

	In the next lemma we show the existence of suitable partitions of the unity associated with reduced $(\eps,\delta, R)$-lattices.

	\begin{lem}\label{lem:43}
		Take $\eps\in (0,1)$, $\delta>0$, and $R\Meg 2$ and a reduced $(\eps,\delta,R)$-lattice $(x_{j,k})_{(j,k)\in K}$ on $G$. Then, there is a family $(\tau_{j,k})_{(j,k)\in K}$ of positive elements of $C^\infty_c(G)$ such that the following hold:
		\begin{enumerate}
			\item[\textnormal{(1)}] $\sum_{k\colon (j,k)\in K} \tau_{j,k}=1$ locally uniformly for every $j\in\N$;
			
			\item[\textnormal{(2)}] $\tau_{j,k}$ is supported in $B(x_{j,k},2\delta R\eps^j)$ for every $(j,k)\in K$;
			
			\item[\textnormal{(3)}] for every  $\lambda>0$ there is a constant $C>0$ such that
			\[
			\abs{X Y^R \tau_{j,k}} \meg C \abs{X}\abs{Y} \eps^{-j(\deg(X)+\deg(Y))}
			\] 
			for every $(j,k)\in K$.
		\end{enumerate} 
	\end{lem}
	  
	 We shall need the following result; cf.~\cite[???]{CalziRizzo} for a proof.
	 
	 \begin{cor}\label{cor:21}
	 	Take $\kappa_2>\kappa_1>0$. There is a family $(\tau_t)_{t>0}$ of elements of $C^\infty_c(G)$ such that the following hold:
	 	\begin{enumerate}
	 		\item[\textnormal{(1)}] $\chi_{B(e,\kappa_1 t)}\meg \tau_t\meg \chi_{B(e,\kappa_2t)}$ for every $t\in (0,1]$;
	 		
	 		\item[\textnormal{(2)}] for every $\lambda>0$ there is a constant $C_\lambda>0$ such that 
	 		\[
	 		\abs{X Y^R \tau_t}\meg C_\lambda\abs{X}\abs{Y} \min(1,t)^{-\deg(X)-\deg(Y)}
	 		\]
	 		for every $X,Y\in U_\lambda$ and $t>0$.
	 	\end{enumerate}
	 \end{cor}

	\begin{proof}[Proof of Lemma~\ref{lem:43}.]
		Take $(\tau_t)_{t>0} $ as in Corollary~\ref{cor:21} with $\kappa_1= 1$ and $\kappa_2=2$, and define 
		\[
		\psi_{j,k}= \tau_{\delta R \eps^j  }(x_{j,k}^{-1}\,\cdot\,) ,
		\]
		so that $\chi_{B(x_{j,k},\delta R\eps^j)}\meg \psi_{j,k}\meg \chi_{B(x_{j,k},2\delta R\eps^j)}$ and,  for every $\lambda>0$, there is a constant $C_{\lambda}>0$ such that  
		\[
		\norm{XY^R \psi_{j,k}}_{L^\infty(G)}\meg  C_{\lambda}\abs{X}\abs{Y} \eps^{-j(\deg(X)+\deg(Y))}
		\]
		for every $X,Y\in U_\lambda$ and for every $(j,k)\in K$.  
		
		Then, it suffices to set $\psi_j\coloneqq \sum_k \psi_{j,k}\Meg 1$ and to define $\tau_{j,k}\coloneqq \psi_{j,k}/\psi_j$. The assertion follows since by~\cite[Lemma 4.3]{Calzi2} there is a constant $N'\in\N$ such that $\sum_k \chi_{B(x_{j,k},2\delta R\eps^j)}\meg N'$ on $G$  for every $j \in \N$.
	\end{proof}
	
	\emph{As before, for simplicity we shall work below as if the families in the spaces $b^{p,q}_\alpha$ and $f^{p,q}_\alpha$ defined below were scalar, but the reader may interpret them as families with values in a more general Banach space $Z$ (as for the elements of Besov and Triebel--Lizorkin spaces) if so desired.}
	
	\begin{deff}
		Take $p,q\in (0,\infty]$, $\alpha\in \R$, a subset $K$ of $\N\times \N$ and $\eps\in (0,1)$. Then, we define  $b^{p,q}_{\alpha}(\eps)$ as the space of $\lambda\in \C^{\N\times \N}$ such that
		\[
		\norm{\eps^{j (Q_*/p-\alpha)}\lambda_{j,k}  }_{\ell^{p,q}_{k,j}(\N,\N)}<\infty,
		\]
		endowed with the corresponding quasi-norm.  We define $b^{p,q}_\alpha(\eps,K)$ as the space of $\lambda\in \C^K$ such that $\lambda^0\in b^{p,q}_\alpha(\eps)$, endowed with the corresponding quasi-norm, where $\lambda^0\in \C^{\N\times \N}$ is chosen so that $\lambda^0_{j,k}=\lambda_{j,k}$ when $(j,k)\in K$, while $\lambda^0_{j,k}=0$ otherwise. 
		
		In addition, take $\delta>0$, $R\Meg 2$, and a reduced $(\eps,\delta,R)$-lattice $(x_{j,k})_{(j,k)\in K'}$. Then, for $p\in (0,\infty)$ we define $f^{p,q}_\alpha((x_{j,k}))$ as the space of $\lambda\in \C^{K'}$ such that
		\[
		\norm{ \chi_{K'}(j,k) \eps^{-j \alpha} \lambda_{j,k} \chi_{B(x_{j,k}, \delta  \eps^j)}(x) }_{L^{q,q,p}_{j,k,x}(\N,\N,G)}<\infty,
		\]
		(with some abuse of notation) endowed with the corresponding quasi-norm. Finally, we define $f^{\infty,q}_\alpha((x_{j,k}))$ as the space of $\lambda\in \C^{K'}$ such that 
		\[
		\norm*{ \norm{ \chi_{K'}(j,k)\eps^{-j \alpha} \lambda_{j,k} \chi_{B(x_{j ,k}, \delta  \eps^{j})}(x)}_{\ell^q_k(\N)} }_{\Cc^{q}_{(j,x)}(\eps)}<\infty,
		\]
		endowed with the corresponding quasi-norm.
		
		We denote with  $\mathring b^{p,q}_\alpha(\eps)$, $\mathring b^{p,q}_\alpha(\eps,K)$, and $\mathring f^{p,q}_\alpha((x_{j,k}))$ the closures of $\C^{(\N\times \N)}$, $\C^{(K)}$, and $\C^{(K')}$ in  $  b^{p,q}_\alpha(\eps )$, $b^{p,q}_\alpha(\eps,K)$, and $f^{p,q}_\alpha((x_{j,k}))$, respectively.
	\end{deff}
	
	Observe that, even though $(x_{j,k})$ determines $\eps$ uniquely, it does not determine $\delta,R$. As Remark~\ref{oss:6} below shows, nonetheless, the space $f^{p,q}_\alpha((x_{j,k}))$ does not depend on $\delta,R$ (its norm does, though).
	
	We also observe that $b^{p,q}_\alpha(\eps,K)$ and $\mathring b^{p,q}_\alpha(\eps,K)$ are retracts of $b^{p,q}_\alpha(\eps)$ and $\mathring b^{p,q}_\alpha(\eps)$, respectively  (the retraction being the restriction mapping $b^{p,q}_\alpha(\eps)\to b^{p,q}_\alpha(\eps,K)$ and the corresponding section being the mapping $\lambda \mapsto \lambda^0$), so that basically every assertion concerning the former spaces follow from the corresponding one on the latter spaces. In other words, we shall generally prove our assertions only for the spaces $b^{p,q}_\alpha(\eps)$ and $\mathring b^{p,q}_\alpha(\eps)$. In addition, for simplicity, we shall \emph{not} write $ \chi_{K'}(j,k)$ in the quasi-norms appearing in the definition of  $f^{p,q}_\alpha((x_{j,k}))$; in other words, we shall convene to consider $0$ the terms in those quasi-norms which are not defined. 
	
	\begin{oss}\label{oss:6}
		Take $p,q\in (0,+\infty]$ with $p<\infty$, $\alpha\in \R$, $\eps\in (0,1)$, $\delta_0>0$, $R\Meg 2$, and $c>1$. Then, there is a constant $C>0$ such that
		\[
		\norm{ \eps^{-j \alpha} \lambda_{j,k} \chi_{B(x_{j,k}, c\delta \eps^j)}(x) }_{L^{q,q,p}_{j,k,x}(\N,\N,G)}\meg C \norm{ \eps^{-j \alpha} \lambda_{j,k} \chi_{B(x_{j,k}, (\delta/c) \eps^j)}(x) }_{L^{q,q,p}_{j,k,x}(\N,\N,G)}
		\]
		and 
		\[
		\norm*{ \norm{\eps^{-j \alpha} \lambda_{j,k} \chi_{B(x_{j,k},c \delta  \eps^j)}(x)}_{\ell^q_k(\N)} }_{\Cc^{q}_{(j,x)}(\eps)}\meg C\norm*{ \norm{\eps^{-j \alpha} \lambda_{j,k} \chi_{B(x_{j,k}, (\delta /c) \eps^j)}(x)}_{\ell^q_k(\N)} }_{\Cc^{q}_{(j,x)}(\eps)}
		\]
		for every reduced $(\eps,\delta,R)$-lattice $(x_{j,k})_{(j,k)\in K}$, with $\delta\in (0,\delta_0]$, and for every $\lambda\in \C^{K}$. 
	\end{oss}
	
	In particular, it does not matter which radius we take for the balls $B(x_{j,k}, \delta \eps^j)$ appearing in the definition of the spaces $f^{p,q}_\alpha((x_{j,k}))$, as long as it is comparable to $\delta \eps^j$. Taking exactly $\delta \eps^j$ makes the balls $B(x_{j,k}, \delta \eps^j)$  disjoint, so that we may replace the inner $\ell^q_k(\N)$ norm with a sum in $k$ (even without taking modules); taking $R\delta \eps^j$ makes the quasi-norm of  $f^{p,q}_\alpha((x_{j,k}))$ better adapted to (the proof of) Proposition~\ref{prop:14}.
	 
	\begin{proof}
		For simplicity, we write $c_1\coloneqq \delta/c$ and $c_2\coloneqq c\delta$, so that $0<c_1<c_2$.  
		Then, by~\cite[Lemma 4.3]{Calzi2}, we may find $N\in\N$ and, for every $j\in\N$, a partition $K_{j,1},\dots K_{j,N}$ of $K_j\coloneqq \Set{k\in \N\colon (j,k)\in K}$ such that the $B(x_{j,k}, (\delta R+c_1+c_2)\eps^j)$, $k\in K_{j,h}$, are pairwise disjoint for every $h=1,\dots, N$. Then, given $ \lambda\in \C^{K}$, define $\lambda^{(h)}$ so that $\lambda^{(h)}_{j,k}=\chi_{K_{j,h}}(k)\lambda_{j,k}$ for every $(j,k)\in K$ and for every $h=1,\dots, N$, so that
		\[
		\norm{ \eps^{-j \alpha} \lambda_{j,k} \chi_{B(x_{j,k}, c_2\eps^j)}(x) }_{L^{q,q,p}_{j,k,x}(\N,\N,G)}^\ell\meg \sum_{h=1}^N \norm{ \eps^{-j \alpha} \lambda^{(h)}_{j,k} \chi_{B(x_{j,k}, c_2\eps^j)}(x) }_{L^{q,q,p}_{j,k,x}(\N,\N,G)}^\ell
		\]
		where $\ell\coloneqq \min(p,q,1)$. 
		If we define  
		\[
		f^{(h)}(j,x)\coloneqq  \eps^{-j \alpha}\sum_{k\in K_{j,h}} \lambda_{j,k}\chi_{B(x_{j,k},c_2 \eps^j)}(x)
		\]
		and 
		\[
		g^{(h)} (j,x)\coloneqq  \eps^{-j \alpha}\sum_{k\in K_{j,h}} \lambda_{j,k}\chi_{B(x_{j,k},c_1 \eps^j)}(x)
		\]
		for every $j\in \N$ and for every $x\in G$,
		then clearly
		\[
		\abs{f^{(h)}(j,x)}=  \left( \dashint_{B(x_{j,k},c_1 \eps^j)} \abs{g ^{(h)}(j,y)}^{\ell/2}\,\dd \beta(y)\right) ^{2/\ell}
		\]
		for every $j\in \N$, for every $x\in B(x_{j,k},c_2 \eps^j)$, and for every $k\in K_{j,h}$, so that there is a constant $C_2>0$ such that
		\[
		\abs{f^{(h)}(j,x)}\meg C_2[\Nc_{\ell/2, 4D/\ell,  \eps^j} g^{(h)}(j,\,\cdot\,)] (x)
		\]
		for every $j\in \N$ and for every $x\in G$. Then, by Corollary~\ref{cor:18},  there is a constant $C_3>0$ such that
		\[
		\begin{split}
			\norm{ \eps^{-j \alpha} \lambda^{(h)}_{j,k} \chi_{B(x_{j,k}, c_2\eps^j)}(x) }_{L^{q,q,p}_{j,k,x}(\N,\N,G)}&=\norm{f^{(h)}}_{L^{q,p}(\N,G)}\\
				&\meg C_3\norm{g^{(h)}}_{L^{q,p}(\N,G)}\\
				&\meg \norm{ \eps^{-j \alpha} \lambda_{j,k} \chi_{B(x_{j,k}, c_1\eps^j)}(x) }_{L^{q,q,p}_{j,k,x}(\N,\N,G)},
		\end{split}
		\]
		so that the first assertion follows. 
		In order to prove the second assertion, it suffices to repeat the previous computations, essentially replacing  Corollary~\ref{cor:18}  with Lemma~\ref{lem:2c}.
	\end{proof}
	
	\begin{prop}\label{prop:12b}
		Take $p,q\in (0,\infty]$, $\alpha\in \R$, $\eps\in (0,1)$, $\delta_0>0$, and $R\Meg 2$. Then, there is a constant $C>0$ such that, if $p<\infty$,
		\[
		\frac{1}{C}\norm{\lambda}_{b^{p,\max(p,q)}_\alpha(\eps,K)}\meg \delta^{-Q_*/p}\norm{\lambda}_{f^{p,q}_\alpha((x_{j,k}))}\meg C  \norm{\lambda}_{b^{p,\min(p,q)}_\alpha(\eps,K)}
		\]
		and, if $p=\infty$,
		\[
		\frac{1}{C}\delta^{Q_*/q}\norm{\lambda}_{b^{\infty,\infty}_\alpha(\eps,K)}\meg \norm{\lambda}_{f^{\infty,q}_\alpha((x_{j,k}))}\meg C  \norm{\lambda}_{b^{\infty,q}_\alpha(\eps,K)},
		\]
		for every reduced $(\eps,\delta,R)$-lattice $(x_{j,k})_{(j,k)\in K}$ in $G$, with $\delta\in (0,\delta_0]$, and for every $\lambda\in \C^{K}$.
	\end{prop}
	
	Notice that the dependence on $\delta$ in the inequalities in the case $p=\infty$ cannot be improved, as one may see considering the case of $\lambda_{j,k}=\delta_{j,0}\delta_{k,0}$ (for the left inequality, assuming that $(0,0)\in K$) and the case of $\delta_{j,k}=\delta_{j,0}$ (for the right inequality).

	\begin{proof}
		For every $\lambda\in \C^{K}$ define $f_\lambda(j,x)\coloneqq \eps^{-j\alpha}\sum_{k} \chi_{B(x_{j,k},\delta   \eps^j)}(x) \lambda_{j,k}$.
		Observe first that there is a constant $C_1>0$ such that
		\[
		\frac{1}{C_1}\norm{\lambda}_{b^{p,q_1}_\alpha(\eps,K)}\meg \delta^{-Q_*/p} \norm*{f_\lambda (j,x) }_{L^{p,q_1}_{x,j}(G,\N)}\meg C_1 \norm{\lambda}_{b^{p,q_1}_\alpha(\eps,K)}
		\]
		for every $q_1\in (0,\infty]$.  
		Then, assume that $p<\infty$ and observe that, by Minkowski's integral inequality and the (contractive) inclusion $\ell^{q/p}(\N)\subseteq \ell^{\max(1,q/p)}(\N^*)$, 
			\[
			\begin{split}
					\norm*{   \norm{f_\lambda(j,x)}_{L^p_x(G)} }_{\ell^{\max(p,q)}_j(\N)}&= \norm*{    \int_G \abs{f_\lambda(j,x)}^p\,\dd \beta(x)  }_{\ell^{\max(1,q/p)}_j(\N)}^{1/p}\\
						&\meg \left(\int_G\norm {  \abs{f_\lambda(j,x)}^p }_{\ell^{\max(1,q/p)}_j(\N)}\,\dd \beta(x) \right)^{1/p}\\
						&\meg \left(\int_G\norm {  \abs{f_\lambda(j,x)}^p }_{\ell^{ q/p}_j(\N)}\,\dd \beta(x)\right)^{1/p} \\
						&= \norm*{ \norm{ f_\lambda(j,x) }_{\ell^q_j(\N)}}_{L^p_x(G)}, 
				\end{split}
			\]
			whence the left inequality in this case.
			If, otherwise, $p=\infty$, then  
			\[
			\eps^{-j\alpha}\abs{\lambda_{j,k}}=\abs{f_\lambda(j,x)}
			\]
			for every $x \in B(x_{j,k},\delta\eps^j)$. Therefore, there is a constant $C_2>0$ such that (cf.~\cite[Lemma 3.2]{Calzi3})
			\[
			\eps^{-j\alpha}\abs{\lambda_{j,k}}\meg C_2 (\delta\eps^j)^{-Q_*/q} \norm{ \chi_{B(x_{j,k},\delta \eps^j)} f_\lambda(j,\,\cdot\,)  }_{L^q(G)}\meg C_2\delta^{-Q_*/q}   \norm{f_\lambda(j,x)}_{\Cc^q_{(j,x)}(\eps)}
			\]
			for every $(j,k)\in K$, whence the  left inequality also in this case.
			
			For what concerns the right inequality, observe that we may assume that $\min(p,q)<\infty$, for otherwise the assertion is trivial. Then, by the (contractive) inclusion $\ell^{\min(p,q)}(\N)\subseteq \ell^q(\N)$ and Minkowski's integral inequality,
			\[
			\begin{split}
					\norm*{ \norm{ f_{\lambda}(j,x) }_{\ell^q_j(\N)}}_{L^p_x(G)}&\meg \norm*{ \norm{ f_{\lambda}(j,x)}_{\ell^{\min(p,q)}_j(\N)}}_{L^p_x(G)}\\
						&= \norm*{ \sum_{j\in\N}  \abs{f_{\lambda}(j,x)}^{\min(p,q)}}_{L^{p/\min(p,q)}_x(G)}^{1/\min(p,q)}\\
						&\meg  \left(\sum_{j\in\N}  \norm{f_{\lambda}(j,x)}_{L^p_x(G)}^{\min(p,q)} \right)^{1/\min(p,q)}\\
						&=\norm*{ \norm{f_{\lambda}(j,x)}_{L^p_x(G)} }_{\ell^{\min(p,q)}_j(\N)},
				\end{split}
			\]
			whence the the conclusion (observing, when $p=\infty$, that the $L^{q,\infty}(\N,G)$ quasi-norm is trivially larger than the $\Cc^q(\eps)$ quasi-norm, up to a multiplicative constant). 
	\end{proof}

	\begin{prop}\label{prop:13b}
		Take $p_1,p_2,q_1,q_2\in (0,\infty]$, $\alpha_1,\alpha_2\in \R$, $\eps\in (0,1)$, $\delta_0>0$, and $R\Meg 2$. Assume that $p_1\meg p_2$ and that $\alpha_1-Q_*/p_1\Meg \alpha_2-Q_*/p_2$.
		 Then, there is a constant $C>0$ such that, if either $q_1\meg q_2$ or  $\alpha_1-Q_*/p_1>\alpha_2-Q_*/p_2$, then
		\[ 
			\norm{\lambda}_{b^{p_2,q_2}_{\alpha_2}(\eps)}\meg C \norm{\lambda}_{b^{p_1,q_1}_{\alpha_1}(\eps)} 
		\]
		and, if either $q_1\meg q_2 $ or $p_1<p_2$ or $\alpha_1-Q_*/p_1>\alpha_2-Q_*/p_2$, then
		\[
		\norm{\lambda}_{f^{p_2,q_2}_{\alpha_2}((x_{j,k}))}\meg C\theta(\delta)  \norm{\lambda}_{f^{p_1,q_1}_{\alpha_1}((x_{j,k}))} 
		\]
		for every $\lambda\in \C^{\N\times \N}$ and 
		for every reduced $(\eps,\delta,R)$-lattice $(x_{j,k})_{(j,k)\in K}$ on $G$, with $\delta\in (0,\delta_0]$ and
		\[
		\theta(\delta)=\begin{cases}
			\delta^{-Q_*(1/q_1-1/q_2)} &\text{if $p_1=p_2=\infty$ and $q_1\meg q_2$}\\
			\delta^{-Q_*/q_1} &\text{if $p_1=p_2=\infty$ and $q_1>q_2$}\\
			\delta^{-Q_*(1/\min(1,p_1)-1/q_2)_+} &\text{if $p_1<p_2=\infty$ and $\alpha_1-Q_*/p_1=\alpha_2$}\\
			\min(\delta^{-Q_*/p_1},\delta^{-Q_*(1/\min(1,p_1)-1/q_2)_+}) & \text{if $p_1<p_2=\infty$ and $\alpha_1-Q_*/p_1>\alpha_2$}\\
			\delta^{-Q_*(p_1/p_2)(1/p_1-1/p_2)} & \text{if $p_1<p_2<\infty$}\\
			$1$ &\text{if $p_1=p_2<\infty$}
		\end{cases}
		\]
	\end{prop}
	
	Notice that the dependence on $\delta$   worsens as $\delta$ decreases, so that the above result  does not clash with Proposition~\ref{prop:14}. 
	
	\begin{proof}
		\textsc{Step I} If $q_1\meg q_2$, then
		\[
		\begin{split}
			\norm{\lambda}_{b^{p_2,q_2}_{\alpha_2}(\eps)}&=\norm{ \eps^{j(Q_*/p_2-\alpha_2)} \lambda_{j,k}  }_{\ell^{p_2,q_2}_{k,j}(\N,\N)}\\
				&\meg \norm{ \eps^{j(Q_*/p_2-\alpha_2)} \lambda_{j,k}  }_{\ell^{p_1,q_1}_{k,j}(\N,\N)}\\
				&\meg  \norm{ \eps^{j(Q_*/p_1-\alpha_1)} \lambda_{j,k}  }_{\ell^{p_1,q_1}_{k,j}(\N,\N)}\\
				&=\norm{\lambda}_{b^{p_1,q_1}_{\alpha_1}(\eps)} 
		\end{split}
		\]
		by the (contractive) inclusion $\ell^{p_1,q_1}(\N,\N)\subseteq \ell^{p_2,q_2}(\N,\N)$. Then, consider the case $q_1>q_2$. By the previous arguments, we may reduce to the case $p_1=p_2$, so that $\alpha_1>\alpha_2$. Then, by H\"older's inequality,
		\[
		\begin{split}
			\norm{\lambda}_{b^{p_1,q_2}_{\alpha_2}(\eps)}&=\norm{ \eps^{j(Q_*/p_1-\alpha_2)} \lambda_{j,k}  }_{\ell^{p_1,q_2}_{k,j}(\N,\N)}\\
			&\meg \norm{ \eps^{j(Q_*/p_1-\alpha_1)} \lambda_{j,k}  }_{\ell^{p_1,q_1}_{k,j}(\N,\N)} \norm{\eps^{j( \alpha_1-\alpha_2)} }_{\ell^{q_3}_j(\N)}\\
			&=\norm{\lambda}_{b^{p_1,q_1}_{\alpha_1}(\eps)}  \norm{\eps^{j( \alpha_1-\alpha_2)} }_{\ell^{q_3}_j(\N)}
		\end{split}
		\] 
		where $q_3\in (0,\infty)$ is such that $\frac{1}{q_2}=\frac{1}{q_1}+\frac{1}{q_3}$. This completes the proof of the first assertion.
		
		\textsc{Step II} Let us now consider the second assertion. 
		We distinguish several cases. 
		If $\alpha_1-Q_*/p_1>\alpha_2-Q_*/p_2$ and $p_1<\infty$, the assertion follows by (1) and Proposition~\ref{prop:12b}, with constant controlled by $\delta^{-Q_*(1/p_1-1/p_2)}$. If $p_1=\infty$ and $q_1>q_2$, then the assertion follows again by (1) and Proposition~\ref{prop:12b}, with constant controlled by $\delta^{-Q_*/q_1}$. If, otherwise, $p_1=\infty$,  and $q_1\meg q_2$, then observe that, by Proposition~\ref{prop:12b} there is a constant $C'>0$ such that
		\[
		\norm{\eps^{-j\alpha_1}\lambda_{j,k}}_{\ell^\infty_{(j,k)}(K)}\meg C' \delta^{-Q_*/q_1}\norm*{\eps^{-j\alpha_1}\sum_{k}\eps^{-j\alpha_1}\lambda_{j,k}\chi_{B(x_{j,k},\delta \eps^j)}(x)}_{\Cc_{(j,x)}^{q_1}(\eps)},
		\]
		so that
		\[
		\begin{split}
		&\eps^{-Q_*/q_2}\norm*{\chi_{[N,+\infty)\times B(x,\eps^N)}(j,y)\sum_{k}\eps^{-j\alpha_1}\lambda_{j,k}\chi_{B(x_{j,k},\delta\eps^j)}(y)}_{L^{q_2}(\N\times G)}\\
			&\qquad\meg C'^{1-q_1/q_2} \delta^{-Q_*(1/q_1-1/q_2)}\norm*{\eps^{-j\alpha_1}\sum_{k}\eps^{-j\alpha_1}\lambda_{j,k}\chi_{B(x_{j,k},\delta \eps^j)}(x)}_{\Cc_{(j,x)}^{q_1}(\eps)}
		\end{split}
		\]
		for every $\lambda\in f^{\infty,q_1}_{\alpha_1}((x_{j,k}))$, for every $N\in\N$, and for every $x\in G$. This proves the assertion when $\alpha_1=\alpha_2$. The case $\alpha_1<\alpha_2$ follows immediately.
		
		Notice that the case in which $p_1=p_2<\infty$ and $\alpha_1=\alpha_2$ is trivial. 
		We may then reduce to the case   in which $p_1<p_2$,  $\alpha_1-Q_*/p_1=\alpha_2-Q_*/p_2$, $q_2\in (0,\min(1,p_1)]$, and $q_1=\infty$.
		Then, take $\lambda\in \C^{K}$ and define $f_j\coloneqq \sum_{k } \chi_{B(x_{j,k},\delta \eps^j)} \lambda_{j,k} $ for every $j\in\N$, in order to simplify the notation.  
		We may reduce to the case in which $ \norm {  \eps^{-j \alpha_1 }f_j(x)}_{L^{\infty,p_1}_{j,x}(\N,G)}=1$.
		Observe that there is a constant $C_1>0$ such that
		\[
		\abs{f_j(x)}\meg C_1 (\delta\eps^j)^{-Q_*/p_1}\norm{ \chi_{B(x_{j,k},\delta\eps^j)}f_j}_{L^{p_1}(G)}\meg C_1 (\delta\eps^j)^{-Q_*/p_1}\norm{  f_j}_{L^{p_1}(G)}
		\]
		for every $x\in B(x_{j,k},\delta \eps^j)$ and for every $k\in\N$, so that 
		\[
		\begin{split}
			\norm{\eps^{-j \alpha_2 }f_j}_{L^\infty(G)}& \meg C_1( \delta\eps^j)^{- Q_*/p_1} \norm{\eps^{-j \alpha_2 }f_j}_{L^{p_1}(G)}= C_1( \delta\eps^j)^{- Q_*/p_2} \norm{\eps^{-j \alpha_1 }f_j}_{L^{p_1}(G)}
		\end{split}
		\]
		for every $j\in\N$.
		Then,  for every $N\in\N$,
		\[
		\begin{split}
			\sum_{j=0}^N \abs{\eps^{-j \alpha_2 }f_j}^{q_2} &\meg C_1^{q_2} \sum_{j=0}^N  \left( (\delta\eps^j)^{- Q_*/ p_2 } \norm*{ \eps^{-j \alpha_1 } f_j(x)}_{L^{\infty,p_1}_{j,x}(\N,G)}\right) ^{q_2}\\
			&=C_1^{q_2}\sum_{j=0}^N (\delta\eps^j)^{- q_2 Q_*/  p_2} 
		\end{split}
		\]
		while  
		\[
		\sum_{j\Meg N}  \abs{\eps^{-j \alpha_2 }f_j}^{q_2}  \meg   \frac{\eps^{ N(1/p_1-1/p_2) q_2 Q_* }}{1- \eps^{ (1/p_1-1/p_2) q_2Q_* }}  \sup_{j\in \N} \abs{\eps^{-j \alpha_1 } f_j}^{q_2}
		\]
		since $\alpha_1-\alpha_2= Q_*(1/p_1-1/p_2)>0$.
		If $p_2=\infty$, then by means of H\"older's inequality this leads to
		\[
		\begin{split}
			\eps^{-N Q_* }\int_{B(x,\eps^{N})} \sum_{j\Meg N}\eps^{- q_2 j  \alpha_2 } \abs{f_j(y)}^{q_2}\,\dd \beta(y)&\meg   \frac{\eps^{ N(q_2/p_1 -1) Q_* }}{1- \eps^{ (q_2/p_1) Q_* }}  \int_{B(x,\eps^{N })} \sup_{j\in \N} \abs{\eps^{-j \alpha_1 } f_j(y)}^{q_2}\,\dd \beta(y)\\
			& \meg \frac{\eps^{ N(q_2/p_1 -1) Q_* }}{1- \eps^{ (q_2/p_1) Q_*  }}  \beta(B(e,\eps^{N}))^{1-q_2/p_1}  ,
		\end{split} 
		\]
		for every $N\in\N$ and for every $x\in G$, so that there is a constant $C_2>0$ such that $\norm{\eps^{-  j  \alpha_2 }  f_j(y)}_{\Cc^{q_2}_{(j,y)}(\eps )}\meg C_2$. The assertion then follows in this case.
		
		Then, assume that $p_2<\infty$ and set $C_3\coloneqq \max( C_1^{q_2},  1/(1- \eps^{ (1/p_1-1/p_2) q_2Q_*  }))$.
		Observe that
		\[
		\norm{\eps^{-j \alpha_2 }f_j(x) }^{p_2}_{L^{q_2,p_2}_{j,x}(\N,G)}=p_2 \int_0^\infty t^{p_2} \beta \bigg( \bigg\{\,x\in G\colon \sum_{j\in \N} \abs{ \eps^{-j \alpha_2 }f_j (x)}^{q_2}>t^{q_2} \,\bigg\} \bigg)\,\frac{\dd t}{t}.
		\]
		Since 
		\[
		\bigg\{\,x\in G\colon\sum_{j\in \N} \abs{ \eps^{-j \alpha_2 }f_j(x)}^{q_2}>t^{q_2}   \,\bigg\}\subseteq \bigg\{\,x\in G\colon  \sup_j\abs{\eps^{-j \alpha_1 }  f_j(x)}>C_3^{-1/q_2}t \,\bigg\},
		\]
		one has
		\[
		\begin{split}
			&\int_0^{ \delta^{-Q_*/p_2}} t^{p_2} \beta \bigg( \bigg\{\,x\in G\colon \sum_{j\in \N} \abs{\eps^{-j \alpha_2 }f_j(x)}^{q_2}>t^{q_2}  \,\bigg\}\bigg) \, \frac{\dd t}{t}\\
			&\qquad\meg (\max(1,\delta^{-Q_*/p_2}))^{p_2-p_1}\int_0^\infty t^{p_1} \beta \left( \Set{x\in G\colon\sup_{j\in\N}\abs{\eps^{-j \alpha_1 } f_j(x)}>C_3^{-1/q_2}t  }\right)\,\frac{\dd t}{t}\\
			&\qquad\meg   (\min(1,\delta))^{  Q_*(p_1/p_2-1)}\frac{C_3^{p_1/q_2}}{p_1} .
		\end{split}
		\]
		In addition, for every $t\Meg\delta^{-Q_*/p_2}$, let $N(t)$ be the largest integer $N$ such that
		\[
		\sum_{j=0}^N \eps^{-j q_2 Q_*/  p_2} <\frac{t^{q_2}}{2 C_3}\delta^{q_2Q_*/p_2};
		\] 
		set  $N(t)\coloneqq0$ if there is no such integer (that is, if $\frac{t^{q_2}}{2 C_3}\delta^{q_2Q_*/p_2}\meg 1$).
		Observe that there is a constant $c>0$ such that $  \eps^{-N(t) Q_*/ p_2} \Meg c t\delta^{Q_*/p_2}$: this is obvious if $\delta^{-Q_*/p_2}\meg t\meg (2 C_3)^{1/q_2}\delta^{-Q_*/p_2}$ (that is, if $N(t)=0$) and follows from the fact that $ \eps^{-N(t) q_2 Q_*/  p_2}$ is comparable to $\frac{t^{q_2}}{2 C_3}\delta^{q_2Q_*/p_2} $ otherwise. The previous computations then show that $\sum_{j\meg N(t)} \abs{ \eps^{-j \alpha_2 }f_j(x)}^{q_2}\meg t^{q_2}/2 $, so that
		\[
		\begin{split}
			\bigg\{\,x\in G\colon \sum_{j\in \N} \abs{ \eps^{-j \alpha_2 }f_j(x)}^{q_2}>t^{q_2}\,\bigg\}&\subseteq \bigg\{\,x\in G\colon \sum_{j\Meg N(t)} \abs{ \eps^{-j \alpha_2 }f_j(x)}^{q_2}>\frac{t^{q_2}}{2}\bigg\}\\
			&\subseteq \Set{x\in G\colon \sup_{j\in\N} \abs{\eps^{-j \alpha_1 }f_j(x)}^{q_2} > \frac{t^{q_2}}{2C_3} \eps^{N(t)(1/p_2-1/p_1)q_2 Q_*}  }
		\end{split}
		\]
		and 
		\[
		t^{q_2}\eps^{N(t)(1/p_2-1/p_1)q_2 Q_* } \Meg t^{q_2} (c t\delta^{Q_*/p_2})^{(p_2/p_1-1)q_2} =(c^{p_2}\delta^{Q_*})^{(1/p_1-1/p_2)q_2  } t^{q_2 p_2/p_1}
		\]
		so that 
		\[
		\begin{split}
			&\int_{\delta^{-Q_*/p_2}}^\infty t^{p_2} \beta \bigg( \bigg\{\,x\in G\colon \sum_{j\in \N} \abs{ \eps^{-j \alpha_2 }f_j(x)}^{q_2}>t^{q_2} \,\bigg\}\bigg) \, \frac{\dd t}{t}\\
			&\qquad\meg \int_{\delta^{-Q_*/p_2}}^\infty t^{p_2} \beta \bigg( \bigg\{\,x\in G\colon \sup_{j\in \N}\abs{ \eps^{-j \alpha_1}f_j(x)}>\frac{ (c^{p_2}\delta^{Q_*})^{1/p_1-1/p_2  } }{(2 C_3)^{1/q_2}} t^{ p_2/p_1}  \,\bigg\}\bigg)\,\frac{\dd t}{t}\\
			&\qquad\meg \frac{p_1}{p_2}\int_0^\infty t^{p_1} \beta \bigg( \bigg\{\,x\in G\colon \sup_{j\in \N} \abs{ \eps^{-j \alpha_1 }f_j(x)}>\frac{ (c^{p_2}\delta^{Q_*})^{1/p_1-1/p_2 } }{(2 C_3)^{1/q_2}} t  \,\bigg\}\bigg)\,\frac{\dd t}{t}\\
			&\qquad=\frac{(2C_3)^{p_1/q_2} }{p_2  (c^{p_2} \delta^{Q_*})^{1-p_1/p_2 } }.
		\end{split}
		\]
		The assertion follows since $\delta^{Q_*(p_1/p_2-1)/p_2}=\delta^{-Q_*(p_1/p_2)(1/p_1-1/p_2)}$.
	\end{proof}

	\begin{prop}\label{prop:14}
		Take $p,q\in (0,+\infty]$, $\alpha\in \R$, $\eps\in (0,1)$, $\delta_0>0$ and $R\Meg 2$.  
		Take $(\psi_{j})\in \Psi_\eps$. Then, there is a constant $C>0$ such that, for every reduced $(\eps^{1/\grado},\delta,R)$-lattice $(x_{j,k})_{(j,k)\in K}$ on $G$, with $\delta\in (0,\delta_0]$,
		if we define $U_{\delta R,\eps}$ as the set of $(z_{j,k})\in G^{K}$ such that $z_{j,k}\in \overline B(e,\delta R\eps^j)$ for every $(j,k)\in K$, and
		\[
		S_{(y_{j,k})} \colon \Sr'(G)\ni f \mapsto ((f*\psi_{2,j})(y_{j,k}))_{j,k}\in \C^{K}
		\]
		for every $(y_{j,k})\in G^K$,
		then
		\[
		\frac{1}{C} \norm{f}_{B^{p,q}_\alpha(G)}\meg \delta^{Q_*/p} \max_{z\in U_{\delta R,\eps}}\norm{ S_{(x_{j,k}z_{j,k})}f}_{b^{p,q}_\alpha(\eps^{1/\grado},K)}\meg C \norm{f}_{B^{p,q}_\alpha(G)}
		\] 
		and
		\[
		\frac{1}{C} \norm{f}_{F^{p,q}_\alpha(G)}\meg \max_{z\in U_{\delta R,\eps}}  \norm{ S_{(x_{j,k}z_{j,k})}f }_{f^{p,q}_\alpha((x_{j,k}))}\meg C \norm{f}_{F^{p,q}_\alpha(G)}
		\]
		for every $f\in\Sr'(G)$. In addition, there is $\delta_->0$ such that, if $\delta\meg \delta_-$, then
		\[
		\frac{1}{C} \norm{f}_{B^{p,q}_\alpha(G)}\meg \delta^{Q_*/p} \min_{z\in U_{\delta R,\eps}}\norm{ S_{(x_{j,k}z_{j,k})}f}_{b^{p,q}_\alpha(\eps^{1/\grado},K)}\meg C \norm{f}_{B^{p,q}_\alpha(G)}
		\] 
		for every $f\in B^{p,q}_\alpha(G)$,  and
		\[
		\frac{1}{C} \norm{f}_{F^{p,q}_\alpha(G)}\meg \min_{z\in U_{\delta R,\eps}}  \norm{ S_{(x_{j,k}z_{j,k})}f }_{f^{p,q}_\alpha((x_{j,k}))}\meg C \norm{f}_{F^{p,q}_\alpha(G)}
		\]
		for every $f\in F^{p,q}_\alpha(G)$.
		
		Finally, if $\delta \meg \delta_-$, then the mapping $S_{(x_{j,k})}$ induces isomorphisms of $\mathring B^{p,q}_\alpha(G)$ and $\mathring F^{p,q}_\alpha(G)$ onto closed vector subspaces of $\mathring b^{p,q}_\alpha(\eps^{1/\grado},K) $ and $\mathring f^{p,q}_\alpha((x_{j,k}))$, respectively. 
	\end{prop}

	\begin{proof}
		The first and third assertions are a consequence of Corollary~\ref{cor:17}, while the second and fourth assertions are a consequence of Proposition~\ref{prop:10} when $p<\infty$, and of Proposition~\ref{prop:10bis} when $p=\infty$. In order to prove the last assertion, observe first that the image of $\Sr(G)$ under $S_{(x_{j,k})}$ is contained in $b^{p_1,q_1}_{\alpha_1}(\eps,K)$ for every $p_1,q_1\in (0,+\infty]$ and for every $\alpha_1\in \R$. The conclusion follows using Propositions~\ref{prop:12b} and~\ref{prop:13b} and the fact that clearly $b^{p_1,q_1}_{\alpha_1}(\eps^{1/\grado},K)=\mathring b^{p_1,q_1}_{\alpha_1}(\eps^{1/\grado},K)$ for $p_1,q_1<\infty$. 
	\end{proof}
	
	We can now reap some results from the previous sampling theorem.

	\begin{prop}\label{prop:12}
		Take $p,q\in (0,\infty]$  and $\alpha\in \R$. Then,
		\[
		B^{p,\min(p,q)}_\alpha(G)\subseteq F^{p,q}_\alpha(G)\subseteq B^{p,\max(p,q)}_\alpha(G)
		\]
		continuously.
	\end{prop}
	
	\begin{proof}
		The assertion follows from Propositions~\ref{prop:12b} and~\ref{prop:14}.
	\end{proof}
	
	\begin{prop}\label{prop:13}
		Take $p_1,p_2,q_1,q_2\in (0,\infty]$ and $\alpha_1,\alpha_2\in \R$. Assume that $p_1\meg p_2$. Then, the following hold:
		\begin{enumerate}
			\item[\textnormal{(1)}] if $\alpha_1-Q_*/p_1\Meg \alpha_2-Q_*/p_2$ and either $q_1\meg q_2$ or  $\alpha_1-Q_*/p_1>\alpha_2-Q_*/p_2$, then
			\[
			B^{p_1,q_1}_{\alpha_1}(G)\subseteq B^{p_2,q_2}_{\alpha_2}(G)
			\]
			continuously;
			
			\item[\textnormal{(1$'$)}] if   $\alpha_1-Q_*/p_1\Meg \alpha_2-Q_*/p_2$ and  either $q_1\meg q_2 $ or $p_1<p_2$ or $\alpha_1-Q_*/p_1>\alpha_2-Q_*/p_2$, then
			\[
			F^{p_1,q_1}_{\alpha_1}(G)\subseteq F^{p_2,q_2}_{\alpha_2}(G)
			\]
			continuously; 
			
			\item[\textnormal{(2)}]    
			if     $\alpha_1\Meg \frac{Q_*}{p_1}-\frac{Q_*}{\max(p_2,1)} $ and either $q_1\meg \min(1,p_2)$ or $\alpha_1> \frac{Q_*}{p_1}-\frac{Q_*}{\max(p_2,1)} $, then
			\[
			B^{p_1,q_1}_{\alpha_1}(G )\subseteq   L^{p_2}(G)\cap L^{\max(p_2,1)}(G)
			\]
			continuously;
			
			\item[\textnormal{(2$'$)}] if     $\alpha_1>  \frac{Q_*}{p_1}-\frac{Q_*}{\max(p_2,1)} $,    then
			\[
			F^{p_1,q_1}_{\alpha_1}(G )\subseteq L^{p_2}(G )\cap L^{\max(p_2,1)}(G)
			\]
			continuously.  
		\end{enumerate}
	\end{prop}
	
	Concerning (2) and (2$'$), we observe explicitly that we wrote $L^{p_2}(\beta )\cap L^{\max(p_2,1)}(G)$ instead of $L^{p_2}(\beta )$ in order to stress the fact that the inclusions are meaningful (and actual inclusions instead of non-injective canonical mappings) only because of the stronger assumption on $\alpha_1$, which ensures the continuous inclusions in $L^{\max(p_2,1)}(G)$. 
	
	Observe that, if $f$ is a $Z$-valued Radon measure and $\abs{f}=h\cdot \beta$ for some positive $h\in L^p(G)$, with $p\in [1,\infty]$, then it is readily seen that $f\in B^{p,\infty}_0(G)$. Assertion (2) tells us that, if $f$ is a little more regular (that is, $f\in B^{p,q}_\alpha(G)$ for some $\alpha>0$ and some $q\in (0,\infty]$), then $f\in L^p(G;Z)$, that is, $f$ is actually (the equivalence class of) a measurable function. This should not come as a surprise, since one may show directly that $\Lc^{-\alpha}f$ is actually a function in $L^p(G;Z)$, at least for $\alpha>Q_*/(p\grado)$.

	\begin{proof} 
		(1) and (1$'$) follow from Propositions~\ref{prop:13b} and~\ref{prop:14}.		
		
		(2)   By (1), we may reduce to proving the first assertion for $p_1=q_1=p_2\meg 1$ and $\alpha_1=(1/p_1-1)_*Q_* $. 
		Assume first that $p_1=1$. Then, observe that 
		\[
		\begin{split}
			\norm{f}_{L^{1}(G)} &=\norm*{\sum_{j\in \N} f*\psi_j }_{L^{ 1}(G)} \meg \sum_{j\in \N} \norm{f*\psi_j}_{L^{1}(G)} 
		\end{split}
		\]
		for $(\psi_j)\in \widetilde \Psi_\eps$,\footnote{The convergence of $\sum_j \psi_j$ (to $\delta_e$) in $\Oc'_{C,L}(G)$, which is needed to ensure that $f=\sum_j f*\psi_j$ in $\Sr'(G)$, is a consequence of Lemma~\ref{lem:41}. Since the sum $\sum_j f*\psi_j$ also converges (absolutely) in $L^1(G)$, the assertion follows. } which leads to the conclusion in this case. 
		Then, assume that $p_1<1$, and observe that $B^{p_1,p_1}_{\alpha_1}(G)\subseteq B^{1,1}_0(G)$ by (1). Consequently, $B^{p_1,p_1}_{\alpha_1}(G)\subseteq L^1(G)$. In particular, if $f\in B^{p_1,p_1}_{\alpha_1}(G)$, then the above computations show that $f=\sum_{j\in \N} f*\psi_j$, with the sum converging in $L^1(G)$, in particular in measure. Then, we may essentially repeat the above computations and obtain\footnote{We observe explicitly that, when $p_1=1$, we were able to perform the analogous computations using the fact that $f=\sum_{j\in \N} f*\psi_j $ in $\Sr'(G)$ and the lower semi-continuity of the norm of $B^{1,1}_0(G)$ (extended by $+\infty$) on $\Sr'(G)$. When $p_1<1$, this is no longer the case, so that we need to have some stronger kind of convergence which is `compatible' with the space $L^{p_1}(G)$ (such as convergence in measure).}
		\[
		\begin{split}
			\norm{f}_{L^{p_1}(G)}^{p_1} &=\norm*{\sum_{j\in \N} f*\psi_j }_{L^{p_1}(G)}^{p_1}\meg \sum_{j\in \N} \norm{f*\psi_j}_{L^{p_1}(G)}^{p_1},
		\end{split}
		\]
		which leads to the conclusion.
		
		(2$'$) This follows from   (1$'$), (2), and Proposition~\ref{prop:12}.
	\end{proof}
	 
	\begin{prop}
		Take $p,q\in (0,+\infty]$ and $\alpha\in \R$. Then, $B^{p,q}_\alpha(G)$ and $F^{p,q}_\alpha(G)$ are complete and embed continuously in $\Sr'(G)$. 
	\end{prop}
	
	\begin{proof}
		Observe that by Proposition~\ref{prop:13} we have the continuous inclusion $B^{p,q}_\alpha(G)\subseteq B^{\infty,\infty}_{\alpha-Q_*/p}(G)$. In addition, by Theorem~\ref{teo:17} there is a continuous inclusion $B^{\infty,\infty}_{\alpha-Q_*/p}(G)\subseteq \Lc_\omega^{Q_*/p-\alpha-1} L^\infty(G)$ for every sufficiently large $\omega\in \R$. 
		Consequently, $B^{p,q}_\alpha(G)\subseteq \Sr'(G)$ continuously. If, now, $(f_j)$ is a Cauchy sequence in $B^{p,q}_\alpha(G) $, then it is a Cauchy sequence in $\Sr'(G)$, so that it has a limit $f$    in $\Sr'(G)$. Since the quasi-norm on $B^{p,q}_\alpha(G) $ (extended by $+\infty$) is lower semi-continuous on $\Sr'(G)$, it is readily seen that $f\in B^{p,q}_\alpha(G)$ and that $f_j$ converges to $f$ in $\Sr'(G)$. The second assertion is proved similarly.
	\end{proof}

	We now pass to `atomic decomposition.' The first step is to define an appropriate notion of atoms. We basically follow~\cite{Hu}. When $Z$ is a general Banach space, one may also weaken the first definition of atoms only requiring that  $a_{0,k'}, \Lc^K a_{0,k'}, b_{j,k}, \Lc^{S+K}b_{j,k}$ be $Z$-valued Radon measures and that the estimates in (2) hold for $\abs{a_{0,k'}}, \abs{\Lc^K a_{0,k'}}, \abs{b_{j,k}}, \abs{\Lc^{S+K}b_{j,k}}$ instead of $a_{0,k'}, \Lc^K a_{0,k'}, b_{j,k}, \Lc^{S+K} b_{j,k}$, respectively.
	
	\begin{deff}\label{def:11}
		Take $\eps\in (0,1)$, $\delta,N>0$, $R\Meg 2$, $K,S\in \N$, $p\in [1,\infty]$ and a reduced $(\eps ,\delta, R)$-lattice $(x_{j,k})_{(j,k)\in K}$. We say that a family $(a_{j,k})_{(j,k)\in K}$ is a system of $(K,S,N,p,\Lc)$-atoms associated with $(x_{j,k})$ if there is a family $(b_{j,k})_{(j,k)\in K, j\Meg 1}$ such that following hold:
		\begin{enumerate}
			\item $a_{0,k'},b_{j,k}\in \Sr'(G)$   and $a_{j,k}=\Lc^S b_{j,k}$  for every $(0,k'),(j,k)\in K$ with $j\Meg 1$;
			
			\item $a_{0,k'}, \Lc^K a_{0,k'}, b_{j,k}, \Lc^{S+K}b_{j,k}\in L^p_\loc(G)$,
			\[
			\norm{(1+d(\,\cdot\,,x_{0,k}) )^{N} a_{0,k}}_{L^p(G)}, \norm{(1+d(\,\cdot\,,x_{0,k}) )^{N}\Lc^K a_{0,k}}_{L^p(G)}\meg 1,
			\]
			and
			\[\norm{(1+d(\,\cdot\,,x_{j,k})/\eps^j )^{N}  b_{j,k}}_{L^p(G)}, \norm{(1+d(\,\cdot\,,x_{j,k})/\eps^j )^{N}  (\eps^{j\grado}\Lc)^{K+S}b_{j,k}}_{L^p(G)}\meg \eps^{j(S\grado+Q_*/p)}
			\]
			for every $(0,k'),(j,k)\in K$ with $j\Meg 1$.
		\end{enumerate} 
		
		We say that a family $(a_{j,k})_{(j,k)\in K}$ is a system of strong $(K,S,\Lc)$-atoms associated with $(x_{j,k})$ if there is a family $(b_{j,k})_{(j,k)\in K, j\Meg 1}$ such that following hold:
		\begin{enumerate}
			\item   $a_{0,k'},b_{j,k}\in C^\infty(G)$  and $a_{j,k}=\Lc^S b_{j,k}$   for every $(0,k'),(j,k)\in K$ with $j\Meg 1$;
			
			\item $a_{0,k'}$ is supported in $B(x_{0,k'},2\delta R)$ and $b_{j,k}$ is supported in $B(x_{j,k}, 2\delta R \eps^{j })$  for every $(0,k'),(j,k)\in K$ with $j\Meg 1$;
			
			\item  $\abs{X a_{0,k'}}\meg \abs{X}$ for every $X\in U_{K\grado}$ and  $\abs{X b_{j,k}}\meg \abs{X}\eps^{j(S\grado-\deg(X) )}$ for every $X\in U_{(K+S)\grado}$, and for every $(0,k'),(j,k)\in K$ with $j\Meg 1$.
		\end{enumerate}  
	\end{deff}
	
	When $p=1$, one may replace $L^1(G)$ with the space  $\Mc^1(G)$ of bounded measures in the above estimates.

	Observe that our definition of (strong) atoms is different from the classical one, since (strong) atoms are \emph{not} normalized using the $L^2$-norm of the ball where they are supported. This leads to different (and more symmetric) conditions in Theorems~\ref{teo:13bis} and~\ref{teo:13} below.

	\begin{oss}
		Take $c\in (0,1)$ so that  $1/c^2\Meg (1+2\delta R)^N[\eps^{-jQ_*} \beta(B(e,2 \delta R\eps^j))]^{1/p} ,\abs{\Lc^K},\abs{\Lc^{K+S}}$ for every $j\in\N$. If $(a_{j,k})$ is a  system of strong $(K,S,\Lc)$-atoms, then $(c a_{j,k})$ is a system of $(K,S,N,p,\Lc)$-atoms (with respect to any fixed reduced $(\eps,\delta,R)$-lattice on $G$).   
	\end{oss}

	\begin{teo}\label{teo:13bis}
		Take $p,q\in (0,\infty]$  and $\alpha\in \R$, and take $K,S\in \N$, $N>0$, and  $p_0\in [1,\infty]$ so that $K>(Q_*/p_0+\alpha)/\grado $, $S>   (Q_*(1/\min(1,p)-1/p_0) -\alpha) /\grado $, and $N>D/\min(1,p)$. Take $\eps\in (0,1) $, $\delta>0$ and $R\Meg 2$. Then, there is a constant $C>0$ such that, for every system of $(K,S,N,p_0,\Lc)$-atoms $(a_{j,k})$ associated with some  reduced $(\eps ,\delta,R)$-lattice $(x_{j,k})_{(j,k)\in J}$,
		\[
		\norm*{\sum_{j,k} \lambda_{j,k} a_{j,k}  }_{B^{p,q}_\alpha(G)}\meg C \norm{\lambda }_{b^{p,q}_{\alpha}(\eps,J )}
		\]
		for every $\lambda\in  b^{p,q}_{\alpha}(\eps,J )$, where the sum converges in $\Sr'(G)$  and also in $\mathring B^{p,q}_\alpha(G)$ if $\lambda\in  \mathring b^{p,q}_{\alpha}(\eps,J )$. In addition, for every $f\in B^{p,q}_\alpha(G)$ there are a system $(a_{f,j,k})_{j,k}$ of strong $(K,S,\Lc)$-atoms associated with $(x_{j,k})$ and $\lambda_f\in b^{p,q}_{\alpha}(\eps,J )$ such that 
		\[
		\norm{\lambda_f }_{b^{p,q}_{\alpha}(\eps,J )}\meg C\norm{f}_{B^{p,q}_\alpha(G)}
		\]
		and such that $f=\sum_{j,k} \lambda_{f,j,k} a_{f,j,k} $.
	\end{teo}

	\begin{teo}\label{teo:13}
		Take $p,q\in (0,\infty]$  and $\alpha\in \R$, and take $K,S\in \N$, $N>0$, and $p_0\in [1,\infty]$ so that $K>(Q_*/p_0+\alpha) /\grado$, $S>   (Q_*(1/\min(1,p,q)-1/p_0) -\alpha)/\grado $, and $N>D/\min(1,p,q)$. Take $\eps\in (0,1) $, $\delta>0$ and $R\Meg 2$. Then, there is a constant $C>0$ such that,  for every system of $(K,S,N,p_0,\Lc)$-atoms $(a_{j,k})$ associated with some reduced $(\eps ,\delta,R)$-lattice $(x_{j,k})_{(j,k)\in J}$,  
		\[
		\norm*{\sum_{j,k} \lambda_{j,k} a_{j,k}  }_{F^{p,q}_\alpha(G)}\meg C \norm{\lambda}_{f^{p,q}_\alpha((x_{j,k}))}
		\]
		for every $\lambda\in f^{p,q}_\alpha((x_{j,k}))$, where the sum converges in $\Sr'(G)$ and also in $\mathring F^{p,q}_\alpha(G)$ if $ \lambda\in \mathring f^{p,q}_\alpha((x_{j,k}))$. In addition,  for every $f\in F^{p,q}_\alpha(G)$ there are a system $(a_{f,j,k})_{j,k}$ of strong $(K,S,\Lc)$-atoms associated with $(x_{j,k})$ and $\lambda_f\in f^{p,q}_\alpha((x_{j,k}))$ such that 
		\[
		\norm{\lambda_f }_{f^{p,q}_\alpha((x_{j,k}))}\meg C\norm{f}_{F^{p,q}_\alpha(G)}
		\]
		and such that $f=\sum_{j,k} \lambda_{f,j,k} a_{f,j,k} $.
	\end{teo}
	
	Both Theorem~\ref{teo:13bis} and~\ref{teo:13} are consequences of the following lemmas.  
	
	\begin{lem}\label{lem:67}
		Take $\eps\in (0,1)$, $\delta_0>0$, $R\Meg 2$, $K,S\in\N$, $N>0$, $p\in [1,\infty]$, and $(\psi_{j})\in \Psi_{\eps^{\grado}}$. Then, there is a constant $C>0$ such that
		\[
		\abs{(a_{j,k}*\psi_{j'})(x)}\meg C  \eps^{(j-j')_+(S\grado+Q_*/p)+(j'-j)_+(K\grado-Q_*/p)}(1+d(x,x_{j,k})/\eps^{\min(j,j') })^{-N} 
		\]
		for every $(j,k)\in J$, for every $j'\in\N$, and for every system  $(a_{j,k})$ of $(K,S,N,\Lc)$-atoms associated with a reduced $(\eps,\delta,R)$-lattice $(x_{j,k})_{(j,k)\in J}$ on $G$, with $\delta\in (0,\delta_0]$.
	\end{lem}
	
	This lemma is essentially a repetition of Lemma~\ref{lem:40}.
	
	\begin{proof}
		Observe that, for $j\Meg j'$,
		\[
		\begin{split}
			\abs{(a_{j,k}*\psi_{j'})(x)}&=\abs{(b_{j,k}*\Lc^S\psi_{j'})(x)} \\
			&\meg \int_G \abs{b_{j,k}(y) \Lc^S \psi_{j'}(y^{-1}x)}\,\dd \beta(y)\\ 
			&\meg \eps^{(j-j') (S\grado+Q_*/p)} \norm*{\eps^{j' Q_*/p} \frac{ [(\eps^{j'\grado}\Lc)^S \psi_{j'}](y^{-1}x)}{(1+d(y,x_{j,k})/\eps^{j })^N}}_{L^{p'}_y(G)}\\
			&\meg  \eps^{(j-j') (S\grado+Q_*/p)} \norm*{\eps^{j' Q_*/p} \frac{ [(\eps^{j'\grado}\Lc)^S \psi_{j'}](y)}{(1+\abs{x_{j,k}^{-1}xy^{-1}}_*/\eps^{j' })^N}}_{L^{p'}_y(G)}\\
			&\meg  \eps^{(j-j')(S\grado+Q_*/p)}(1+d(x,x_{j,k})/\eps^{j' })^{-N} \norm*{\eps^{j' Q_*/p} (1+\abs{\,\cdot\,}_*/\eps^{j' })^{N} (\eps^{j'\grado}\Lc)^S \psi_{j'}}_{L^{p'}(G)},
		\end{split}
		\]
		so that the assertion follows from Corollary~\ref{cor:12} in this case (here we took $b_{j,k}$ corresponding to $a_{j,k}$ as in Definition~\ref{def:11}).
		Conversely, take $j<j'$, and observe that $j'\Meg 1$, so that we may consider $\Lc^{-K}\psi_{j'}$ according to our previous conventions. Then,  arguing as before, 
		\[
		\begin{split}
			\abs{(a_{j,k}*\psi_{j'})(x)}&=\abs{(\Lc^K a_{j,k}*\Lc^{-K}\psi_{j'})(x)} \\
			&\meg \eps^{(j'-j) (K\grado-Q_*/p)} \norm*{\eps^{j'Q_*/p}\frac{  [(\eps^{j'\grado}\Lc)^{-K} \psi_{j'}](y)}{(1+\abs{x_{j,k}^{-1}xy^{-1}}_*/\eps^{j })^N}}_{L^{p'}_y(G)}\\
			&\meg  \eps^{(j'-j)(K\grado-Q_*/p)}(1+d(x,x_{j,k})/\eps^{j })^{-N} \norm*{\eps^{j'Q_*/p} (1+\abs{\,\cdot\,}_*/\eps^{j })^{N} (\eps^{j'\grado}\Lc)^{-K} \psi_{j'}}_{L^{p'}_y(G)}\\
			&\meg  \eps^{(j'-j)(K\grado-Q_*/p)}(1+d(x,x_{j,k})/\eps^{j })^{-N} \norm*{\eps^{j'Q_*/p} (1+\abs{\,\cdot\,}_*/\eps^{j'})^{N} (\eps^{j'\grado}\Lc)^{-K} \psi_{j'}}_{L^{p'}(G)},
		\end{split}
		\]
		so that the assertion follows again from Corollary~\ref{cor:12}.
	\end{proof}
	
	\begin{lem}\label{lem:68bis} 
		Take $p,q\in (0,\infty]$, $\alpha \in\R$, $\eps\in (0,1)$, $\delta_0>0$, $R\Meg 2$, $K>\alpha $, $S> Q_*/\min(1,p) -\alpha $, and $N>D/\min(1,p)$. Then, there is a constant $C>0$ such that the following hold. 	
		Take a reduced $(\eps,\delta, R)$-lattice $(x_{j,k})_{(j,k)\in J}$ on $G$ and $a\in \C^{J\times J}$ such that
		\[
		\abs{a_{(j,k),(j',k')}}\meg \eps^{(j'-j)_+S +(j-j')_+K }(1+d(x_{j,k},x_{j',k'})/\eps^{\min(j,j') })^{-N}  
		\]
		for every $(j,k),(j',k')\in J$. If we define
		\[
		T\colon \lambda \mapsto \Big(\sum_{j',k'} a_{(j,k),(j',k')}\lambda_{j',k'} \Big)_{j,k},
		\]
		(pointwise convergence),
		then
		\[
		\norm{T \lambda}_{b^{p,q}_\alpha(\eps ,J)}\meg C\norm{\lambda}_{b^{p,q}_\alpha(\eps ,J)}  
		\]	
		for every $\lambda\in b^{p,q}_\alpha(\eps,J )$.
	\end{lem}
	
	The proof is similar to that of the following Lemma~\ref{lem:68} (in this case, one reduces to the case $p>1$ and there is no need to deal separately with the case $p=\infty$).
	
	\begin{lem}\label{lem:68}
		Take $p,q\in (0,\infty]$, $\alpha \in\R$, $\eps\in (0,1)$, $\delta_0>0$, $R\Meg 2$, $K>\alpha $, $S> Q_*/\min(1,p,q) -\alpha $, and $N>D/\min(1,p,q)$. Then, there is a constant $C>0$ such that the following hold. 	
		Take a reduced $(\eps,\delta, R)$-lattice $(x_{j,k})_{(j,k)\in J}$ on $G$ and $a\in \C^{J\times J}$
		\[
		\abs{a_{(j,k),(j',k')}}\meg \eps^{(j'-j)_+S +(j-j')_+K }(1+d(x_{j,k},x_{j',k'})/\eps^{\min(j,j') })^{-N}  
		\]
		for every $(j,k),(j',k')\in J$. If we define
		\[
		T\colon \lambda \mapsto \Big(\sum_{j',k'} a_{(j,k),(j',k')}\lambda_{j',k'} \Big)_{j,k},
		\]
		(pointwise convergence),
		then
		\[
		 \norm{T \lambda}_{f^{p,q}_\alpha((x_{j,k}))}\meg C\delta^{-Q_*} \norm{\lambda}_{f^{p,q}_\alpha((x_{j,k}))} 
		\]	
		for every  $\lambda\in f^{p,q}_\alpha((x_{j,k}))$.
	\end{lem}
	
	\begin{proof}[First part of the proof of Lemma~\ref{lem:68}.]
		\textsc{Step I} Assume that the assertion has been proved for $p,q> 1$, $\alpha=0$, and $\lambda\in \C^{(J)}$, and let us prove it in the general case. Fix $p_0\in (0,\min(1,p,q))$ so that $S>Q_*/p_0-\alpha$ and $N>D/p_0$. In addition, consider $a'_{(j,k),(j',k')}\coloneqq \eps^{(j'-j)p_0\alpha}\abs{a_{(j,k),(j',k')}}^{p_0} $, so that
		\[
			\abs{a'_{(j,k),(j',k')}}\meg  \eps^{(j'-j)_+ (S+\alpha)p_0 +(j-j')_+ (K-\alpha)p_0 }(1+d(x_{j,k},x_{j',k'})/\eps^{\min(j,j') })^{-N p_0} .
		\]
		By our assumption, there is a constant $C_1>0$ such that, if we define
		\[ 
		T'\colon \lambda \mapsto \Big(\sum_{j',k'} a'_{(j,k),(j',k')}\lambda_{j',k'} \Big)_{j,k},
		\]
		then 
		\[
		\norm{T' \lambda}_{f^{p/p_0,q/p_0}_0((x_{j,k}))}\meg C_1 \norm{\lambda}_{f^{p/p_0,q/p_0}_0((x_{j,k}))}.
		\]
		Then, observe that
		\[
		\begin{split}
			\norm{T \lambda}_{f^{p,q}_\alpha((x_{j,k}))}&=\norm{(\eps^{-jp_0\alpha}\abs{(T \lambda)_{j,k}}^{p_0})_{j,k}}_{f^{p/p_0,q/p_0}_0((x_{j,k}))}^{1/p_0}\\
				&\meg \norm{T'[ (\eps^{-j'p_0\alpha}\abs{\lambda_{j',k'}}^{p_0})_{j',k'}]}_{f^{p/p_0,q/p_0}_0((x_{j,k}))}^{1/p_0}\\
				&\meg C_1 \norm{  (\eps^{-j'p_0\alpha}\abs{\lambda_{j',k'}}^{p_0})_{j',k'}}_{f^{p/p_0,q/p_0}_0((x_{j,k}))}^{1/p_0}\\
				&\meg  C_1 \norm{\lambda}_{f^{p,q}_\alpha((x_{j,k}))}
		\end{split}
		\]
		for every $\lambda\in \C^{(J)}$. Pointwise convergence for general $\lambda\in f^{p,q}_\alpha((x_{j,k}))$ is then ensured considering the case in which $\lambda_{j,k}\Meg 0$ for every $j,k\in\N$ (since each component of $a'$ is positive); the above inequality is then extended to every $\lambda\in f^{p,q}_\alpha((x_{j,k}))$ by lower semi-continuity. 
		
		\textsc{step II} We shall now consider the case $p,q> 1$, $\alpha=0$,   and $p<\infty$. Take a constant $C_2>0$ such that $C_2 r^{Q_*}\meg \beta(B(e,r))\meg C_2 r^{Q_*}$ for every $r\in (0,R\delta_0]$. Define
		\[
		(\iota\lambda)(j,x)\coloneqq \sum_{k\in\N} \lambda_{j,k}\chi_{B(x_{j,k},\delta \eps^j)}(x)
		\] 
		for every $\lambda\in \C^{J}$, for every $j\in\N$ and for every $x\in G$, and observe that
		\[
		\begin{split}
			\abs{(\iota T \lambda)(j,x)}&\meg\sum_{k} \chi_{B(x_{j,k},\delta \eps^j)}(x)\sum_{j',k'} \abs{a_{(j,k),(j',k')}} \abs{\lambda_{j',k'}}\\
				&\meg C_2\delta^{-Q_*}\sum_{k} \chi_{B(x_{j,k},\delta \eps^j)}(x)\sum_{j',k'}\eps^{-j' Q_*} \int_{B(x_{j',k'},\delta \eps^{j'})} \abs{a_{(j,k),(j',k')}}\abs{(\iota\lambda)(j',x')}\,\dd \beta(x')\\
				&\meg (2\max(1,\delta_0))^N C_2 \delta^{-Q_*} \sum_{j'} \int_G \Ac((j,x),(j',x'))\abs{(\iota\lambda)(j',x')}\,\dd \beta(x')
		\end{split}
		\]
		for every $\lambda\in \C^{(J)}$ and for every $(j,x)\in\N\times G$, where
		\[
		\Ac((j,x),(j',x'))\coloneqq \eps^{(j'-j)_+S +(j-j')_+K -j' Q_* }(1+d(x,x')/\eps^{\min(j,j') })^{-N}
		\] 
		for every $(j,x),(j',x')\in\N\times G$. It is then clear that
		\[
		\abs{(\iota T \lambda)(j,x)}\meg  (2\max(1,\delta_0))^N C_2 \delta^{-Q_*} \sum_{j'} \eps^{(j'-j)_+(S-Q_*)+(j-j')_+K  } [\Nc_{1,N}(\iota\lambda)(j',\,\cdot\,)](x)
		\]
		for every $(j,x)\in\N$. Therefore, combining Lemmas~\ref{lem:25d} and~\ref{lem:39}, we see that there is a constant $C_3>0$ such that
		\[
		\norm{(\iota T \lambda)(j,x)}_{L^{q,p}_{(j,x)}(\N,G)}\meg C_3\delta^{-Q_*} \norm{[\Mc_1(\iota\lambda)(j,\,\cdot\,)](x)}_{L^{q,p}_{(j,x)}(\N,G)}, 
		\]
		so that the assertion follows from, e.g.,~\cite{Singular}.
	\end{proof}

	\begin{proof}[First part of the proof of Theorem~\ref{teo:13}.]
		Let $U$ be the set of $(z_{j,k})\in G^{J}$ such that $z_{j,k}\in B(e,\delta R \eps^j)$ for every $(j,k)\in J$. Take $(\psi_j)\in \widetilde\Psi_{\eps^{\grado}}$. 
		Then, combining Lemmas~\ref{lem:67} and~\ref{lem:68}, we see that there is a constant $C_1>0$ such that
		\[
		\norm{ ((A \lambda)*\psi_j)(x_{j,k}z_{j,k}) }_{f^{p,q}_\alpha((x_{j,k}))}\meg C_1 \norm{\lambda}_{f^{p,q}_\alpha((x_{j,k}))}
		\]
		for every $\lambda\in \C^{(J)}$ and for every $(z_{j,k})\in U$, where $A\lambda\coloneqq \sum_{j,k} \lambda_{j,k} a_{j,k}$ for every $\lambda\in \C^{(J)}$. Then, Proposition~\ref{prop:14} shows that there is a constant $C_2>0$ such that
		\[
		\norm{A \lambda}_{F^{p,q}_\alpha(G)}\meg C_2\norm{\lambda}_{f^{p,q}_\alpha((x_{j,k}))} 
		\] 
		for every $\lambda\in \C^{(J)}$, hence also for every $\lambda\in \mathring f^{p,q}_\alpha((x_{j,k}))$ by density and continuity (with convergence in $F^{p,q}_\alpha(G)$ of the sum in the definition of $A$).

		\textsc{Step II} Take a family $(\tau_{j,k})$ as in Lemma~\ref{lem:43}.
		Take $f\in F^{p,q}_\alpha(G)$, and define $\widetilde a_{f,0,k'}\coloneqq \tau_{0,k'} (f*\psi_0)$ and $\widetilde a_{f,j,k}\coloneqq \Lc^S(\tau_{j,k}(f*\Lc^{-S}\psi_j))$ for every $(0,k')\in J$ and for every $(j,k)\in J$ with $j\Meg 1$. 
		Observe that clearly $\widetilde a_{f,j,k}$ is of class $C^\infty$ and is supported in $ B(x_{j,k},2 \delta R \eps^{j })$; if $j\Meg 1$, it is the $\Lc^S$ derivative  of some $C^\infty$ function supported in $B(x_{j,k},2\delta R \eps^{j })$. 
		Let us now define, for simplicity, $W^{\lambda,\infty}(V)$ as the set of $g\in L^\infty(V)$ such that $\sup_{X\in U_\lambda, \abs{X}\meg 1} \norm{X g}_{L^\infty(V)}<\infty$ for every $\lambda\Meg 0$, and for every open subset $V$ of $G$.
		Take $p_0\in (0,\min(1,p,q))$ so that $S\grado> Q_*/p_0-\alpha$ and $N>D/p_0$, and observe that, by Lemmas~\ref{lem:33} and~\ref{lem:39}, there is a constant $C_3>0$ such that, setting $\widetilde \psi_0\coloneqq \psi_0$ and $\widetilde \psi_j\coloneqq (\eps^{j\grado} \Lc)^{-S} \psi_j$ for every $j\Meg 1$,
		\[
		\abs{(X (f*\psi_0))(x)}\meg 2^{D/p_0+1} \Nc_{\infty,D/p_0+1,\delta_0 }(X(f*\psi_0))(y)\meg C_3 \abs{X} \Mc_{p_0}(f*\widetilde\psi_0)(y)
		\]
		for every $x,y\in G$ with $d(x,y)<\delta $ and for every $X\in U_{K\grado}$,
		and, analogously,
		\[
		\abs{(X (f*\Lc^{-S}\psi_j))(x)}\meg C_3 \abs{X} \eps^{j(S\grado-\deg(X) )} \Mc_{p_0}(f*\widetilde\psi_j)(y)
		\]
		for every $j\Meg 1$, for every $x,y\in G$ with $d(x,y)< \delta \eps^{j }$, and for every $X\in U_{(K+S)\grado}$. 
		Therefore, using  the properties of the $\tau_{j,k} $, it is readily seen that there is a constant $C_{4}>0$ such that
		\[
		\abs{X \widetilde a_{f,0,k} }\meg C_{4} \inf_{B(x_{0,k}, \delta  )} \Mc_{p_0}(f*\widetilde \psi_0)  \eqqcolon \lambda_{f,0,k}
		\]
		for every $X\in U_{K\grado}$ with $\abs{X}\meg 1$ and for every $(0,k)\in J$,
		and  
		\[
		\eps^{-j(S\grado-\deg(X) )}\abs{X(\tau_{j,k}(f*\Lc^{-S}\psi_j)) }\meg C_{4}  \inf_{B(x_{j,k}, \delta \eps^{j })} \Mc_{p_0}(f*\widetilde \psi_j) \eqqcolon \lambda_{f,j,k}
		\]
		for every $X\in U_{(K+S)\grado}$ with $\abs{X}\meg 1$, for every $(j,k)\in J$ with $j\Meg 1$. 
		
		Consequently, if we set   $a_{f,j,k}\coloneqq \widetilde a_{f,j,k}/\lambda_{j,k}$ if $\lambda_{f,j,k}\neq 0$, and  $a_{f,j,k}=0$ if $\lambda_{f,j,k}=0$ (in which case $\widetilde a_{f,j,k}=0$), then $(a_{f,j,k})$ is a system of strong $(K,S,\Lc)$-atoms associated with $(x_{j,k})$ such that 
		\[
		\begin{split}
			\sum_{j,k} \lambda_{f,j,k}a_{f,j,k}&=\sum_k \tau_{0,k} (f*\psi_0)+\sum_{j\Meg 1}\sum_k \Lc^S(\tau_{j,k} f*\Lc^{-S}\psi_j)\\
			&=f*\psi_0 +\sum_{j\Meg 1} \Lc^S(f*\Lc^{-S} \psi_j)\\
			&=\sum_j f*\psi_j=f
		\end{split}
		\]
		in $\Sr'(G)$. In addition,
		\[
		\begin{split}
			\norm{\eps^{-j\alpha } \lambda_{f,j,k} \chi_{B(x_{j,k},\delta \eps^{j })}(x) }_{\ell^q(\N)}&= C_{4}\norm{\eps^{-j\alpha }   \chi_{B(x_{j,k}, \delta \eps^{j })}(x) \inf_{B(x_{j,k}, \delta \eps^{j })} \Mc_{p_0}(f*\widetilde \psi_j) }_{\ell^q(\N)}\\
			&\meg C_{4}\norm{\eps^{-j\alpha } \Mc_{p_0}(f*\widetilde \psi_j)(x) }_{\ell^q(\N)}
		\end{split}
		\] 
		for every $j\in\N$. Consequently, by means of, e.g.,~\cite{Singular}, we see that there is a constant $C_{5}>0$ such that
		\[
		\norm{\lambda_f}_{f^{p,q}_\alpha((x_{j,k}))}\meg C_{5} \norm{f}_{F^{p,q}_\alpha(G)},
		\] 
		whence the conclusion.
	\end{proof}

	\section{Duality}\label{sec:5}
	
	In this section, functions and distributions in the spaces $B^{p,q}_\alpha(G)$, $\mathring B^{p,q}_\alpha(G)$, etc., are assumed to take values in $Z$, while functions and distributions in the spaces $B^{p',q'}_{-\alpha+(1/p-1)_+Q_*}(G)$, etc., are assumed to take values in $Z'$. Similar conventions hold for the corresponding spaces of sequences.
	In addition, when $Z$ is not a Hilbert space, sesquilinear pairings have to be replaced by bilinear ones. This causes no harm, since we may assume that the operator $\Lc$ has real coefficients, so that $\Kc_\Lc(m)$ is real valued for every real valued $m$. The entire section could in fact have been written directly with bilinear pairings, but some constructions appeared to be slightly  more natural with sesquilinear ones.
	
	\begin{teo} \label{teo:12b}
		Take $p,q\in (0,+\infty]$ and $\alpha\in \R$. 
		There is a unique continuous sesquilinear mapping $B\colon \mathring B^{p,q}_\alpha(G)\times B^{p',q'}_{-\alpha+(1/p-1)_+Q_*}(G)\to \C$ which extends the canonical sesquilinear pairing on $\Sr(G)\times \Sr'(G)$.
		In addition, $B$ induces an antilinear isomorphism of $B^{p',q'}_{-\alpha+(1/p-1)_+Q_*}(G)$ onto the dual of $\mathring B^{p,q}_\alpha(G)$.  
	\end{teo}
	
	\begin{teo}\label{teo:12c}
		Take $p ,q\in (0,+\infty]$  and $\alpha\in \R$. If $p\Meg 1$ and either $p=1$ or $q\Meg 1$, then there is a unique continuous sesquilinear mapping $B\colon F^{p,q}_\alpha(G)\times F^{p',q'}_{-\alpha }(G)\to \C$ which extends the canonical sesquilinear pairing on $\Sr(G)\times \Sr'(G)$. In addition, $B$   induces an antilinear isomorphism of $F^{p',q'}_{-\alpha }(G)$ onto the dual of $\mathring F^{p,q}_\alpha(G)$. 
		
		If  $p<1$, then there is a unique continuous sesquilinear mapping $B\colon F^{p,q}_\alpha(G)\times B^{\infty,\infty}_{-\alpha+(1/p-1)_+Q_*}(G)\to \C$ which extends the canonical sesquilinear pairing on $\Sr(G)\times \Sr'(G)$. In addition, $B$   induces an antilinear isomorphism of $B^{\infty,\infty}_{-\alpha+(1/p-1)_+Q_*}(G)$ onto the dual of $\mathring F^{p,q}_\alpha(G)$.  
	\end{teo}
	
	We collect in the following remark a description of $B$ which will be proved in the proof of Theorems~\ref{teo:12b} and~\ref{teo:12c}.
	
	\begin{oss}\label{oss:7}
		Take $\eps\in (0,1)$ and $(\psi_j),(\psi'_j)\in \Psi_\eps$  such that $(\psi_j*\psi'^*_j)\in \widetilde \Psi_\eps$. Then, with the notation of Theorems~\ref{teo:12b} and~\ref{teo:12c},
		\[
		B(f,g)=\sum_{j\in \N} \langle f*\psi_j\vert g*\psi'_j\rangle.
		\]
	\end{oss}

	Before we pass to the proof of these results, it will be convenient to find the duals of the spaces $\mathring b$ and $\mathring f$.
	
	\begin{prop}\label{prop:15}
		Take $p,q\in (0,+\infty]$, $\alpha\in \R$, $\eps\in (0,1)$, a reduced $(\eps,\delta,R)$-lattice $(x_{j,k})_{(j,k)\in J}$ for some $\delta>0$ and some $R\Meg 2$, and a bounded family $(c_{j,k})\in \R^{J}$. Then, the sesquilinear form
		\[
		B \colon   (a,b)\mapsto \sum_{(j,k)\in J}  \ee^{c_{j,k}} \eps^{jQ_*} a_{j,k}\overline{b_{j,k}} 
		\]
		induces well defined and continuous mappings 
		\[
		\begin{aligned}
			 b^{p,q}_\alpha(\eps,J)&\times b^{p',q'}_{-\alpha+(1/p-1)_+Q_*}(\eps,J)&&\to \C\\
			f^{p,q}_\alpha((x_{j,k}))&\times f^{p',q'}_{-\alpha }((x_{j,k}))&&\to \C \qquad \text{if    $p,q \in [1,\infty]$}\\
			f^{p,q}_\alpha((x_{j,k}))&\times b^{\infty,\infty}_{-\alpha+(1/p-1)_+Q_* }(\eps ,J)&&\to \C \qquad \text{if either $p\in (0,1)$ or $(p,q)\in \Set{1}\times (0,1]$.}\\
		\end{aligned}
		\]
		In addition, $B$ induces antilinear isomorphisms
		\[
		\begin{aligned}
			b^{p',q'}_{-\alpha+(1/p-1)_+Q_*}(\eps,J)&\to \mathring b^{p,q}_\alpha(\eps,J)' \qquad  \\
			f^{p',q'}_{-\alpha }((x_{j,k}))&\to \mathring f^{p,q}_\alpha((x_{j,k}))' \qquad &&\text{if  $p,q\in [1,\infty]$}\\
			b^{\infty,\infty}_{-\alpha+(1/p-1)_+Q_* }(\eps,J )&\to \mathring f^{p,q}_\alpha((x_{j,k}))' \qquad &&\text{if either $p\in (0,1)$ or $(p,q)\in \Set{1}\times (0,1]$.}\\
		\end{aligned}
		\] 
	\end{prop}
	
	We observe explicitly that, in the classical case,~\cite[Remarks 5.9 and 5.11]{FrazierJawerth} (essentially) state that $B$ induces an antilinear isomorphism of $f^{p',q'}_{-\alpha }((x_{j,k}))$ onto $  f^{p,q}_\alpha((x_{j,k}))'$ for $p\in[1,\infty)$ and $q\in (0,\infty)$. However, the proof is incorrect, since it applies the Hahn--Banach theorem to extend a continuous linear functional on a (closed) vector subspace of $L^p(\beta;\ell^q(\N))$ to the whole $L^p(\beta;\ell^q(\N))$. However, when $q<1$ the space $L^p(\beta;\ell^q(\N))$ is \emph{not} locally convex, so that the Hahn--Banach theorem cannot be applied. When $p=1$, one may still obtain the desired  duality using Proposition~\ref{prop:12b}, but when $p\in (1,\infty)$ we have not found any alternative proof (nor do we know whether the assertion be correct or not).
	
	Observe, in addition, that we added a coefficient $\ee^{c_{j,k}}$ to simplify the application of this result in the proof of Theorems~\ref{teo:12b} and~\ref{teo:12c}, even though the `natural' choice would be $c_{j,k}=0$ for the spaces $b$, and $c_{j,k}=\log( \eps^{-jQ_*} \beta(B(e,\delta \eps^j)) )$ for the spaces $f$ (since this gives $B(a,b)= \sum_{j,k}\int_{B(x_{j,k},\delta \eps^j)} a_{j,k}\overline {b_{j,k}}\,\dd \beta$).
	
	In order to deal with the case $p=1$, we need the following results.

	\begin{lem}\label{lem:60b}
		Take $p\in (0,\infty)$, $q\in (0,\infty]$, $\eta,\eps\in (0,1)$, $\delta_0>0$ and $R\Meg 2$. Then, there is a constant $C>0$ such that the following hold. Take a reduced $(\eps,\delta,R)$-lattice $(x_{j,k})_{(j,k)\in J}$ on $G$, with $\delta\in (0,\delta_0]$, and define $\iota\colon \C^{J}\to \C^{\N\times G}$ so that $\iota(\lambda) (j,x)\coloneqq \sum_{k} \lambda_{j,k}\chi_{B(x_{j,k}, \delta \eps^j)}(x)$ for every $\lambda\in \C^{J}$ and for every $(j,x)\in \N\times G$. In addition, for every $j\in\N$ take a $\beta$-measurable subset  $E_j$ of $  G$ such that 
		\[
			\beta(E_j\cap B(x_{j,k},\delta\eps^j))\Meg \eta \beta(B(e,\delta\eps^j))  
		\]
		for every  $(j,k)\in J$. Then,  
		\[
		\norm{ \iota(\lambda) }_{L^{q,p}(\N,G)}\meg C \norm{\chi_{E_j}(x) \iota(\lambda)(j,x) }_{ L^{q,p}_{(j,x)}(\N,G)}
		\]
		and
		\[
		\norm{ \iota(\lambda) }_{\Cc^{q}(\eps)}\meg C \norm{\chi_{E_j}(x) \iota(\lambda)(j,x) }_{ L^{q,\infty }_{j,x}(\N, G)}
		\]
		for every $\lambda\in \C^{J}$.
	\end{lem}
	 
		\begin{proof} 
			Set $p_0\coloneqq \min(p/2,q/2,1)$, and observe that there is a constant $C_1>0$ such that
			\[
			\begin{split}
				\abs{\iota(\lambda)(j,x)}&=\left( \dashint_{ B(x_{j,k},\delta\eps^j)\cap E_j}  \abs{\iota(\lambda)(j,y)}^{p_0}\,\dd \beta(y)\right) ^{1/p_0}\\
					&\meg  \eta^{-1/p_0}\left(\dashint_{ B(x_{j,k},\delta\eps^j) }  \chi_{E_j}(y)\abs{\iota(\lambda)(j,y)}^{p_0}\,\dd \beta(y)\right) ^{1/p_0} \\
					& \meg C_1 \eta^{-1/p_0} \left(\dashint_{ B( x,2\delta\eps^j) }  \chi_{E_j}(y)\abs{\iota(\lambda)(j,y)}^{p_0}\,\dd \beta(y)\right) ^{1/p_0}
			\end{split}
			\]
			for  every $j\in\N$ and for every $x\in B(x_{j,k},\delta \eps^j)$ for some $k\in\N$ with $(j,k)\in J$.  Consequently,
			\[
			\abs{\iota(\lambda)(j,x)}\meg C_1 \eta^{-1/p_0} \Mc_{p_0}( \chi_{E_j}  \iota(\lambda)(j,\,\cdot\,))(x)
			\]
			for every $(j,x)\in \N\times G$, so that the first assertion follows from the vector-valued boundedness of $\Mc_{p_0}$ (cf., e.g.,~\cite{Singular}).
			Observe that the same arguments show that
			\[
			\chi_{B(x,\eps^N)}\abs{\iota(\lambda)(j,\,\cdot\,)}\meg C_1 \eta^{-1/p_0} \Mc_{p_0}( \chi_{E_j\cap B(x,(1+2\delta)\eps^N)}  \iota(\lambda)(j,\,\cdot\,))
			\]
			for every $x\in G$, for every $N\in\N$, and for every $j\Meg N$.
			Consequently, there is a constant $C_2>0$ such that
			\[
			\begin{split}
				\norm{\chi_{[N,+\infty)\times B(x,\eps^N)}(j,y)   \iota(\lambda)(j,y) }_{L^q_{(j,y)} (\N\times G)}& \meg C_2\norm{\chi_{[N,+\infty)\times B(x,  (1+2\delta)\eps^N )}(j,y) \chi_{E_j}(y)\iota(\lambda)(j,y) }_{L^q_{(j,y)} (\N\times G)}\\
				&  \meg C_2 \beta(e, (1+2\delta_0) \eps^N)^{1/q}  \norm{ \chi_{E_j}(y)\iota(\lambda)(j,y)}_{L^{q,\infty}_{(j,y)} (\mi, \beta)}
			\end{split}
			\]
			for  every $N\in \N$ and for every $x\in G$. The second assertion follows.
		\end{proof}

	\begin{lem}\label{lem:62}
		Take $q\in (0,\infty]$, $r\in (0,\infty)$, $\eps\in (0,1)$, $\delta_0>0$ and $R\Meg 2$. Then, there is a constant $C>0$ such that
		\[
		\frac{1}{C} \norm{\iota(\lambda)}_{\Cc^q(\eps)}\meg \sup_{(N,x)\in \N\times G} \eps^{-N Q_*/r}\norm{\chi_{[N,+\infty)\times B(x,\eps^N )}\iota(\lambda) }_{L^{q,r} (\N,G)}  \meg C\norm{\iota(\lambda)}_{\Cc^q (\eps)}
		\]
		for every $\lambda\in \C^{J}$ and for every reduced $(\eps,\delta,R)$-lattice $(x_{j,k})_{(j,k)\in J}$ on $G$, with $\delta\in (0,\delta_0]$, and where $\iota(\lambda)(j,x)\coloneqq \sum_{k} \lambda_{j,k} \chi_{B(x_{j,k}, \delta \eps^j)}$.
	\end{lem}
	
	This result extends~\cite[Corollary 5.7]{FrazierJawerth} to our setting.  
	
	\begin{proof}
		Take $C_1>0$ such that $\rho^{Q_*}/C_1\meg \beta(B(e,\rho))\meg C_1 \rho^{Q_*}$ for every $\rho \in (0,\delta_0 R]$.
		Assume first that $r\Meg q$. Then, by H\"older's inequality,
		\[
		\eps^{-N Q_*/q}\norm{\chi_{[N,+\infty)\times B(x,\eps^N)} \iota(\lambda) }_{L^{q,q} (\N,G)}  \meg C_1^{1/q-1/r} \eps^{-N Q_*/r}\norm{\chi_{[N,+\infty)\times B(x,\eps^N)}\iota(\lambda) }_{L^{q,r} (\N,G)}  
		\]
		for every $(N,x)\in \N\times G$, so that the left inequality follows in this case. For what concerns the right inequality, take $E$ as in Lemma~\ref{lem:59b}, so that there is a constant $C_2>0$ (independent of $\lambda$) such that
		\[
		\norm{\iota(\lambda)}_{\Cc^q(\eps)}\Meg C_2 \norm{ \chi_E \iota(\lambda)  }_{L^{q,\infty}(\N,G)}.
		\]
		Observe that, arguing as in the proof of Lemma~\ref{lem:60b}, we may find a constant $C_3>0$ such that
		\[
		\begin{split}
			\norm{\chi_{[N,+\infty)\times B(x,\eps^N)}  \iota(\lambda) }_{L^{q,r} (\N, G)}& \meg C_3\norm{\chi_{[N,+\infty)\times B(x,(1+2\delta)\eps^N)} \chi_E \iota(\lambda) }_{L^{q,r} (\N,G)}\\
				&\meg C_3 \beta(e, (1+2\delta_0)\eps^N)^{1/r}\norm{ \chi_E \iota(\lambda)  }_{L^{q,\infty}(\N,G)}
			\end{split}
		\]
		for every   $(N,x)\in \N\times G$, so that also the right inequality follows in this case.  
		
		Now, assume that $q<r$. Then, the right inequality follows by H\"older's inequality as before. For what concerns the left inequality, set $\Nc(f)\coloneqq \sup_{(t,x)\in (0,1]\times G} t^{- Q_*/r}\norm{\chi_{[\log_\eps t,+\infty)\times B(x,t)}\iota(\lambda)  }_{L^{q,r} (\N,G)}$ and   take $\eta >(2C_1)^{1/r} \Nc(f) $. Then, for  $t\in (0,1]$,
		\[
		\begin{split}
				\beta(\{\,y\in B(x,t)\colon \norm{\chi_{(0,t]}(\eps^j)\iota(\lambda)(j,y)}_{\ell^q_j(\N)}>\eta \, \})&\meg \eta^{-r} \int_{B(x, t)}\norm{\chi_{(0,t]}(\eps^j)\iota(\lambda)(j,y)}_{\ell^q_j(\N)}^r\,\dd \beta\\
				&\meg \eta^{-r} t^{Q_*}  \Nc(\iota(\lambda))^r \\
				&< \frac{1}{2} \beta(B(e,t)),
			\end{split}
		\]
		so that $(m^q_{1/2,1,t} \iota(\lambda))(x)\meg (2C_1)^{1/r} \Nc(\iota(\lambda)) $ for every $x\in G$ and for every $t\in (0,1]$ (cf.~Definition~\ref{def:1}). Consequently, $\norm{m^q_{1/2,1}\iota(\lambda)}_{L^\infty(G)}\meg (2C_1)^{1/r}\Nc(\iota(\lambda)) $, whence the result by Lemmas~\ref{lem:59b} and~\ref{lem:60b}. 
	\end{proof}

	\begin{proof}[Proof of Proposition~\ref{prop:15}.]
		We may reduce to the case $\alpha=0$. In addition, since $b^{p_1,q_1}_0(\eps,J)$ is a retract of $b^{p_1,q_1}_0(\eps)$ for every $p_1,q_1\in (0,\infty]$, with the same retractions and sections (compatible with $B$), we may assume that $J=\N\times \N$ when proving the first assertion.
		
		\textsc{Step I} Take $C_1\Meg 1$ so that $\frac{1}{C_1} \meg \ee^{c_{j,k}}\meg C_1$ for every $j,k\in\N$.  Observe that
		\[
		\begin{split}
			\sum_{j,k\in \N}  \ee^{c_{j,k}}  \eps^{j Q_*} \abs{a_{j,k}\overline{b_{j,k}}}&\meg C_1 \sum_{j,k\in \N} \eps^{jQ_*/p}\abs{a_{j,k}} \eps^{j(Q_*/p'-(1/p-1)_+Q_*)}\abs{b_{j,k}}\\
				&\meg C_1\norm{\eps^{jQ_*/p } a_{j,k}}_{\ell^{\max(1,p),\max(1,q)}_{k,j}(\N,\N)} \norm{\eps^{j(Q_*/p'-(1/p-1)_+Q_*)} b_{j,k} }_{\ell^{p',q'}_{k,j}(\N,\N)}\\
				&\meg C_1\norm{\eps^{jQ_*/p} a_{j,k}}_{\ell^{p,q}_{k,j}(\N,\N)} \norm{\eps^{j(Q_*/p'-(1/p-1)_+Q_*)} b_{j,k} }_{\ell^{p',q'}_{k,j}(\N,\N)}\\
				&=C_1 \norm{a}_{b^{p,q}_0(\eps)}\norm{b}_{b^{p',q'}_{(1/p-1)_+Q_*}(\eps)}
		\end{split}
		\]
		for every $a,b\in \C^{\N\times \N}$,
		whence the definiteness and continuity of $B$. Using the standard duality between (the closure of $\C^{(\N\times \N)}$ in) $\ell^{p,q}(\N,\N)$ and $\ell^{p',q'}(\N,\N)$ as well as the fact that the mapping $a\mapsto (\ee^{c_{j,k}}a_{j,k})$ is an automorphism of $b^{p,q}_{0}(\eps)$, we then conclude that $B$ induces an antilinear isomorphism of $b^{p',q'}_{(1/p-1)_+Q_*}(\eps)$ onto $\mathring b^{p,q}_0(\eps)'$.
		
		\textsc{Step II} Assume that $p\in (1,\infty)$.  
		Observe that, if we take $C_2>1$ so that $\eps^{j Q_*}/C_2\meg \beta(B(e,\delta \eps^j))\meg C_2\eps^{jQ_*}$ for every $j\in\N$, then
		\[
		\begin{split}
		\sum_{(j,k)\in J}  \ee^{c_{j,k}}  \eps^{j Q_*} \abs{a_{j,k}\overline{b_{j,k}}}&\meg C_1 C_2 \sum_{(j,k)\in J} \int_{B(x_{j,k},\delta \eps^j)}  \abs{a_{j,k}}  \abs{b_{j,k}}\,\dd \beta\\
			&\meg C_1 C_2 \norm{a}_{f^{p,\max(1,q)}_0((x_{j,k}))}\norm{b}_{f^{p',q'}_{0}((x_{j,k}))}\\
			&\meg C_1 C_2\norm{a}_{f^{p,q}_0((x_{j,k}))}\norm{b}_{f^{p',q'}_{0}((x_{j,k}))},
		\end{split}
		\]
		whence the definiteness and continuity of $B$. Then, assume that $q\in [1,\infty]$ and take $\lambda\in \mathring f^{p,q}_0((x_{j,k}))'$. Define $b_{j,k}\coloneqq \ee^{-c_{j,k}} \eps^{-jQ_*}\overline{\lambda(e_{j,k})} $ for every $(j,k)\in J$, where $(e_{j,k})_{j',k'}=\delta_{j,j'}\delta_{k,k'}$ for every $(j',k')\in J$. Clearly, 
		\[
		\begin{split}
			\lambda(a)&= \sum_{(j,k)\in J} a_{j,k} \lambda(e_{j,k}) = \sum_{(j,k)\in J} \ee^{c_{j,k}}\eps^{jQ_*}a_{j,k} \overline{b_{j,k}} = B(a,b)
		\end{split}
		\]
		for every $a\in \C^{(J)}$.  
		Now, by~\cite[Theorem 2]{BenedekPanzone},
		\[
		\begin{split}
			\norm{b}_{f^{p',q'}_{0}((x_{j,k}))} 
				&=   \norm*{ \sum_{k }\chi_{B(x_{j,k},\delta \eps^j)}(x) b_{j,k}     }_{L^{q',p'}_{j,x}(\N,G)}\\
				&= \sup_{f\in U} \sum_{j\in \N} \int_G \abs{f(j,x)} \sum_{k }\chi_{B(x_{j,k},\delta \eps^j)}(x) \abs{b_{j,k}}\,\dd \beta(x) \\
				&\meg C_1 C_2   \sup_{f\in U} \sum_{(j,k)\in J}\ee^{c_{j,k}}\eps^{j Q_*} a^{(f)}_{j,k}  \overline{b_{j,k}} \\
				&= C_1 C_2  \sup_{f\in U} \sup_{n\in\N} \lambda(  a^{(f)} \chi_{\Set{0,\dots,n}^2} )\\
				&\meg C_1 C_2   \norm{\lambda}_{\mathring f^{p,q}_0((x_{j,k}))'} \sup_{f\in U} \norm{ a^{(f)} }_{f^{p,q}_0((x_{j,k}))}
		\end{split}
		\]
		where $U$ is the unit ball of   $L^{q,p}(\N,G)$ and $a^{(f)}_{j,k}=\sgn(b_{j,k}) \dashint_{B(x_{j,k},\delta \eps^j)} \abs{f(j,x)}\,\dd \beta(x)$ for every $(j,k)\in J$. Now, observe that there is a constant $C_3>0$ such that
		\[
		\abs{a^{(f)}_{j,k}}\meg C_3 \Nc_{1,2 D,\eps^j}(f(j,\,\cdot\,) )(x)
		\]
		for every $x\in B(x_{j,k},\delta \eps^j)$ and for every $(j,k)\in J$, so that by Corollary~\ref{cor:18} there is a constant $C_4>0$ such that
		\[
		\begin{split}
		\norm{a^{(f)}}_{f^{p,q}_0(G)}&= \norm*{ \sum_{k} \chi_{B(x_{j,k},\delta \eps^j)}(x)  a^{(f)}_{j,k}    }_{L^{q,p}_{j,x}(\N,G)}\\
			&\meg C_3 \norm*{   \Nc_{1,2D,\eps^j}(f(j,\,\cdot\,) )(x) }_{L^{q,p}_{j,x}(\N,G)}\\
			&\meg C_4\norm{f}_{L^{q,p}(\N,G)}\meg C_4
		\end{split}
		\]
		for every $f\in U$.
		The assertion therefore follows in this case.  
		
		 \textsc{Step III} Consider now the case $p\meg 1$. 
		 If $q\in (0,1]$, then Proposition~\ref{prop:12b}  shows that $b^{p,\min(p,q)}_0(\eps)\subseteq f^{p,q}_0((x_{j,k}))\subseteq b^{p,\max(p,q)}_0(\eps)$ continuously, so that the assertion follows from~\textsc{step I}. Analogously, if $p\in (0,1)$ and $q\in (1,\infty]$, then Proposition~\ref{prop:12b} and~\ref{prop:13b} show that $b^{p,\min(p,q)}_0(\eps)\subseteq f^{p,q}_0((x_{j,k}))\subseteq f^{1,1}_{-(1/p-1)Q_*}((x_{j,k}))= b^{1,1}_{-(1/p-1)Q_*}(\eps)$ continuously, so that the assertion follows again from~\textsc{step I}. We are left with the case $p=1$ and $q\in (1,\infty]$.  
		
		  Let us first prove that $B$ is well defined and continuous on $f^{1,q}_0((x_{j,k}))\times f^{\infty,q'}_{0}((x_{j,k}))$. Observe that it will suffice to show that the mapping $(a,b)\mapsto a \overline b$ maps $f^{1,q}_0((x_{j,k}))\times f^{\infty,q'}_{0}((x_{j,k}))$ continuously into $f^{1,1}_0((x_{j,k}))=b^{1,1}_0(\eps)$. To this aim, take $a\in f^{1,q}_0((x_{j,k}))$ and $b\in f^{\infty,q'}_{0}((x_{j,k}))$, and define $\iota(a')(j,x)\coloneqq\sum_{k} \chi_{B(x_{j,k},\delta \eps^j)}(x) a'_{j,k}$ for every $j\in\N$, for every $x\in G$, and for every $a'\in \C^{J}$. Then, by Lemmas~\ref{lem:59b} and~\ref{lem:60b} we may find a Borel subset $E$ of $\N\times G$ and a constant $C_5>0$ such that
		 \[
		 \norm{a \overline b}_{f^{1,1}_0((x_{j,k}))}=\norm{\iota(a) \iota(\overline b)}_{L^{1,1}(\N,G)}\meg C_5 \norm{\chi_E \iota(a) \iota(\overline b)}_{L^{1,1}(\N,G)}
		 \]
		 and such that
		 \[
		 \norm{\chi_E \iota(\overline b)}_{L^{q',\infty}(\N,G)}\meg C_5 \norm{\iota(\overline b)}_{\Cc^{q'}(\eps)}= C_5\norm{b}_{f^{q',\infty}_{0}((x_{j,k}))}.
		 \]
		 Since clearly $   \norm{\iota(a)}_{L^{q,1}(\N,G)}= \norm{a}_{f^{q,1}_{0}((x_{j,k}))}$, the definiteness and continuity of $B$ follows.
		 Now, take $\lambda\in \mathring f^{1,q}_0((x_{j,k}))'$ and observe that, if we define $b_{j,k}\coloneqq \ee^{-c_{j,k}}\eps^{-jQ_*} \overline{\lambda(e_{j,k})}$, then $\lambda(a)= B(a,b)$ for every $a\in \C^{(J)}$, as in~\textsc{step II}. 
		 We then have to prove that $b\in f^{\infty,q'}_{0}((x_{j,k}))$. Observe that we may reduce to the case in which $\ee^{c_{j,k}}=\eps^{-jQ_*}\beta(B(e,\delta \eps^j))$ for every $(j,k)\in J$, so that
		 \[
		 \lambda(a)=B(a,b)=\sum_{j\in \N} \int_G \iota(a)(j,x)\iota(\overline b)(j,x)\,\dd \beta(x)
		 \]
		 for every $a\in \C^{(J)}$.  
		 Define $a \in \C^{J}$ so that $a_{j,k}= b_{j,k}\abs{b_{j,k}}^{q'-2}$ for every $(j,k)\in J$.
		 Given $x\in G$ and $n,N\in\N$, let $J_{n,N,x}$ be the set of $(j,k)\in J_n\coloneqq \Set{0,\dots, n}^2$ such that $j\Meg N$ and $B(x_{j,k},\delta \eps^j)\cap B(x,\eps^N)\neq \emptyset$.
		 If $q<\infty$, then Lemma~\ref{lem:62} (essentially, cf.~\cite[Lemma 3.2]{Calzi3}) shows that there is a constant $C_6>0$ such that
		 \[
		 \begin{split}
		 \norm{\chi_{J_{n,N,x}} a}_{f^{1,q}_0((x_{j,k}))}&= \norm{ \iota(\chi_{J_{n,N,x}}b)}_{L^{q',q'/q}(\N,G)}^{q'/q}\\
		 	&=\norm{\chi_{[N,+\infty)\times B(x,(1+2\delta)\eps^N)} \iota(\chi_{J_{n,N,x}}b)}_{L^{q',q'/q}(\N,G)}^{q'/q}\\
		 	&\meg  C_6\eps^{N Q_*  }\norm{\iota(\chi_{J_{n,N,x}}b)}_{\Cc^{q'}(\eps )}^{q'/q}\\
		 	&\meg C_6\eps^{N Q_* }\norm{\iota(\chi_{J_n}b)}_{\Cc^{q'}(\eps)}^{q'/q},
		 \end{split}
		 \]
		 and
		 \[
		 \begin{split}
		 \lambda(\chi_{J_{n,N,x}}a)= \norm{ \iota(\chi_{n,N,x} b)}_{L^{q'}(\N\times G)}^{ q'},
		 \end{split}
		 \]
		 so that
		 \[
		 \eps^{-N Q_* }\norm{ \iota(\chi_{J_{n,N,x}}b)}_{L^{q'}(\N\times G)}^{ q'}\meg C_6\norm{\lambda}_{\mathring f^{1,q}_0((x_{j,k}))'} \norm{\iota(\chi_{J_{n}}b)}_{\Cc^{q'}(\eps )}^{q'/q}.
		 \]
		 By the arbitrariness of $x$ and $N$, we then infer that
		 \[
		 \norm{\iota(\chi_{J_n} b)}_{\Cc^{q'}(\eps)}\meg C_6^{1/q'}\norm{\lambda}_{\mathring f^{1,q}_0((x_{j,k}))'}^{1/q'} \norm{\iota(\chi_{J_{n}}b)}_{\Cc^{q'}(\eps)}^{1/q},
		 \]
		 so that  
		 \[
		 \norm{\iota(\chi_{J_{n}}b)}_{\Cc^{q'}(\eps)}\meg C_6\norm{\lambda}_{\mathring f^{1,q}_0((x_{j,k}))'},
		 \]
		 whence our assertion by the arbitrariness of $n$.  
		 
		 If, otherwise, $q=\infty$, then there is a constant $C_7>0$ such that
		 \[
		 \begin{split}
		 	\norm{\chi_{J_{n,N,x}} a}_{f^{1,\infty}_0((x_{j,k}))}&\meg \beta(B(e,(1+2\delta)\eps^N))\meg C_7 \eps^{NQ_*}
		 \end{split}
		 \]
		 and
		 \[
		 \lambda(\chi_{J_{n,N,x}}a)= \norm{  \iota(\chi_{J_{n,N,x}}b)}_{L^{1}(\N\times G)} ,
		 \]
		 so that arguing as before we find
		 \[
		 \norm{b}_{f^{\infty,1}_0((x_{j,k}))}\meg C_7 \norm{\lambda}_{\mathring f^{1,\infty}_0((x_{j,k}))}
		 \] 
		 
		 \textsc{Step IV} Finally, we consider the case $p=\infty$ and $q\in [1,\infty]$. We could proceed as in~\textsc{step III}, but we prefer to get this case as a consequence of the case $p=1$. Let $U$ be the closed unit ball in $f^{1,q'}_0((x_{j,k}))$ and let $V$  be the closed unit ball in $f^{\infty,q}_0((x_{j,k}))$. 
		 We shall tacitly use the fact that $U$ and $V$ are closed and convex in $\C^{J}$ and that both sets are invariant under multiplication by elements of the closed unit ball of $\ell^\infty(J)$. 
		 Observe first that the closures of $\C^{(J)}\cap U$ and $\C^{(J)}\cap V$  in $\C^{J}$  are $U$ and $V$, respectively, while their closures in $\mathring f^{1,q'}_0((x_{j,k}))$ and $\mathring f^{\infty,q}_0((x_{j,k}))$ are $U\cap \mathring f^{1,q'}_0((x_{j,k}))$ and $V\cap \mathring f^{\infty,q}_0((x_{j,k}))$, respectively. 
		 As a consequence, by~\textsc{step III} there is a constant $C_8>1$ such that $C_8^{-1} V\subseteq (\C^{(J)}\cap U)^\circ\subseteq C_8 V$, where $(\C^{(J)}\cap U)^\circ$ denotes the polar of $\C^{(J)}\cap U$ with respect to the sesquilinear form $B'$ induced by $B$ on $\C^{(J)}\times \C^{J}$. By the bipolar theorem, this shows that $C_8^{-1} V^\circ\subseteq \C^{(J)}\cap U\subseteq C_8 V^\circ$, where $V^\circ$ denotes the polar of $V$ with respect to $B'$.  Taking closures in $\C^{J}$, we then see that $C_8^{-1} (V\cap \C^{(J)})^\circ\subseteq   U\subseteq C_8 (V\cap \C^{(J)})^\circ$, where $(V\cap \C^{(J)})^\circ$ denotes the polar of $V\cap \C^{(J)}$ with respect to $B'$. This leads to the conclusion.\footnote{It may appear that, in the vector-valued case, this argument require that $Z$ be reflexive. Nonetheless, it does not, since it suffices to observe that every Banach space embeds naturally as a closed subspace of its bidual.}
	\end{proof}
	
	We may now complete the proof of Lemma~\ref{lem:68}. We shall then be able to apply Theorems~\ref{teo:13bis} and~\ref{teo:13} in order to prove Theorems~\ref{teo:12b} and~\ref{teo:12c}. This is possible since the only part of Theorems~\ref{teo:13bis} and~\ref{teo:13} which we have not proved yet\footnote{Actually, the case $p=\infty$ of Theorem~\ref{teo:13}  has so far a `formal' proof, since it is based on the corresponding case of Lemma~\ref{lem:68}, which will be proved below.} concerns  the convergence of $\sum_{j,k} \lambda_{j,k} a_{j,k}$ in $\Sr'(G)$ when $\lambda \not \in \mathring b^{p,q}_\alpha(\eps)$ or $\lambda \not \in \mathring f^{p,q}_\alpha((x_{j,k}))$, which is not needed in the proof of Theorems~\ref{teo:12b} and~\ref{teo:12c}. We shall, in fact, prove this latter assertion with the aid of Theorems~\ref{teo:12b} and~\ref{teo:12c}.
	
	\begin{proof}[Second part of the proof of Lemma~\ref{lem:68}.]
		\textsc{Step III} We now deal with the case $p=\infty$, $q>1$, and $\alpha=0$. Because of Proposition~\ref{prop:15}, we may proceed by duality, so that the assertion follows from the case $p=1$.
	\end{proof}
	
	\begin{proof}[Proof of Theorems~\ref{teo:12b} and~\ref{teo:12c}.]
		As mentioned before, we shall use the part of Theorem~\ref{teo:13} which we have already proved.
		
		Define $X$ and $Y$ so that either one of the following conditions hold:
		\begin{itemize}
			\item $X=  B^{p,q}_\alpha(G)$ and $Y=B^{p',q'}_{-\alpha+(1/p-1)_+Q_*}(G)$;
			
			\item $p,q\in [1,\infty]$, $X= F^{p,q}_\alpha(G)$, and $Y=F^{p',q'}_{-\alpha }(G)$;
			
			\item $p\in (0,1)$ or $(p,q)\in \Set{1}\times (0,1]$, $X= F^{p,q}_\alpha(G)$, and $Y=B^{\infty,\infty}_{-\alpha+(1/p-1)_+Q_* }(G)$.
		\end{itemize}
		Define $\mathring X$ as the closure of $\Sr(G)$ in $X$, and   $S_{X}, \mathring S_{X}$, and $S_Y$ as the corresponding spaces of sequences.
		Set $\eps'\coloneqq \eps^{1/\grado}$ and take a reduced $(\eps',1,2)$-lattice $(x_{j,k})$ on $G$, and $(\psi_j),(\psi'_{j'})$ as in the statement. Then, for every $f\in \Sr(G)$ and for every $g\in Y$, 
		\[
		\langle f\vert g \rangle= \sum_{j\in \N} \langle f*\psi_{j}*\psi'^*_j\vert g \rangle= \sum_{j\in\N} \langle f*\psi_{j}\vert g*\psi'_{j}\rangle
		\] 
		by Lemma~\ref{lem:41}, so that $B$ induces the canonical sesquilinear pairing on $\Sr(G)\times Y$. In addition, setting $a_{j,k}\coloneqq (\max_{B(x_{j,k},2\eps'^{j } )} \abs{f*\psi_{j}} )_{j,k}$ and $b_{j,k}\coloneqq (\max_{B(x_{j,k},2\eps'^{j } )} \abs{g*\psi'_{j}} )_{j,k}$, by means of  Proposition~\ref{prop:14}  we see that  there is a constant $C_1>0$ such that
		\[
		\norm{a}_{S_{X}}\meg C_1\norm{f}_{X}
		\]
		and
		\[
		\norm{b}_{S_Y}\meg C_1 \norm{g}_Y.
		\] 
		In addition, by Proposition~\ref{prop:15} there is a constant $C_2>0$ such that
		\[
		\sum_{j\in\N}\abs*{ \langle f*\psi_{j}\vert g*\psi'_{j}\rangle}\meg \sum_{j,k} \beta(B(x_{j,k},2\eps'^{j } )) a_{j,k} b_{j,k}\meg C_2 \norm{a}_{S_{X}}\norm{b}_{S_Y}.
		\]
		This proves the existence and continuity of $B$  on $\mathring X\times Y$.

		Now, take $ \lambda\in X'$, and take $g\in \Sr'(G)$ so that $\lambda(f)= \langle f \vert g \rangle$ for every $f\in \Sr(G)$. Observe that, by Corollary~\ref{cor:12}, there is a constant $c>0$ such that $(c \eps'^{j Q_*} \psi_{j}^*(z_{j,k}^{-1}x_{j,k}^{-1}\,\cdot\,))_{j,k}$ is a system of $(K,S,N,\infty,\Lc)$-atoms for every $z_{j,k}\in \overline B(e,2\eps^j)$ ($(j,k)\in J$), where $K,S,N\in\N$ are chosen so that $K>\alpha$, $S>Q_*/\min(1,p,q)-\alpha$, and $N>D/\min(1,p,q)$. Then, Theorems~\ref{teo:13bis} and~\ref{teo:13} show that the mapping
		\[
		\C^{(J)}\ni a' \mapsto \sum_{j,k} \eps'^{jQ_*}a'_{j,k} \overline{(\overline g*\psi_{j})(x_{j,k}z_{j,k})}= \lambda \left(\sum_{j,k} a'_{j,k}  \eps'^{jQ_*}\psi_{j}^*(z_{j,k}^{-1}x_{j,k}^{-1}\,\cdot\,) \right)\in \C
		\]
		is continuous for the topology of $S_{X}$. By Proposition~\ref{prop:15}, this implies that $(\overline g*\psi_j)(x_{j,k}z_{j,k})\in S_Y $. More precisely, there is a constant $C_3>0$ such that
		\[
		\norm{((\overline g*\psi_j)(x_{j,k}z_{j,k}))_{j,k}}_{S_Y}\meg C_3 \norm{\lambda}_{X'}.
		\]
		By Proposition~\ref{prop:14}  we then see that $\overline g\in Y$ and that there is a constant $C_4>0$ such that $\norm{\overline g}_Y\meg C_4 \norm{\lambda}_{X'}$.  The conclusion follows from the fact that conjugation induces an automorphism of $Y$, as one sees taking $\Lc$ with real coefficients, so that the $\psi_j$ are real-valued (if the corresponding $\phi_j$ are real valued). 
	\end{proof}

\begin{deff}
	Take $p,q\in [1,\infty]$ and $\alpha\in \R$. Then, we shall denote with $\langle\,\cdot\,\vert \,\cdot\,\rangle$ the   extension of the canonical sesquilinear pairing on $\Sr(G)\times \Sr'(G)$ to $\mathring B^{p,q}_\alpha(G)\times B^{p',q'}_{-\alpha}(G)$ and $\mathring F^{p,q}_\alpha(G)\times F^{p',q'}_{-\alpha}(G)$, and also the  extension of the canonical sesquilinear pairing on $\Sr'(G)\times \Sr(G)$ to $ B^{p,q}_\alpha(G)\times \mathring B^{p',q'}_{-\alpha}(G)$ and $  F^{p,q}_\alpha(G)\times  \mathring F ^{p',q'}_{-\alpha}(G)$.
\end{deff}

\begin{oss}
	Using Propositions~\ref{prop:12} and~\ref{prop:13}, one may see that $\langle\,\cdot\,\vert \,\cdot\,\rangle$ induces continuous sesquilinear forms on 
	\[
	\begin{aligned}
		\mathring B^{p,q}_\alpha(G)&\times B^{p',q'}_{-\alpha+(1/p-1)_+Q_*}(G) &  B^{p,q}_\alpha(G)&\times \mathring B^{p',q'}_{-\alpha+(1/p-1)_+Q_*}(G)\\
		\mathring F^{p,q}_\alpha(G)&\times F^{p',q'}_{-\alpha }(G) &  F^{p,q}_\alpha(G)&\times \mathring F^{p',q'}_{-\alpha }(G) & &(p\Meg 1)\\
		\mathring F^{p,q}_\alpha(G)&\times B^{\infty,\infty}_{-\alpha+(1/p-1) Q_*}(G) &  F^{p,q}_\alpha(G)&\times \mathring B^{\infty,\infty}_{-\alpha+(1/p-1) Q_*}(G) & &(p<1).
	\end{aligned}
	\]
	As a matter of fact, one may actually define a continuous sesquilinear form on $B^{p,q}_\alpha(G)\times B^{p',q'}_{-\alpha+(1/p-1)_+Q_*}(G)$ and $F^{p,q}_\alpha(G)\times F^{p',q'}_{-\alpha+(1/p-1)_+Q_* }(G)$ using the description of $\langle\,\cdot\,\vert\,\cdot\,\rangle$ provided in Remark~\ref{oss:8} as a definition. Nonetheless, proving that this definition does \emph{not} depend on the choice of $\eps$, $\Lc$, and $(\psi_j)$, $(\psi'_j)$ is considerably involved, as well as proving `elementary' statements such as $\langle f*\psi\vert g\rangle= \langle f\vert g*\psi^*\rangle$ for suitable $\psi\in \Sr(G)$. Since the only cases that cannot be dealt with using the above procedure essentially reduce to $ B^{p,\infty}_\alpha(G) \times B^{p',1}_{-\alpha+(1/p-1) Q_*}(G)$, for $p\meg 1$, and $F^{1,\infty}_\alpha(G) \times F^{\infty,1}_{-\alpha }(G) $ and since we shall not need the more general sesquilinear form, we shall not pursue this investigation.
\end{oss}

\begin{proof}[Second part of the proof of Theorem~\ref{teo:13}.]
	It remains to show that the sum $\sum_{j,k} \lambda_{j,k} a_{j,k}$ converges in $\Sr'(G)$. Observe that, by Proposition~\ref{prop:13b}, we may reduce to the case $p,q\Meg 1$.
	Take $(\psi_j)\in \widetilde \Psi_{\eps^{\grado}}$ and take $(\psi'_j)\in \Psi_{\eps^{\grado}}$ so that $\psi_j=\psi_j*\psi'_j$ for every $j\in\N$. We may also assume that $\psi'_j=\psi_j'^*$ for every $j\in\N$. Take a bounded subset $B$ of  $\Sr(G)$ and observe that
	\[
	\lambda_{j,k}\langle a_{j,k}\vert \phi\rangle=\sum_{j'} \lambda_{j,k} \langle a_{j,k}\vert \phi*\psi_{j'}\rangle =\sum_{j'} \lambda_{j,k} \langle a_{j,k}*\psi'_{j'}\vert \phi*\psi_{j'}\rangle
	\]
	for every $j,k\in\N$ and for every $\phi\in B$, thanks to Lemma~\ref{lem:41}. Observe that, by Proposition~\ref{prop:14} and Lemma~\ref{lem:67}, there is a constant $C_1>0$ such that
	\[
	\norm{ \Sc_+ \phi  }_{f^{p',q'}_{-\alpha}((x_{j,k}))}\meg C_1\norm{\phi}_{F^{p',q'}_{-\alpha}(G)}
	\]
	for every $\phi\in \Sr(G)$
	and
	\[
	\abs{(a_{j,k}*\psi'_{j'})(x)}\meg C_1 \eps^{(j-j')_+( S\grado+Q_*/p_0)+ (j'-j)_+ (K\grado-Q_*/p_0)}(1+d(x,x_{j,k})/\eps^{\min(j,j')})^{-N}
	\]
	for every $j,j',k\in\N$,
	where $(\Sc_+ \phi)_{j,k}=\max_{\overline B(x_{j,k},\delta R \eps^j)} \abs{\phi*\psi_j}$ for every $j,k\in\N$.
	Therefore,
	\[
	\abs{ \langle a_{j,k}\vert \phi\rangle }\meg  C_1 C_2 (\delta R)^{Q_*}[2\max(1,\delta R)]^N  \sum_{j',k'}  (\Sc_+ \phi)_{j',k'} \frac{\eps^{(j-j')_+( S\grado+Q_*/p_0)+ (j'-j)_* (K\grado-Q_*/p_0)+j'Q_*}}{(1+d(x_{j',k'},x_{j,k})/\eps^{\min(j,j')})^{N}}
	\]
	for every $(j,k)\in J$ and for every $\phi\in \Sr(G)$, where $C_2>1$ is such that $C_2^{-1} r^{Q_*}\meg \beta(B(e,r))\meg C_2 r^{Q_*}$ for every $r\in (0,\delta_0 R]$.  
	Consequently, by Lemma~\ref{lem:68} there is a constant $C_3>0$ such that
	\[
	\norm*{ (\eps^{-jQ_*}\langle a_{j,k}\vert \phi\rangle  )_{j,k}}_{f^{p',q'}_{-\alpha}((x_{j,k}))}\meg C_3\norm{ \Sc_+ \phi  }_{f^{p',q'}_{-\alpha}((x_{j,k}))}\meg C_1C_3\norm{\phi}_{F^{p',q'}_{-\alpha}(G)},
	\]
	which is bounded for $\phi\in B$. By means of Proposition~\ref{prop:15} we then see that the sum $\sum_{j,k} \lambda_{j,k} \langle a_{j,k}\vert \phi\rangle$ converges, uniformly as $\phi$ runs through $B$, whence the conclusion.
\end{proof}

The following results provide a weaker analogue of Proposition~\ref{prop:14}. One the one hand, much greater flexibility is allowed on how one may compute `samples.' On the other hand, we only get one inequality, so that in general the mapping $f \mapsto (\eps^{-jQ_*}\langle f\vert a_{j,k}\rangle)_{j,k} $ is not an isomorphism onto its image.   

\begin{prop}\label{prop:40bis}
	Take $p,q\in (0,\infty]$ and $ \alpha\in \R$, and take $K,S\in \N$, $N>0$, and $p_0\in [1,\infty]$ so that $S>( Q_*/p_0'+\alpha)/\grado$, $K>(Q_*(1/\min(1,p)-1/p_0')  -\alpha)/\grado$, and $N>D/\min(1,p)$.
	Take $\eps\in (0,1)$, $\delta>0$, and $R\Meg 2$. Then, there is a constant $C>0$ such that, for every system of $(K,S,N,p_0,\Lc)$-atoms $(a_{j,k})$ associated with some reduced $(\eps,\delta,R)$-lattice $(x_{j,k})_{(j,k)\in J}$,
	\[
	\norm{ (\eps^{-jQ_*}\langle f\vert a_{j,k}\rangle)_{j,k}  }_{b^{p,q}_\alpha(\eps,J)}\meg C \norm{f}_{B^{p,q}_\alpha(G)}
	\]
	for every $f\in B^{p,q}_\alpha(G)$.  
\end{prop}

\begin{prop}\label{prop:40}
	Take $p,q\in (0,\infty]$ and $ \alpha\in \R$, and take $K,S\in \N$, $N>0$, and $p_0\in [1,\infty]$ so that $S>( Q_*/p_0'+\alpha)/\grado$, $K>(Q_*(1/\min(1,p,q)-1/p_0')  -\alpha)/\grado$, and $N>D/\min(1,p,q)$.
	Take $\eps\in (0,1)$, $\delta>0$, and $R\Meg 2$. Then, there is a constant $C>0$ such that, for every system of $(K,S,N,p_0,\Lc)$-atoms $(a_{j,k})$ associated with some reduced $(\eps,\delta,R)$-lattice $(x_{j,k})_{(j,k)\in J}$,
	\[
	\norm{ (\eps^{-jQ_*}\langle f\vert a_{j,k}\rangle)_{j,k}  }_{f^{p,q}_\alpha((x_{j,k}))}\meg C \norm{f}_{F^{p,q}_\alpha(G)}
	\]
	for every $f\in F^{p,q}_\alpha(G)$. 
\end{prop}

\begin{proof}[Proof of Proposition~\ref{prop:40}.]
	Observe first that, by Theorem~\ref{teo:13}, $a_{j,k}\in \mathring F^{p',q'}_{-\alpha+(1/p-1)_+Q_*}(G)$ (which is contained in $\mathring B^{\infty,\infty}_{-\alpha+(1/p-1)Q_*}(G)$ when $p<1$ by Proposition~\ref{prop:12}) for every $(j,k)\in J$.
	Take $(\psi_j)\in \widetilde\Psi_{\eps^\grado}$ and take $(\psi'_j)\in \Psi_{\eps^\grado}$ so that $\psi_j=\psi_j*\psi'_j$ for every $j\in\N$, and observe that
	\[
	\langle f\vert a_{j,k}\rangle=  \sum_{j'\in \N} \langle f*\psi_{j'} \vert a_{j,k}*\psi'^*_{j'}\rangle
	\]
	for every $f\in F^{p,q}_\alpha(G)$ and for every $(j,k)\in J$, thanks to Remark~\ref{oss:7}. 
	Then, observe that by Lemma~\ref{lem:67} there is a constant $C_1>0$ such that
	\[
	\begin{split}
		\eps^{-jQ_*}\abs{\langle f\vert a_{j,k}\rangle}&\meg \eps^{-jQ_*}\sum_{j',k' } \int_G \max_{\overline B(x_{j',k'},\delta R \eps^j)} \abs{f*\psi_{j'}}\chi_{ B(x_{j',k'},\delta R\eps^{j'})} \abs{ a_{j,k}*\psi'^*_{j'} }\,\dd \beta\\
		&\meg C_1\sum_{j',k' }  \max_{\overline B(x_{j',k'},\delta R \eps^j)} \abs{f*\psi_{j'}}  \frac{\eps^{(j'-j)Q_*}  \eps^{(j-j')_+(S \grado+Q_*/p_0)+(j'-j)_+(K\grado-Q_*/p_0) }}{ (1+d(x_{j',k'},x_{j,k})/\eps^{\min(j,j')})^{N}}\\
		&= C_1\sum_{j',k'}  \max_{\overline B(x_{j',k'},\delta R \eps^j)} \abs{f*\psi_{j'}}  \frac{  \eps^{(j-j')_+(S \grado-Q_*/p'_0)+(j'-j)_+(K\grado+Q_*/p'_0) }}{ (1+d(x_{j',k'},x_{j,k})/\eps^{\min(j,j')})^{N}}.
	\end{split}
	\]
	The assertion then follows from Lemmas~\ref{lem:68bis} and~\ref{lem:68} combined with Proposition~\ref{prop:14}.
\end{proof}

\end{document}